\documentclass[reqno]{amsart}

\pdfoutput=1
  \pdfpageattr{/Group <</S /Transparency /I true /CS /DeviceRGB>>}

\makeatletter
\def\ps@pprintTitle{%
 \let\@oddhead\@empty
 \let\@evenhead\@empty
 \def\@oddfoot{}%
 \let\@evenfoot\@oddfoot}
\makeatother

\makeatletter
\def\input@path{{./Images/}{./}}
\makeatother

\usepackage{hyperref}
\usepackage{amsfonts,amsmath,amssymb,amsthm}
\usepackage{graphicx}
\usepackage{pgfplots}
\usepackage[textsize=scriptsize]{todonotes}
\usepackage{bm}
\usepackage{lscape}
\usepackage{blkarray}
\usepackage{enumitem}
\usepackage{caption}
\usepackage[labelfont=rm]{subcaption}

\usepackage{tikz}
\usetikzlibrary{matrix}

\newtheorem{theorem}{Theorem}
\newtheorem{lemma}[theorem]{Lemma}
\newtheorem{proposition}[theorem]{Proposition}
\newtheorem{corollary}[theorem]{Corollary}
\newtheorem*{theoremnonum}{Theorem}
\theoremstyle{definition}
\newtheorem{definition}[theorem]{Definition}

\theoremstyle{remark}
\newtheorem{remark}[theorem]{Remark}

\newtheorem{question}[theorem]{Question}
\newtheorem{example}[theorem]{Example}

\numberwithin{theorem}{section} 

\newcommand\TTmin{\mathbb{T}_{\min}}
\newcommand\TTmax{\mathbb{T}_{\max}}
\newcommand\TTclosed{\overline{\mathbb{T}}}
\newcommand\TP{\mathbb{TP}}

\newcommand\TPmax{\mathbb{TP}_{\max}}
\newcommand\puiseux[2]{{#1}\{\!\{{#2}\}\!\}}
\newcommand\SetOf[2]{\left\{\left.#1\vphantom{#2}\ \right|\ #2\vphantom{#1}\right\}}
\newcommand\IdealOf[2]{\left\langle\left.#1\vphantom{#2}\ \right|\ #2\vphantom{#1}\right\rangle}
\newcommand\tcone[1]{\operatorname{tcone}(#1)}
\newcommand\tconv[1]{\operatorname{tconv}(#1)}
\newcommand\clos[1]{\operatorname{cl}(#1)}

\newcommand\monomial[1]{{\sf M}(#1)}
\newcommand\imonomial[2]{{\sf M}_{#1}(#2)}
\newcommand\closedmonomial[1]{\overline{\sf M}(#1)}
\newcommand\complementarymonomial[1]{\rotatebox[origin=c]{180}{\sf M}(#1)}
\newcommand\closedcomplementarymonomial[1]{\overline{\rotatebox[origin=c]{180}{\sf M}}(#1)}
\newcommand\monocone[2]{{\sf MC}_{#1}(#2)}
\DeclareMathOperator{\supp}{supp}
\newcommand\KK{{\mathbb K}}
\DeclareMathOperator\val{val}
\DeclareMathOperator\tbary{tbary}
\newcommand\RR{{\mathbb R}}
\newcommand\ZZ{{\mathbb Z}}
\newcommand{\NN}{\mathbb N}
\newcommand{\1}{\bf 1}

\DeclareMathOperator{\lcm}{lcm}
\newcommand\maxunit{\mathcal{E}_{\max}}
\newcommand\minunit{\mathcal{E}_{\min}}
\newcommand\homogmaxunit{\widehat{\mathcal{E}}_{\max}}
\newcommand\modmaxunit{\overline{\mathcal{E}}_{\max}}
\newcommand\modminunit{\overline{\mathcal{E}}_{\min}}
\newcommand\neighbour{\mathcal{N}}

\DeclareMathOperator{\conv}{conv}
\DeclareMathOperator{\cone}{cone}
\DeclareMathOperator{\hull}{hull}
\DeclareMathOperator{\im}{im}
\DeclareMathOperator{\sign}{sign}

\newcommand\cB{{\mathcal B}}
\newcommand\cC{{\mathcal C}}
\newcommand\cF{{\mathcal F}}

\newcommand\cL{{\mathcal L}}
\newcommand\cM{{\mathcal M}}
\newcommand\cP{{\mathcal P}}
\newcommand\cV{{\mathcal V}}
\newcommand\trans[1]{#1^{\top}}
\newcommand\added[1]{#1}

\title[Face posets of tropical polyhedra and monomial ideals]{Face posets of tropical polyhedra \\ and monomial ideals}
\author{Georg Loho}
\email{g.loho@utwente.nl}

\author{Ben Smith}
\email{benjamin.smith-3@manchester.ac.uk}
\thanks{Georg Loho was supported by the ERC Starting Grant ScaleOpt and by the Swiss National Science Foundation (SNSF) within the project \emph{Convexity, geometry of numbers, and the complexity of integer programming (Nr.~163071)}. Ben Smith was supported by the EPSRC grant \emph{Arrangements of tropical linear spaces (1673882)}.}



\begin{document}
\maketitle

\begin{abstract}
  We exhibit several posets arising from commutative algebra, order theory, tropical convexity as potential face posets of tropical polyhedra, and we clarify their inclusion relations.
  We focus on monomial tropical polyhedra, and deduce how their geometry reflects properties of monomial ideals.
  Their vertex-facet lattice is homotopy equivalent to a sphere and encodes the Betti numbers of an associated monomial ideal. 
\end{abstract}

\section{Introduction}

A union of shifted copies of the positive orthant is a seemingly simple but fundamental object in mathematics.
We call such an object a \emph{monomial tropical polyhedron}.
It occurs in the study of monomial ideals in commutative algebra~\cite{MillerSturmfels:2005}, in multicriteria and vector optimisation~\cite{Ehrgott:2005,Jahn:2011}, in order theory~\cite{Trotter:1992} and tropical convexity~\cite{JoswigLoho:2017}. 
While the starting point of our investigation is the search for the concept of \emph{faces} of \emph{tropical polyhedra}, arising from convexity over the $(\max,+)$-semiring, we do not focus on the geometric viewpoint but rather on the combinatorial side of a face poset.
Our work demonstrates that the search for the `right' notion of faces for a tropical polytope is actually a far deeper question  that branches out into commutative algebra and order theory.

\smallskip

In the development of the theory of tropical polyhedra, it turned out that the classical approaches to the definition of a \emph{face} are all flawed, see~\cite{Joswig:2005,DevelinYu:2007,GaubertKatz:2011,AllamigeonKatz:2013,AllamigeonBenchimolGaubertJoswig:2015,AllamigeonKatz:2017}; they do not properly tile the boundary of the polyhedra, they do not fulfil the desirable characterization of a face in terms of its defining vertices or do not tie in with the extremality property arising from linear programming. 
Our work emerges from the introduction of the \emph{vertex-facet lattice}, defined in Section \ref{sec:vertex+facet+lattice}.
This is a new \emph{face lattice} for monomial tropical polyhedra, building on work from Joswig~\cite{Joswig:2005} and from Develin and Yu~\cite{DevelinYu:2007}.
We define it as the intersection lattice of the vertices contained in the facets, which are well-defined for monomial tropical polyhedra.
As the work ~\cite{DevelinYu:2007} exhibited and we also demonstrate later, it is subtle to put this notion of face in correspondence with parts of the boundary of the tropical polyhedron, therefore we view it as a purely combinatorial object.
The restriction to monomial tropical polyhedra is justified because they form the building blocks for general tropical polyhedra.

\begin{theoremnonum}[{Prop.~\ref{prop:monomial+cone+decomp}}]
  A $d$-dimensional tropical polyhedron is the intersection of $d+1$ monomial tropical polyhedra, one for each possible affine tropical direction. 
\end{theoremnonum}

We establish a common framework based on \emph{covector graphs} to compare several posets partly originating in commutative algebra or order theory, which serve some purpose of a face poset of a monomial tropical polyhedron.
We associate the Scarf poset, CP-order, max-min poset, vertex-facet lattice, max-lattice and pseudovertex poset to a monomial tropical polyhedron.

\begin{theoremnonum}[{Synopsis of Section~\ref{sec:face+posets}}]
The six posets from left to right are embedded in each other, where the embedding of the Scarf poset has the additional property of being cover preserving.
Furthermore, if the monomial tropical polyhedron is sufficiently generic, the first four posets are isomorphic. 
\end{theoremnonum}

Those six, along with two further occurring posets, are visualised in Figure~\ref{fig:several+face+posets}. 
The strictness of the inclusions is deduced through the construction of separating examples.
Each of these posets is a natural candidate as a face poset of a monomial tropical polyhedron that emphasises different properties.
The pseudovertex poset is a highly refined tropical object, derived from the covector decomposition of tropical polyhedra introduced in~\cite{DevelinSturmfels:2004}, that records all possible candidates for faces.
The max-min poset is a far simpler poset that restricts to well-behaved faces that overcome some of the discrepancy displayed in \cite{DevelinYu:2007}, as well as exhibiting the natural duality of monomial tropical polyhedra.
The max-lattice is the natural generalisation of the LCM-lattice, an object from the study of monomial ideals that preserves many homological properties of the monomial ideal \cite{GasharovPeevaWelker:1999,IchimKatthanMoyanoFernandez:2017}.
The CP-order is an object from order theory \cite{FelsnerKappes:2008,Kappes:2006}, used to study orthogonal surfaces~\cite{Miller:2002} (called `grid surfaces' there) and that captures many of the desirable geometric properties one would want a face to exhibit.
Finally, the Scarf poset is derived from the construction of \emph{primitive sets} for Scarf's algorithmic proof of Brouwer's fixed point theorem in~\cite{Scarf:1973}.
While his work operates under a genericity assumption, his definition was generalised to the language of more general monomial ideals in~\cite{BayerPeevaSturmfels:1998, Miller:2000}.
\added{From this viewpoint, the Scarf poset can be viewed as an extension of the Scarf complex to allow unbounded faces.}

\begin{figure} 
  \begin{tikzpicture}
  \matrix (m) [matrix of math nodes,row sep=3em,column sep=5em,minimum width=2em]
          {
            \text{pseudovertex poset} & \\ \text{max-lattice} & \text{LCM-lattice} \\ \text{vertex-facet lattice} & \\ \text{max-min poset} & \\ \text{CP-order} & \text{Betti poset}\\ \text{Scarf poset} & \text{Scarf complex} \\
          };
          \path[]
          (m-1-1) edge [white] node [black] {\rotatebox[origin=c]{90}{$\subset$}} (m-2-1)
          (m-2-1) edge [white] node [black] {\rotatebox[origin=c]{90}{$\subset$}} (m-3-1)
          (m-3-1) edge [white] node [black] {\rotatebox[origin=c]{90}{$\subset$}} (m-4-1)
          (m-4-1) edge [white] node [black] {\rotatebox[origin=c]{90}{$\subset$}} (m-5-1)
          (m-5-1) edge [white] node [black] {\rotatebox[origin=c]{90}{$\subset$}} (m-6-1)
          (m-5-2) edge [white] node [black] {\rotatebox[origin=c]{90}{$\subset$}} (m-6-2)          
          (m-2-2) edge [white] node [black] {\rotatebox[origin=c]{90}{$\subset$}} (m-5-2)
          (m-2-1) edge [white] node [black] {$\supset$} (m-2-2)
          (m-5-1) edge [white] node [black] {$\supset$} (m-5-2)
          (m-6-1) edge [white] node [black] {$\supset$} (m-6-2);

\end{tikzpicture}
  \caption{Different posets serving as face posets; they are all subposets of $\left(\{-\infty\} \cup \RR \cup \{+\infty\}\right)^d$ with the componentwise order. The posets on the left are motivated by geometric constructions, while those on the right arise from commutative algebra. }
  \label{fig:several+face+posets}
\end{figure}

\smallskip

Cryptomorphic to the vertex-facet lattice, we define the \emph{facet complex}, the simplicial complex whose maximal simplices are the vertices incident to a single facet of the monomial tropical polyhedron.
Section~\ref{sec:facet+complex} is dedicated to establishing properties of this complex.
We show that the facet complex captures a certain universal structure:
\begin{theoremnonum}[{Synopsis of Section~\ref{subsec:embedding+facet+complex}}]
The facet complex of a monomial tropical polyhedron contains the following objects as natural subcomplexes:
\begin{enumerate}
\item the facet complex of any lift of the monomial tropical polyhedron,
\item the facet complex of any deformation of the monomial tropical polyhedron,
\item the Scarf complex of the monomial tropical polyhedron.
\end{enumerate}
\end{theoremnonum}
Develin and Yu give a list of desirable behaviour that a face lattice for tropical polytopes should have \cite[Conjecture 4.7]{DevelinYu:2007}.
In particular, they say it should have the homology of a sphere, which the facet complex satisfies:
\begin{theoremnonum}[{Theorem~\ref{thm:facet+complex+sphere}}]
The facet complex of a $d$-dimensional monomial tropical polyhedron is homotopy equivalent to a $(d-1)$-sphere.
\end{theoremnonum}

\smallskip

Finally, we establish a dictionary in Section \ref{sec:monomial+ideals} between monomial ideals and monomial tropical polyhedra to prove the following:
\begin{theoremnonum}[{Theorem~\ref{thm:facet+complex+Betti}}]
The facet complex of a monomial tropical polyhedron encodes the Betti numbers of its associated monomial ideal.
\end{theoremnonum}
We show how recent advances on the structure and resolutions of monomial ideals~\cite{IchimKatthanMoyanoFernandez:2017,Chen:2019,EagonMillerOrdog:2019} are reflected in the geometry of monomial tropical polyhedra.
This complements multiple other connections between tropical convexity and monomial ideals noted in~\cite{BlockYu:2006,DochtermannJoswigSanyal:2012,Noren:2016,MillerSturmfels:2005}.
We make an explicit interpretation of the duality for monomial tropical polyhedra demonstrated in~\cite{JoswigLoho:2017} as Alexander duality for monomial ideals, mirroring the \v{C}ech hull construction from \cite{Miller:1998}.
We further expand on the connection between cellular resolutions of monomial ideals and lifts of monomial tropical polyhedra which was explored in \cite{DevelinYu:2007}.
We also consider two posets associated to monomial ideals, the LCM-lattice \cite{GasharovPeevaWelker:1999} and the Betti poset \cite{ClarkMapes:2014b}, and investigate their relation with the face posets of monomial tropical polyhedra established in Section~\ref{sec:face+posets}, briefly outlined in Figure \ref{fig:several+face+posets}.

\smallskip

We begin our investigation, in Section~\ref{sec:monomial+tropical+polyhedra}, by identifying the \emph{tropical hypercube} $\left(\{-\infty\} \cup \RR \cup \{\infty\}\right)^d$ as the natural space for the face posets and their duality.
In particular, we later extend existing poset constructions to obtain the complete structure mimicking the face poset of a classical polytope.
While one can use projective transformations to reduce the combinatorial study of classical polyhedra to polytopes, this fails in the tropical world due to the lack of appropriate transformations. 
Hence, dealing with rays and generators with non-finite entries was often avoided in former work as it imposes additional technical obstacles.
We accept this additional overhead to lay the groundwork for the further study of face posets of tropical polyhedra.

\section{Monomial tropical polyhedra} \label{sec:monomial+tropical+polyhedra}

\subsection{Tropical hypercube} \label{sec:tropical+hypercube}

We work over $\TTmax = (\RR \cup \{-\infty\}, \oplus, \odot)$, the \emph{$\max$-tropical semiring}, where $\oplus$ denotes the $\max$ operation and $\odot$ denotes addition.
Our definitions of tropical convexity follow~\cite{GaubertKatz:2007}.
We define the \emph{tropical convex hull} of a finite set $V = \{ v^{(1)}, \dots, v^{(n)} \} \subset \TTmax^d$ by 
\begin{equation} \label{eq:convex+hull}
\tconv{V} = \SetOf{\bigoplus_{j=1}^{n} \lambda_j \odot v^{(j)}}{v^{(j)} \in V \ , \ \lambda_j \in \TTmax \ , \ \bigoplus \lambda_j = 0} \enspace .
\end{equation}
This is the \emph{tropical polytope} generated by $V$.
A set is \emph{tropically convex} if it contains the tropical convex hull of each of its finite subsets.
Additionally, we define the \emph{tropical conic hull} of a finite set $W = \{ w^{(1)}, \dots, w^{(m)} \}$ by
\begin{equation} \label{eq:conic+hull}
\tcone{W} = \SetOf{\bigoplus_{j = 1}^{m} \lambda_j \odot w^{(j)}}{w^{(j)} \in W \ , \ \lambda_j \in \TTmax} \enspace .
\end{equation}
\begin{remark}
Some parts of the literature refer to \eqref{eq:conic+hull} as the tropical convex hull.
When working in $\RR^d$, the condition $\bigoplus \lambda_j = 0$ can be obtained by quotienting by scalar addition, a standard practice in tropical geometry.
However this does not hold when working with infinite coordinates in $\TTmax^d$, and so tropical convex and conic hull are necessarily different notions.
\end{remark}
More generally, we can define a \emph{tropical polyhedron} as the tropical sum
\begin{equation} \label{eq:minkowski+decomposition}
Q = \tconv{V} \oplus \tcone{W} = \SetOf{v \oplus w}{v \in \tconv{V} \ , \ w \in \tcone{W}}
\end{equation}
for two finite subsets $V, W \subset \TTmax^d$. 
In this representation, the set $\tcone{W}$ is unique and it is called the \emph{tropical recession cone} of $Q$.

There is a distinguished subset of the generators of $\tconv{V}$ called the \emph{extreme} points, elements that cannot be written as the tropical convex hull of other points of $\tconv{V}$.
These form a minimal generating set for the tropical polytope.
Analogously, there is a distinguished family of points of $\tcone{W}$ called \emph{extreme} that cannot be written as the tropical sum of other points of $\tcone{W}$.
If $w \in \tcone{W}$ is extreme, the points in the set $\SetOf{\lambda \odot w}{\lambda \in \TTmax}$ are also extreme and form an \emph{extremal ray} of the tropical cone.
A set of representatives from the extremal rays yields a minimal generating set for $\tcone{W}$, unique up to choice of representative.
These two minimal generating sets comprise a minimal generating set for $Q$.

Given $Q \subseteq \TTmax^d$, we obtain its \emph{homogenisation} $\widehat{Q} \subseteq \TTmax^{d+1}$ as the tropical cone defined as
\begin{align*}
\widehat{Q} &= \tcone{\widehat{V} \cup \widehat{W}} \enspace , \\
\widehat{V} &= \SetOf{(0,v_1,\dots,v_d)}{v \in V} \enspace , \\
\widehat{W} &= \SetOf{(-\infty,w_1,\dots,w_d)}{w \in W} \enspace .
\end{align*}
Similarly, we refer to $\widehat{V}$ and $\widehat{W}$ as the homogenisation of the points and rays respectively.
By \cite[Proposition 4]{AllamigeonGaubertGoubault:2013}, identifying $\TTmax^d$ with $\{0\} \times \TTmax^d \subset \TTmax^{d+1}$ allows us to recover $Q$ from $\widehat{Q}$ via
\begin{align*}
\{0\} \times Q = \widehat{Q} \cap (\{0\} \times \TTmax^d) \enspace .
\end{align*}
Moreover, the (minimal) generators of $\widehat{Q}$ define the (minimal) generators of $Q$.

By the tropical Minkowski-Weyl theorem, \cite[Theorem 1]{GaubertKatz:2011}, such a tropical polyhedron can also be written as the intersection of finitely many max-tropical \discretionary{half-}{spaces}{halfspaces}, which are of the form
\begin{equation} \label{eq:general+halfspace}
H(a,I) = \SetOf{x \in \TTmax^{d}}{\bigoplus_{i \in I} a_i \odot x_i \geq \bigoplus_{j \in [d]_0 \setminus I} a_j \odot x_j \ , \ x_0 = 0} \enspace .
\end{equation}
for some $(a_0,a_1,\ldots,a_d) \in \TTmax^{d+1}$ and a subset $I$ of $[d]_0 := \{0,1,\dots, d\}$.

The dual point $-a \in \TTmin^{d+1}$ is unique up to scaling and it is often called the apex of the halfspace.
It has the property that evaluating the inequality in~\eqref{eq:general+halfspace} yields the same value for all the products $a_k \odot x_k$ with $k \in [d]$.
The point naturally lives in the dual space $\TTmin^{d+1}$, where $\TTmin = (\RR \cup \{+\infty\}, \min, +)$ is the \emph{min-tropical semiring}.
We will use the notion of an apex in a slightly different way tailored to the specific class of tropical polyhedra we are interested in. 

We consider the $\max$-tropical semiring embedded in the space
\[
\TTclosed = \{-\infty\} \cup \RR \cup \{\infty\} \enspace .
\]
This leads to the $d$-dimensional \emph{tropical hypercube}, the space $\TTclosed^d$.

Tropical polyhedra and their defining halfspaces naturally live in spaces that are dual to each other, namely $\TTmax$ and $\TTmin$.
To capture both features at once, it will be beneficial to consider a larger space that comprises the two spaces. 
$\TTclosed$ is precisely that, with $\TTmax$ and $\TTmin$ identified along their common elements.

\begin{remark}
  It will be useful to us to consider the tropical hypercube as a topological space, in particular, as a compactification of $\TTmax^d$.
  We imbue $\TTmax^d$ with the product topology induced by the order topology on $\TTmax$.
With this topology, $\TTmax^d$ is dense in $\TTclosed^d$ and so the tropical hypercube is the compactification of $\TTmax^d$.
This follows by considering any point $p \in \TTclosed^d$ as the limit of a points $p^{(k)} = (p_1^{(k)},\dots,p_d^{(k)}) \in \TTmax^d$, where $p_i^{(k)} = p_i$ if $p_i \in \TTmax$ and $(p_i^{(k)})_{k \in \NN}$ is a strictly increasing divergent sequence otherwise.
Note that $\TTclosed^d$ is also the compactification of $\TTmin^d$ with respect to this topology.
\end{remark}

\begin{remark}
Compactifications are widely used in tropical geometry.
However, one usually takes the quotient 
\[
\TPmax^d =\left(\TTmax^{d+1} \setminus \{(-\infty,\dots,-\infty)\}\right) / \RR \cdot \1
\]
as a compactification of the $d$-dimensional space.
In dimension one, \added{there is a bijection between the tropical projective line $\TP^1$ and our compactification $\TTclosed$ given by $(x_1,x_2) \mapsto x_1-x_2$.}
The classical projective line $\mathbb{P}^1$ can be formed by taking two copies of $\mathbb{A}^1$ and gluing them by identifying $x$ and $x^{-1}$.
Tropically, this is done by gluing two copies of $\TTmax$ by identifying $x$ and $-x$, or equivalently by gluing $\TTmax$ and $\TTmin$ along $\RR$.
\added{This bijection extends to higher dimensions, where there is a natural identification between $\TTclosed^d$ and $(\TP^1)^d$.}
\end{remark}

The \emph{$\max$-tropical unit vectors} $e^{(1)}, \ldots, e^{(d)} \in \TTmax^{d}$ are given by
\[
e^{(i)}_{k} \ = \
\begin{cases}
  0  & \mbox{ if }  i = k \\
  -\infty & \mbox{ otherwise }
\end{cases}
\qquad \text{for } 1 \leq i,k \leq d \enspace .
\]
We set
\begin{equation} \label{eq:def+maxunit}
\maxunit \ = \ \bigl\{ e^{(1)}, e^{(2)}, \dots, e^{(d)} \bigr\} \ \subseteq \ \TTmax^{d} \enspace . 
\end{equation}
This also gives rise to the dual \emph{$\min$-tropical unit vectors} $\minunit = -\maxunit$.

\subsection{Monomial tropical polyhedra}

Our main object of study are tropical polyhedra whose recession cone is $\tcone{\maxunit}$, the span of the $\max$-tropical unit vectors.
\begin{definition} \label{def:monomial+tropical+polyhedron}
For a finite set $V \subset \TTmax^d$, we define the \emph{monomial tropical polyhedron} by
\[
\monomial{V} = \tconv{V} \oplus \tcone{\maxunit} \enspace .
\]
\end{definition}
Due to the special structure of its recession cone, these tropical polyhedra have a unique minimal set of extremal generators for the polytope part $\tconv{V}$, which we call its \emph{vertices}.

A different point of view on monomial tropical polyhedra comes from the observation that we can also represent it as a classical Minkowski sum with a non-negative orthant
\begin{equation} \label{eq:representation+sum+orthant}
\monomial{V} = V + \RR_{\geq 0}^d \enspace .
\end{equation}
This also yields the connection with multicriteria optimisation~\cite{DaechertKlamrothLacourVanderpooten:2017} where the latter construction leads to the set of points dominated by $V$.

For a subset $J \subseteq [d]$ we introduce the vector $f^J \in \TTclosed^d$
\[
f^J_{k} \ = \
\begin{cases}
  +\infty  & \mbox{ if } k \in J \\
  0 & \mbox{ otherwise }
\end{cases}
\qquad \text{for } 1 \leq i,k \leq d \enspace .
\]
To capture all features on the boundary, we define the \emph{closed monomial tropical polyhedron} by
\[
\closedmonomial{V} = \bigcup_{J \subseteq [d]} \left(f^J + \monomial{V}\right) \subset \TTclosed^d \enspace .
\]
Note that this differs slightly from the use of this notation in~\cite{JoswigLoho:2017}; this has been adapted due to the focus on tropical polyhedra than on tropical cones.

The external representation of monomial tropical polyhedra is as follows.
All defining halfspaces are of the form 
\begin{equation} \label{eq:principal+halfspace}
H(c) = \SetOf{x \in \TTmax^{d}}{\bigoplus_{i \in [d]} c_i \odot x_i \geq 0} \enspace ,
\end{equation}
with $(c_1,\dots,c_d) \in \TTmax^{d}$. We call the dual point $a = -c \in \TTmin^{d}$ the \emph{apex} of the halfspace.
Unlike general tropical polyhedra, monomial tropical polyhedra have a unique minimal exterior description.

This can be nicely seen through a particular duality exhibited in~\cite{JoswigLoho:2017}.
Let us define $\closedcomplementarymonomial{V}$ as the closure in $\TTclosed^d$ of the complement of the closed monomial tropical polyhedron $\TTclosed^d \setminus \closedmonomial{V}$ and $\complementarymonomial{V}$ as $\closedcomplementarymonomial{V} \cap \TTmin^d$. 
We state~\cite[Theorem 10]{JoswigLoho:2017} in a slightly modified version, tailored to our purposes.

\begin{theorem} \label{thm:complementary-cones}
  The set $\complementarymonomial{V}$ is a $\min$-tropical polyhedron in $\TTmin^{d}$.
  It has the exterior description $\bigcap_{v \in V} -H(-v)$.
  Furthermore, if $\mathcal{H}$ is a set of $\max$-tropical halfspaces such that $\bigcap \mathcal{H} = \monomial{V}$, then
  \[
  \complementarymonomial{V} \ = \ - \monomial{-A} \enspace ,
  \]
  where $A \subset \TTmin^{d}$ is the set of apices of the tropical halfspaces in $\mathcal{H}$.
 In particular, $\complementarymonomial{V} = \tconv{A} \oplus \tcone{\minunit}$.
\end{theorem}

\begin{corollary} \label{cor:unique+exterior+description}
  Let $\monomial{V}$ be a monomial tropical polyhedron generated by $V$. Then there is a unique inclusionwise minimal finite set $A \subset \TTmin^d$ such that
  \[
  \monomial{V} = \bigcap_{a \in A} H(-a) \enspace .
  \]
\end{corollary}

For a classical pointed polyhedron $P$, there is always a projective transformation mapping it to a polytope $\overline{P}$, allowing one to assign the face lattice of a polytope to a polyhedron~\cite{JoswigKaibelPfetschZiegler:2001}.
We shall perform a similar construction for $\monomial{V}$ and $\closedmonomial{V}$.

While the Corollary~\ref{cor:unique+exterior+description} describes the \emph{principal halfspaces}, we also introduce additional inequalities \added{corresponding to the \emph{non-negativity constraints} on $\TTmax^d$. 
Explicitly, we include} (at most) $d$ inequalities of the form $x_i \geq -\infty$ for $i \in [d]$.
We add the inequality $x_i \geq -\infty$ if there is a vertex $v \in V$ with $v_i = -\infty$.
These are the \emph{boundary inequalities}, and including them allows us to get the exterior polyhedral description in an analogous form as for a classical polytope.
\added{We note that there have been recent attempts to formulate tropical polyhedra over the signed tropical numbers, where one does not have these non-negativity constraints~\cite{LohoVegh:2020}}.

\added{To represent boundary inequalities as tropical halfspaces with apices, we note that each boundary inequality is the limit of the inequality $x_i \geq -c$ as $c$ goes to infinity.
These can be equivalently expressed as the tropical linear inequality
\begin{equation} \label{eq:boundary+inequality}
x_i \odot c \geq 0 \oplus \bigoplus_{j \in [d] \setminus \{i\}}(x_j \odot -\infty)  \enspace ,
\end{equation}
which has a well-defined apex $(\infty, \dots, -c,\dots, \infty) \in \TTmin^d$ with all entries $+\infty$ except for $-c$ in the $i$th coordinate.
Taking the limit of this apex as $c \rightarrow \infty$ gives us the \emph{boundary apex} $(\infty, \dots, -\infty,\dots, \infty) \in \TTclosed^d$ of the $i$th boundary inequality.
By defining the boundary apex as a limit of apices of principal halfspaces, the combinatorics of boundary apices mirrors that of apices of principal halfspaces; see Proposition~\ref{prop:facetapex+characterisation}.}

Moreover, we introduce a special superfluous inequality $0 \geq -\infty$, equivalent to
\begin{equation} \label{eq:far+face}
0 \geq \bigoplus_{i \in [d]} x_i \odot -\infty \enspace ,
\end{equation}
which determines the \emph{far face}.
The notion of the far face originally occurs in the classical construction as the unique maximal face of $\overline{P}$ not in the images of faces of $P$.

While Corollary~\ref{cor:unique+exterior+description} describes $\monomial{V}$ in $\TTmax^d$, the purpose of these additional inequalities is to describe $\closedmonomial{V}$ including its boundary correctly in $\TTclosed^d$. 
For example, the far face inequality~\eqref{eq:far+face} is not tight for any points of $\monomial{V}$, but it is tight for points of the boundary of $\closedmonomial{V}$.

\begin{example} \label{ex:model}

We examine ``the model" introduced by Develin and Yu~\cite{DevelinYu:2007}.
This is the tropical polytope $\tconv{V} \subset \TTmax^3$ shown in Figure~\ref{fig:the+model}, generated by
\begin{align*}
V \ = \
\begin{blockarray}{cccccc}
A & B & C & D & E & F \\
\begin{block}{(cccccc)}
2 & 2 & 1 & 1 & 1 & 1 \\
0 & 1 & 2 & 3 & 4 & 5 \\
1 & 0 & 5 & 4 & 3 & 2 \\
\end{block}
\end{blockarray} \enspace .
\end{align*}
\added{Note that we will use the same non-standard axes for all 3D figures as those given in used in~Figure~\ref{fig:the+model}.
This allows us to display boundary behaviour cleanly in later examples, and allows direct comparison with the examples given in~\cite{DevelinYu:2007} where the same axes are used.}

\begin{figure}[h]
    \includegraphics[width=0.5\columnwidth]{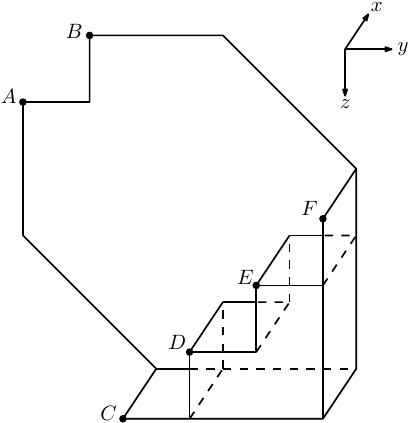}%
    \caption{The model $\tconv{V}$ from~\cite{DevelinYu:2007}, \added{along with the relevant axes}.}
    \label{fig:the+model}
\end{figure}
\begin{figure}[h]
    \includegraphics[width=0.6\columnwidth]{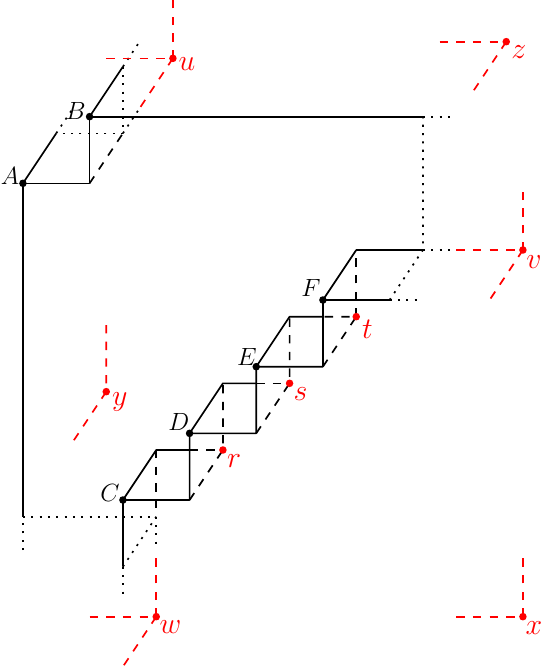}%
    \caption{The monomial tropical polyhedron $\monomial{V}$; the apices of its principal halfspaces are marked in red.}
    \label{fig:the+model+mono+generators}
\end{figure}

The monomial tropical polyhedron $\monomial{V}$ generated by $V$ is shown in Figure \ref{fig:the+model+mono+generators}.
The apices of its principal halfspaces are shown in red.
Explicitly, they are the following points in $\TTmin^3$:
\begin{align*}
\begin{blockarray}{ccccccccc}
r & s & t & u & v & w & x & y & z \\
\begin{block}{(ccccccccc)}
2 & 2 & 2 & \infty & 2 & 2 & 1 & \infty & \infty \\
3 & 4 & 5 & 1 & \infty & 2 & \infty & 0 & \infty \\
5 & 4 & 3 & 1 & 2 & \infty & \infty & \infty & 0 \\
\end{block}
\end{blockarray}
\end{align*}
As no generators have $-\infty$ in any coordinate, we do not require any boundary inequalities of the form \eqref{eq:boundary+inequality}.
The only additional inequality we add is the far-face given in equation \eqref{eq:far+face}.
\end{example}

\section{Vertex-facet lattice} \label{sec:vertex+facet+lattice}

As monomial tropical polyhedra have a unique minimal set of non-redundant vertices and defining halfspaces, their incidences form a canonical notion of \emph{combinatorial type} for a monomial tropical polyhedron.
While this notion is motivated from the combinatorial type of classical polytopes, we discuss the relation with the existing notion of combinatorial type for tropical polytopes arising from the covector decomposition~\cite{FinkRincon:2015} in Section~\ref{subsec:covector+decomposition}.

To define a notion of face lattice, we require concrete characterisations of apices of facets and incidence in $\TTclosed^d$.

\subsection{Facet-apices and incidence}

The natural partial order on $\TTclosed$ is the standard partial order on $\RR^d$ extended to $\TTclosed^d$.
This means that
\begin{equation}
  x \leq y \quad \Leftrightarrow \quad x_i \leq y_i \text{ for all } i \in [d] \enspace ,
\end{equation}
and with $\leq$ replaced by $<$, respectively.

The following statement appears in the literature in~\cite{AllamigeonKatz:2013} and~\cite{DaechertKlamrothLacourVanderpooten:2017} for tropical polyhedra in $\RR^d$.
We give a simple extension to $\TTclosed^d$.

\begin{proposition} \label{prop:facetapex+characterisation}
  A point $a \in \TTclosed^{d}$ is an apex of a principal or non-redundant boundary halfspace of the monomial tropical polyhedron $\monomial{V}$ if and only if
  \begin{enumerate} 
  \item there is no generator $v \in V$ with $v < a$, \label{eq:apex+cond+1}
  \item for each $i \in [d]$ with $a_i \neq \infty$ there exists a generator $v$ such that $v_i = a_i$ and $v_k < a_k$ for all $k \neq i$. \label{eq:apex+cond+2}
  \end{enumerate}
\end{proposition}
\begin{proof}
The arguments of~\cite{AllamigeonKatz:2013} and~\cite{DaechertKlamrothLacourVanderpooten:2017} still hold if one allows for $-\infty$ as coordinates of the generators i.e. $a \in \TTmin^d$ a principal apex.
For any other points $a \notin \TTmin^d$, at least one coordinate must satisfy $a_i = -\infty$\added{: this immediately implies the first condition cannot hold.}
For any coordinate $k \neq i$, if $a_k \neq \infty$ then the second condition implies we can find some generator $v$ such that $v_k = a_k$ and $v_i < a_i = -\infty$, giving a contradiction.
Therefore the only remaining points that may satisfy these conditions are boundary apices.
\added{The boundary apices satisfy condition~\eqref{eq:apex+cond+2} precisely when the boundary inequality $x_i \geq -\infty$ is tight with respect to some generator.}
\end{proof}

We define a \emph{facet-apex} to be any point in $\TTclosed^d$ satisfying the condition of Proposition \ref{prop:facetapex+characterisation}, and we denote the set of them by $F$.
Those arising as apices of principal halfspaces we call \emph{principal apices} and those arising from non-redundant boundary inequalities we call \emph{boundary apices}.
We also associate a \emph{far-apex} $b^{\infty}$ without geometric meaning in $\TTclosed^d$ to the inequality \eqref{eq:far+face} corresponding to the far face.

Using Proposition~\ref{prop:facetapex+characterisation} and the duality from Theorem~\ref{thm:complementary-cones} we get an analogous characterisation of the vertices.

\begin{corollary} \label{cor:vertex+containment}
A point $v \in \TTclosed^d$ is a minimal generator of the monomial tropical polyhedron $\monomial{V}$ if and only if
\begin{enumerate}
\item there is no facet-apex $a \in A$ with $v < a$,
\item for each $i \in [d]$ with $v_i \neq -\infty$ there exists a facet-apex $a$ such that $v_i = a_i$ and $v_k < a_k$ for all $k \neq i$.
\end{enumerate}
\end{corollary}

We are now ready to define the notion of incidence among points and rays. 
Let $p$ be a point in $\TTmax^d$ and $q \in \TTmin^d$ the apex of the halfspace $H(-q)$ of the form \eqref{eq:principal+halfspace}.
We say $p$ is incident to $H(-q)$ if and only if the tropical inequality $\bigoplus_{k \in [d]} x_k \odot -q_k \geq 0$ is tight at the point $p$.
This leads to the following notion of incidence in $\TTclosed^d$.

\begin{definition} \label{def:vertex+apex+incidence}
A point $p \in \TTclosed^d$ is \emph{incident} with a point $q \in \TTclosed^d$ if $p \leq q$ and there exists some coordinate $i \in [d]$ such that $p_i = q_i$.
A ray $e^{(i)} \in \maxunit$ is incident with a point $q \in \TTclosed^d$ if $q_i = \infty$.
Additionally, each of the rays is incident with the far-apex. 
\end{definition}

Note that these two notions of incidence coincide when we homogenise, where points and rays are indistinguishable.
Consider the points $\widehat{e}^{(i)} = (-\infty, e^{(i)})$ in $\TTmax^{d+1}$ and $\widehat{q} = (0,q)$ in $\TTmin^{d+1}$.
The point $\widehat{e}^{(i)}$ is tight at the homogeneous tropical inequality $\bigoplus_{k \in [d]} x_k \odot -q_k \geq x_0 \odot - q_0$ if and only if $-q_i = -\infty$.
One can also check that these homogeneous unit vectors are tight with the homogenised far face inequality~\eqref{eq:far+face}.

We are now ready to consider the incidences in a monomial tropical polyhedron $\monomial{V}$.
Recall that the elements of $V$, the unique inclusion-wise minimal set needed to describe a monomial tropical polyhedron in the form~\eqref{eq:representation+sum+orthant}, are the \emph{vertices}.
The tropical unit vectors $\maxunit$ are its \emph{rays}.
Let $\overline{V} = V \cup \maxunit$ denote the set of vertices and rays of $\monomial{V}$.
Let $\overline{F} = F \cup b^{\infty}$ denote the set of facet-apices and the far-apex of $\monomial{V}$.

\begin{definition} \label{def:vertex-facet-incidence}
The \emph{vertex-facet incidence graph} is the bipartite graph on the node set $\overline{V} \sqcup \overline{F}$ with edge $(v,a)$ if $v$ is incident to the apex $a$.
\end{definition}

\begin{example} \label{ex:2d+monomial+polyhedron}
Consider the monomial tropical polyhedron $\monomial{V}$ shown in Figure~\ref{fig:2d+monomial+polyhedron}.
Its generators and facet-apices are
\[
V \ = \
\begin{bmatrix}
v^{(1)} & v^{(2)}
\end{bmatrix}
\ = \
\begin{bmatrix}
1 & 2 \\
2 & -\infty
\end{bmatrix}
\quad , \quad
F \ = \
\begin{bmatrix}
a^{(1)} & a^{(2)} & a^{(3)}
\end{bmatrix}
\ = \
\begin{bmatrix}
1 & 2 & \infty \\
\infty & 2 & -\infty
\end{bmatrix} \enspace .
\]
Its vertex-facet incidence graph is also shown in this figure.
\end{example}
\begin{figure}
\includegraphics[width=0.49\textwidth]{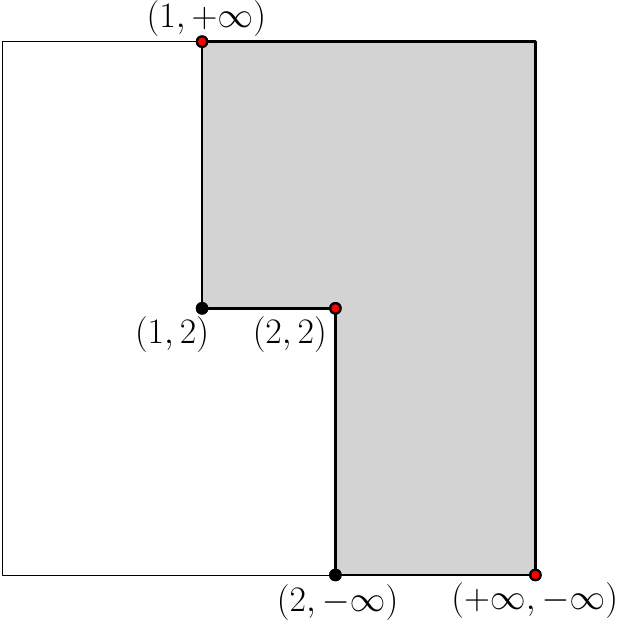}
\raisebox{0.2\height}{\resizebox{0.35\textwidth}{!}{\begin{tikzpicture}
	\node (v1) at (0,3){$v^{(1)}$};
	\node (v2) at (0,2){$v^{(2)}$};
	\node (e1) at (0,1){$e^{(1)}$};
	\node (e2) at (0,0){$e^{(2)}$};
	
	\node (a2) at (3,3){$a^{(1)}$};
	\node (a1) at (3,2){$a^{(2)}$};
	\node (b2) at (3,1){$a^{(3)}$};
	\node (inf) at (3,0){$b^{\infty}$};	
	
	\draw (v1) -- (a2);
	\draw (v1) -- (a1);
	\draw (e2) -- (a2);
	\draw (e2) -- (inf);
	\draw (e1) -- (inf);
	\draw (e1) -- (b2);
	\draw (v2) -- (b2);
	\draw (v2) -- (a1);
\end{tikzpicture}}}
\caption{$\monomial{V}$ and its vertex-facet incidence graph defined in Example \ref{ex:2d+monomial+polyhedron}.}
\label{fig:2d+monomial+polyhedron}
\end{figure}

\subsection{Vertex-facet lattice}

We briefly recall some necessary concepts from lattice theory, we refer to~\cite{Birkhoff:1967} for more details.
Let $L$ be a lattice with minimal element $\hat{0}$ and maximal element $\hat{1}$.
The \emph{atoms} of $L$ are the elements $a \in L$ such that $\hat{0} < b \leq a$ implies $a = b$.
The \emph{coatoms} of $L$ are the elements $c \in L$ such that $c \leq d < \hat{1}$ implies $c = d$.
We shall insist that our lattices are finite, therefore the atoms and coatoms of $L$ are well defined.
Furthermore, every finite lattice is \emph{complete} i.e., every subset has a greatest lower bound and a least upper bound.

Let $\cF = \{F_1,\dots, F_k\}$ be a finite collection of sets on some ground set $E$.
We define a closure operator $\clos{\cdot}$ on $E$ by
\[
\clos{A} = \bigcap_{A \subseteq F_i \in \cF} F_i \subseteq E \enspace .
\]
This induces a corresponding closure operator on $\cF$ given by
\[
\clos{\{F_{i_1},\dots,F_{i_k}\}} = \SetOf{F_i \in \cF}{F_i \supseteq \bigcap_{\ell = 1}^k F_{i_\ell}} \enspace .
\]
This construction describes a complete lattice $L$, given by either the closed subsets of $E$ ordered by inclusion or the closed subsets of $\cF$ ordered by reverse inclusion.
We consider these as two distinct labellings for $L$.

\begin{example} \label{ex:closure+from+bipartite+graph}
Let $G$ be a bipartite graph on disjoint node sets $U, V$.
Then $G$ describes a set system on the ground set $U$ via the neighbourhoods $\neighbour(v)$ of $v \in V$, i.e.,
\[
\SetOf{\neighbour(v) \subseteq U}{v \in V} \enspace .
\]
This gives rise to the lattice $L$ of closed sets of $U$.
Reversing the roles of $U$ and $V$ in this construction gives the alternative labelling of $L$ by closed sets of $V$.
\end{example}

\begin{example} \label{ex:cone+face+lattice}
Let $C$ be a cone with rays $R$ and facets $F$.
The ray-facet incidence relations of $C$ form a bipartite graph on the node set $R\sqcup F$ whose edges encode whether a ray is contained in a facet.
The lattice $L$ of closed sets is the \emph{face lattice} of $C$, and can be defined by the conic hull of rays or the intersection of facets.
We remark that the dual lattice of closed sets of $F$ ordered by inclusion is the face lattice of the dual cone $C^*$.

Any polyhedron $P \subseteq \RR^d$ can be realised as the intersection of a cone $C \subseteq \RR^{d+1}$ with the hyperplane $x_0 = 0$.
Furthermore, the face lattice of $C$ is isomorphic to the face lattice of $P$, and duality is preserved i.e., the face lattice of $C^*$ is isomorphic to the face lattice of $P^*$.
If $P$ is a polytope, this is the lattice induced by the vertex-facet incidence relations.
However if $P$ is unbounded, certain rays of $C$ will not be geometrically realised as vertices of $P$, rather as unbounded rays.
It is still beneficial to record the incidence data of these rays as it is crucial for recovering the face lattice of the dual polyhedron, see also~\cite{JoswigKaibelPfetschZiegler:2001}.
\end{example}

Example~\ref{ex:cone+face+lattice} motivates our definition for a face lattice of a monomial tropical polyhedron.
\begin{definition} \label{def:vertex-facet-lattice}
Consider the set system induced by the vertex-facet incidence graph, in the sense of Example~\ref{ex:closure+from+bipartite+graph}.
The lattice of closed sets is the \emph{vertex-facet lattice} $\cV$ of~$\monomial{V}$.
\end{definition}


\begin{figure}
\resizebox{0.6\textwidth}{!}{\begin{tikzpicture}[scale=1.5]
	\node (0) at (0,0.3){$\emptyset$};
	\node (v1) at (-1.5,1){$v^{(1)}$};
	\node (v2) at (-0.5,1){$v^{(2)}$};
	\node (e1) at (0.5,1){$e^{(1)}$};
	\node (e2) at (1.5,1){$e^{(2)}$};
	\node (v1v2) at (-1.5,2){$v^{(1)}v^{(2)}$};
	\node (v1e2) at (-0.5,2){$v^{(1)}e^{(2)}$};	
	\node (v2e1) at (0.5,2){$v^{(2)}e^{(1)}$};
	\node (e1e2) at (1.5,2){$e^{(1)}e^{(2)}$};
	\node (v1v2e1e2) at (0,2.7){$v^{(1)}v^{(2)}e^{(1)}e^{(2)}$};
	
	\draw (0) -- (v1);
	\draw (0) -- (v2);
	\draw (0) -- (e1);
	\draw (0) -- (e2);
	\draw (v1) -- (v1v2);
	\draw (v1) -- (v1e2);
	\draw (v2) -- (v1v2);
	\draw (v2) -- (v2e1);
	\draw (e1) -- (v2e1);
	\draw (e1) -- (e1e2);
	\draw (e2) -- (v1e2);
	\draw (e2) -- (e1e2);
	\draw (v1v2) -- (v1v2e1e2);
	\draw (v1e2) -- (v1v2e1e2);
	\draw (v2e1) -- (v1v2e1e2);
	\draw (e1e2) -- (v1v2e1e2);
\end{tikzpicture}}
\caption{The vertex-facet lattice of $\monomial{V}$ from Example \ref{ex:2d+monomial+polyhedron}.}
\label{fig:2d+vf+lattice}
\end{figure}
The vertex-facet lattice of the monomial tropical polyhedron described in Example \ref{ex:2d+monomial+polyhedron} is given in Figure \ref{fig:2d+vf+lattice}.
We note that the two-dimensional case is misleading in its simplicity.
The following example demonstrates that vertex-facet lattices can have seemingly pathological behaviour.

\begin{example} \label{ex:vif+model}
We continue with $\monomial{V}$ from Example~\ref{ex:model}.
The vertex-facet incidence graph is shown in Figure~\ref{fig:model+incidence}.
The corresponding vertex-facet lattice is large, therefore we show just a section of it.
Specifically, we show all maximal chains passing through $AB$.
Unlike the face lattice of a classical polytope, we note that not all maximal chains are of the same length.
As a result, vertex-facet lattices do not admit a grading and so we have no notion of the `combinatorial' dimension of a face.

\begin{figure}[ht!]%
\resizebox{0.3\textwidth}{!}{
\begin{tikzpicture}
	\node (a) at (0,8){$A$};
	\node (b) at (0,7){$B$};
	\node (c) at (0,6){$C$};
	\node (d) at (0,5){$D$};
	\node (e) at (0,4){$E$};
	\node (f) at (0,3){$F$};
	\node (e1) at (0,2){$e^{(1)}$};
	\node (e2) at (0,1){$e^{(2)}$};
	\node (e3) at (0,0){$e^{(3)}$};
	
	\node (r) at (5,8.5){$r$};
	\node (s) at (5,7.5){$s$};
	\node (t) at (5,6.5){$t$};
	\node (u) at (5,5.5){$u$};
	\node (v) at (5,4.5){$v$};
	\node (w) at (5,3.5){$w$};
	\node (x) at (5,2.5){$x$};
	\node (y) at (5,1.5){$y$};
	\node (z) at (5,0.5){$z$};
	\node (inf) at (5,-0.5){$\bm \infty$};
	
	\draw (a) -- (r);
	\draw (a) -- (s);
	\draw (a) -- (t);
	\draw (a) -- (u);
	\draw (a) -- (v);
	\draw (a) -- (w);
	\draw (a) -- (y);
	\draw (b) -- (r);
	\draw (b) -- (s);
	\draw (b) -- (t);
	\draw (b) -- (u);
	\draw (b) -- (v);
	\draw (b) -- (w);
	\draw (b) -- (z);
	\draw (c) -- (r);
	\draw (c) -- (w);
	\draw (c) -- (x);
	\draw (d) -- (r);
	\draw (d) -- (s);
	\draw (d) -- (x);
	\draw (e) -- (s);
	\draw (e) -- (t);
	\draw (e) -- (x);
	\draw (f) -- (t);
	\draw (f) -- (v);
	\draw (f) -- (x);
	\draw (e1) -- (u);
	\draw (e1) -- (y);
	\draw (e1) -- (z);
	\draw (e1) -- (inf);
	\draw (e2) -- (v);
	\draw (e2) -- (x);
	\draw (e2) -- (z);
	\draw (e2) -- (inf);
	\draw (e3) -- (w);
	\draw (e3) -- (x);
	\draw (e3) -- (y);
	\draw (e3) -- (inf);
	
\end{tikzpicture}%
}
\resizebox{0.68\textwidth}{!}{
\begin{tikzpicture}
\node (0) at (0,0){$\emptyset$};
\node (a) at (-1,1){$A$};
\node (b) at (1,1){$B$};
\node (ab) at (0,2){$AB$};
\node (abc) at (-1,4){$ABC$};
\node (abd) at (1,4){$ABD$};
\node (abe) at (3,4){$ABE$};
\node (abf) at (5,4){$ABF$};
\node (abe1) at (-5,6){$ABe^{(1)}$};
\node (abcd) at (0,6){$ABCD$};
\node (abde) at (2,6){$ABDE$};
\node (abef) at (4,6){$ABEF$};
\node (abfe2) at (6,6){$ABFe^{(2)}$};
\node (abce3) at (-2,6){$ABCe^{(3)}$};
\node (abcdefe1e2e3) at (0,8){$ABCDEFe^{(1)}e^{(2)}e^{(3)}$};

\draw (0) -- (a);
\draw (0) -- (b);
\draw (a) -- (ab);
\draw (b) -- (ab);
\draw (ab) -- (abc);
\draw (ab) -- (abd);
\draw (ab) -- (abe);
\draw (ab) -- (abf);
\draw (ab) -- (abe1);
\draw (abc) -- (abcd);
\draw (abd) -- (abcd);
\draw (abd) -- (abde);
\draw (abe) -- (abde);
\draw (abe) -- (abef);
\draw (abf) -- (abef);
\draw (abf) -- (abfe2);
\draw (abc) -- (abce3);
\draw (abcd) -- (abcdefe1e2e3);
\draw (abde) -- (abcdefe1e2e3);
\draw (abef) -- (abcdefe1e2e3);
\draw (abfe2) -- (abcdefe1e2e3);
\draw (abce3) -- (abcdefe1e2e3);
\draw (abe1) -- (abcdefe1e2e3);

\end{tikzpicture}%
}
\caption{Pictured left is the vertex-facet incidence graph of the monomial tropical polyhedron induced by the model. Pictured right is the subposet of the vertex-facet lattice given by all chains containing $AB$.}%
\label{fig:model+incidence}%
\end{figure}

\end{example}

Section \ref{sec:face+posets} is dedicated to comparing the vertex-facet lattice $\cV$ to existing posets.
To do this, it will be necessary to consider a subposet of $\cV$.

\begin{definition} \label{def:affine+part+vif}
  The \emph{affine part of $\cV$} is the induced subposet of $\cV$ whose elements are closed subsets $S \subseteq \overline{V}$ such that $S \cap V \neq \emptyset$ or $S = \emptyset$.
\end{definition}
We shall see that these are the elements that have geometric representations in $\TTclosed^d$ via their tropical barycenter described in Section~\ref{subsec:lcm+lattice}. 

\begin{lemma}
The affine part of the vertex-facet lattice is a lattice.
\end{lemma}
\begin{proof}
Let $S_1,S_2$ be closed subsets of $\overline{V}$ such that $S_i \cap V \neq \emptyset$ or $S_i = \emptyset$, it suffices to show the affine part contains a unique least upper bound and greatest lower bound of $S_1,S_2$.
The least upper bound of $S_1, S_2$ in $\cV$ contains $S_1 \cup S_2$ (with equality if $S_1 = S_2 = \emptyset$), therefore also satisfies this condition.
Consider the greatest lower bound $S$ of $S_1,S_2$ in $\cV$, either it satisfies $S \cap V \neq \emptyset$, or it is not contained in the affine part.
In the latter case, the greatest lower bound of $S_1$ and $S_2$ is the empty set in the affine part.
\end{proof}

\begin{remark}
The facets in the sense of Joswig~\cite[\S 3]{Joswig:2005} are exactly the halfspaces given by the facet-apices.
The construction described around~\cite[Theorem 3.7]{Joswig:2005}, extended to tropical polyhedra and considered as a lattice, yields a face poset isomorphic to the vertex-facet lattice.
This follows since the main condition of facets to form a face is that their meet is not empty.
However, we refrain from assigning a geometric object due to the discrepancy described in Section~\ref{subsec:max+min+poset} between the two possible labels arising from the duality between vertices and facet-apices.

Similar issues were discussed by Develin \& Yu concerning the facets by Joswig in~\cite{DevelinYu:2007}.
Consequently, they introduced a concept of face in terms of lifts of tropical polyhedra (see Section~\ref{sec:resolutions} for the definition of a lift).
Explicitly, a face is a minimal subset of the boundary which corresponds to faces of a lift of the same dimension, for any choice of lift.
This notion is rather coarse as it tries to unify the face structure of \emph{all} lifts.
This makes it less suitable for our framework of face posets. 
\end{remark}

\begin{remark}
  The paper~\cite{DaechertKlamrothLacourVanderpooten:2017} gives an efficient algorithm for computing the optima of a discrete multicritera optimisation problem.
  From our perspective, their result uses the combinatorics of the vertex-facet incidences to update the search region in the computation of the optima.
  Their update procedure is to introduce a new vertex and efficiently update the corresponding vertex-facet incidences. 
  Indeed, \cite[\S 3 \& 4]{DaechertKlamrothLacourVanderpooten:2017} shows a clever way to decompose and traverse the dual graph of the facet-apices. 
\end{remark}

\subsection{Intermediate faces}
As Proposition~\ref{prop:facetapex+characterisation} and Corollary~\ref{cor:vertex+containment} give us a geometric characterisation of the vertices and facet-apices, the next natural question is whether we can characterise all faces of the vertex-facet lattice.
The following example suggests this question is more difficult by demonstrating how degenerate some faces can be.
\begin{example} \label{ex:degenerate+polyhedra}
  We consider two small examples which highlight the degenerate nature of monomial tropical polyhedra.
  Consider the generating sets
  \begin{align*}
  V_1 \ = \
  \begin{blockarray}{ccc}
    a & b & c \\
  \begin{block}{(ccc)}
    1 & 0 & 0 \\
    0 & 1 & 0 \\
    0 & 0 & 1 \\
  \end{block}
  \end{blockarray}
  \qquad
  V_2 \ = \
  \begin{blockarray}{ccc}
    u & v & w \\
  \begin{block}{(ccc)}
    0 & 1 & 1 \\
    1 & 0 & 1 \\
    1 & 1 & 0 \\
  \end{block}
  \end{blockarray}
  \enspace .
  \end{align*}
The monomial tropical polyhedra $\monomial{V_1}$ and $\monomial{V_2}$ are displayed in Figure \ref{fig:degenerate+polyhedra}.
The first, $\monomial{V_1}$, has the facets 
\[
abc \ , \ abe^{(1)}e^{(2)} \ , \ ace^{(1)}e^{(3)} \ , \ bce^{(2)}e^{(3)} \enspace ,
\]
and the far-face.
Note that $\monomial{V_1}$ is degenerate as perturbations of $V_1$ yield more facets.
However, its vertex-facet lattice is realisable as the face lattice of the classical polyhedron $\conv(a,b,c) + \RR^3_{\geq 0}$.
This can also be seen as the dual monomial tropical polyhedron is generic. 
  
The second monomial tropical polyhedron $\monomial{V_2}$ has the facets 
\[
uvwe^{(1)} \ , \ uvwe^{(2)} \ , \ uvwe^{(3)} \ , \ ue^{(2)}e^{(3)} \ , \ ve^{(1)}e^{(3)} \ , \ we^{(1)}e^{(2)} \enspace ,
\]
along with the far-face.
$\monomial{V_2}$ is also degenerate, but unlike the previous example its vertex-facet lattice is not realisable by a classical polyhedron.
This is due to the element $uvw$, the intersection of the facets $uvwe^{(1)}, uvwe^{(2)}, uvwe^{(3)}$, being an `edge' containing three vertices.
This example has been identified as problematic from the perspective of monomial ideals and orthogonal surfaces, see \cite[Example 3.8]{EagonMillerOrdog:2019} and \cite[Section 3.1]{Kappes:2006}.
\end{example}
\begin{figure}
\includegraphics[width=0.7\textwidth]{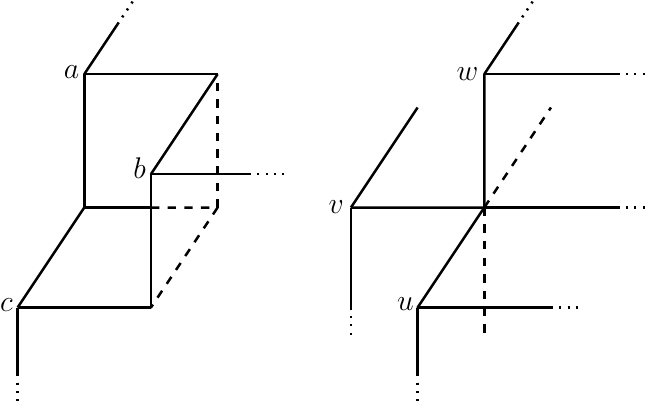}
\caption{Two degenerate monomial tropical polyhedra, $\monomial{V_1}$ and $\monomial{V_2}$ from Example \ref{ex:degenerate+polyhedra}.}
\label{fig:degenerate+polyhedra}
\end{figure}

The natural way to characterise faces of ordinary polyhedra is via minimising linear functions.
Given $c \in \TTmin^d$, consider the tropical linear functional
\begin{equation*}
\begin{split}
\varphi_c : \TTmax^d &\rightarrow \TTmax \\
p &\mapsto \bigoplus_{i \in [d]} -c_i \odot p_i = \max_{i \in [d]}(p_i - c_i) \end{split} \enspace ,
\end{equation*}
We say a vertex $v\in V$ minimises $\varphi_c$ if $\varphi_c(v) \leq \varphi_c(w)$ for all $w \in V$.
Similar to our previous treatment of tropical rays, we say $e^{(i)}$ minimises $\varphi_c$ if $c_i = \infty$.
The following proposition shows every face of the vertex-facet lattice is the minimum of some tropical linear functional.

\begin{proposition} \label{prop:linear+functional+face}
  For each closed set $S \subseteq \overline{V}$ in the affine part of the vertex-facet lattice, there exists some $c \in \TTmin^d$ such that every $v \in S$ minimises $\varphi_c$ and every $w \in \overline{V} \setminus S$ does not minimise $\varphi_c$.
\end{proposition}
\begin{proof}
Let $S$ be a closed set and $T \subseteq \overline{F}$ its corresponding closed set of facets.
Define $c := \min\SetOf{a}{a \in T}$.
For $v \in S \cap V$, we have $v \leq c$ componentwise but $v \geq c$ for at least one component, hence $\max_{i \in [d]}(v - c) = 0$.
For each $w \in V \setminus S$, there is some $a \in T$ such that $w_i > a_i$ for at least one $i \in [d]$.
This implies $\max_{i \in [d]}(w - p) > 0$, and therefore the only elements of $V$ that minimise $\varphi_c$ are those in $S$.
Note that $e^{(i)} \in S$ if and only if all facet-apices of $T$ have $a_i = \infty$, which is equivalent to $e^{(i)}$ minimises $\varphi_c$.
\end{proof}

Unfortunately, the characterisation of faces of ordinary polyhedra does not carry over to monomial tropical polyhedra.
The following example demonstrates that the reverse direction of Proposition~\ref{prop:linear+functional+face} is not correct. 
It extends~\cite[Remark 3.10]{Joswig:2005} by demonstrating that also our combinatorially defined faces do not fulfil this desirable property. 

\begin{figure}
  \begin{tikzpicture}[scale=0.9] 

  \tikzset{Linesegment/.style = {black,thick}}

  
  \coordinate (a) at (1,-4,-1); 
  \coordinate (b) at (2,-3,-2); 
  \coordinate (c) at (2,-2,-3); 
  \coordinate (d) at (1,-1,-4); 

  \node at (a) [label=left:$A$]{};
  \node at (b) [label=above:$B$]{};
  \node at (c) [label=left:$C$]{};
  \node at (d) [label=left:$D$]{};

  \draw[Linesegment] (a) -- (1,-4,-4); 
  \draw[Linesegment] (d) -- (1,-4,-4); 
  \draw[Linesegment] (b) -- (2,-4,-2); 
  \draw[Linesegment] (2,-4,-2) -- (5,-4,-2);
  \draw[Linesegment] (b) -- (2,-3,-3); 
  \draw[Linesegment] (c) -- (2,-3,-3); 
  \draw[Linesegment] (a) -- (5,-4,-1);
  \draw[Linesegment] (a) -- (1,-5,-1);
  \draw[Linesegment] (d) -- (1,-1,-5);
  \draw[Linesegment] (d) -- (5,-1,-4);
  \draw[Linesegment] (b) -- (5,-3,-2);
  \draw[Linesegment] (c) -- (5,-2,-3);
  \draw[Linesegment] (c) -- (2,-2,-4);
  \draw[Linesegment] (2,-2,-4) -- (5,-2,-4);
  \draw[Linesegment] (2,-3,-3) -- (5,-3,-3); 
  \draw[Linesegment, dashed] (1,-4,-4) -- (2,-4,-4);
  \draw[Linesegment, dashed] (2,-2,-4) -- (2,-4,-4);
  \draw[Linesegment, dashed] (2,-4,-2) -- (2,-4,-4);

  \node[circle, fill=red, inner sep=0.7] at (2,-4,-3){}; 
  \draw[Linesegment, dashed, red] (2,-4,-3) -- (2,-1,-3);
  \draw[Linesegment, dashed, red] (2,-4,-3) -- (2,-4,0);
  \draw[Linesegment, dashed, red] (2,-4,-3) -- (-2,-4,-3);

\end{tikzpicture}
  \caption{$\monomial{V}$ from Example \ref{ex:vertices+incident+linear+functional+no+face}. The red point is the apex of the tropical linear functional minimised at $ABC$.}
  \label{fig:not+face+linear+functional}
\end{figure}

\begin{example} \label{ex:vertices+incident+linear+functional+no+face}
Consider the monomial tropical polyhedron $\monomial{V}$ depicted in Figure~\ref{fig:not+face+linear+functional} whose generators are
\[
V \ = \
\begin{blockarray}{cccc}
A & B & C & D \\
\begin{block}{(cccc)}
1 & 2 & 3 & 4 \\ 4 & 3 & 2 & 1 \\ 1 & 2 & 2 & 1 \\
\end{block} \end{blockarray} \enspace .
\]
Its facets are
\[
Ae^{(2)}e^{(3)} \ , \ De^{(1)}e^{(3)} \ , \ ADe^{(1)}e^{(2)} \ , \ ABCD \ , \ CDe^{(3)} \ , \ BCe^{(3)} \ , \ ABe^{(3)} \enspace ,\]
along with the far-face.
Its other (affine) faces are
\[
AB \ , \ AD \ , \ BC \ , \ CD \ , \ Ae^{(3)} \ , \ Ae^{(2)} \ , \ Be^{(3)} \ , \ Ce^{(3)} \ , \ De^{(1)} \ , \ De^{(3)} \enspace .
\]
The tropical linear functional $\max(x_1 - 3,x_2 - 4, x_3 - 2)$ is minimised on the set of vertices $ABC$, however this is not a face in the vertex-facet lattice.
\end{example}

\section{Face posets} \label{sec:face+posets}

A closed monomial tropical polyhedron $\closedmonomial{V}$ is, as a subset of $\TTclosed^d$, equipped with the componentwise partial order of $\TTclosed^d$, which extends the natural partial order of $\RR^d$.
All the following posets are naturally subposets of $\TTclosed^d$ equipped with that order. 

\subsection{Covector decomposition} \label{subsec:covector+decomposition}

In~\cite{JoswigLoho:2016}, the notion of covector graphs introduced in~\cite{DevelinSturmfels:2004} were used to study the combinatorics of point configurations at infinity, extending results from~\cite{FinkRincon:2015}.
We introduce that notion in a way slightly adapted to our purposes.

Given a point $v \in \TTmax^d$, its \emph{$i$th affine sector} is defined as
\begin{align} \label{eq:aff+point+sector}
\begin{split}
S_i(v) &= \left(\bigcap_{k \in [d]} \SetOf{z \in \TTmax^d}{z_i + v_k \leq z_k + v_i}\right) \cap \SetOf{z \in \TTmax^d}{z_i \leq v_i} \\
S_0(v) &= \bigcap_{k \in [d]} \SetOf{z \in \TTmax^d}{v_k \leq z_k}
\end{split} \enspace .
\end{align}
Given a ray $w \in \TTmax^d$, its $i$th affine sector is
\begin{align} \label{eq:aff+ray+sector}
S_i(w) = \bigcap_{k \in [d]} \SetOf{z \in \TTmax^d}{z_i + \added{w_k} \leq z_k + \added{w_i}} \enspace \quad i \in [d] ,
\end{align}
where $S_0(w)$ is empty.
These definitions are compatible with the usual notion of sector, which we discuss further in Section~\ref{sec:representation+by+mono+trop+poly}.

This leads to the notion of \emph{affine covector graph}.
Given a finite set of points and rays $V, W \subseteq \TTmax^d$, we define for $p \in \TTmax^d$ the set
\[
\neighbour_p(v) = \SetOf{i \in [d]_0}{p \in S_i(v)} \qquad \text{ for } v \in V\cup W \enspace .
\]
Recall that the definition of affine sector $S_i(v)$ differs depending on whether $v$ is a point or a ray.
The \emph{covector graph} $G_p$ of $p$ with respect to $V\cup W$ is the bipartite graph on $(V \cup W) \times [d]_0$ with edges $(v,i)$ for $i \in \neighbour_p(v)$.
Note that we will only be considering covectors with respect to monomial tropical polyhedra, and so our set of rays will always be $W = \maxunit$.
These rays describe the boundary strata: a point $p$ is in $S_i(e^{(k)})$ for $i \neq k$ if and only if $p_i = -\infty$.
\added{Note that every point $p \in \TTmax^d$ is contained in $S_i(e^{(i)})$.}

We can define covector graphs for points in $\TTclosed^d$ in the following way.
For any point $p \in \TTclosed^d$, consider $\tilde{p} \in \TTmax^d$ defined as
\[
\tilde{p_i} \ = \
\begin{cases}
p_i & \text{ if } p_i \in \TTmax \\
M & \text{ if } p_i = \infty \\
\end{cases}
\]
for arbitrarily large $M \in \RR$.
We define the covector graph $G_p$ as 
\[
G_p = G_{\tilde{p}} \cup \SetOf{(e^{(i)},0)}{p_i = \infty} \enspace .
\]
The intuition behind this definition is as follows.
\added{Let $\widehat{e}^{(i)} = (-\infty,e^{(i)})$ be the homogenisation of $e^{(i)}$.
In $\TTmax^{d+1}$, the zeroth sector of $\widehat{e}^{(i)}$ is defined to be}
 \[
S_0(\widehat{e}^{(i)}) = \SetOf{z \in \TTmax^{d+1}}{z_0 \leq z_i - \infty}.
\]
In affine space $\TTmax^d$, the zeroth coordinate is always equal to zero, therefore we can formally define $p \in S_0(e^{(i)})$ if and only if $p_i = \infty$.

Covector graphs have a natural poset structure given by containment.
However, we will consider another partial ordering derived from the natural partial order on $\TTclosed^d$.
Our elements will be the \emph{pseudovertices}, the points $p \in \TTclosed^d$ whose covector graphs $G_p$ are connected.
\added{Via the procedure described in~\cite{DevelinSturmfels:2004, JoswigLoho:2017}, one can obtain these points geometrically as follows.
If one replaces each principal halfspace by the corresponding tropical hyperplane, one gets a polyhedral decomposition of $\TTmax^d$ containing $\monomial{V}$ as a subcomplex.
The pseudovertices are precisely the zero-dimensional cells of this decomposition.
While not all of them are vertices of $\monomial{V}$ in the sense of tropical convexity, they are vertices of a polyhedral complex whose support is $\monomial{V}$.}
We also include the unique minimal element $\bm{-\infty} = (-\infty,\dots,-\infty)$; note that the unique maximal element is the pseudovertex $\bm{\infty} = (\infty,\dots,\infty)$.
We let $\cP$ be the set of pseudovertices including $\bm{-\infty}$ with partial order given by the standard partial order $\leq$ on $\RR^d$ extended to $\TTclosed^d$.
We call $(\cP,\leq)$ the \emph{pseudovertex poset}.

Note that $\cP$ is not a lattice as the following example demonstrates.
\begin{figure}
\includegraphics[width=\textwidth]{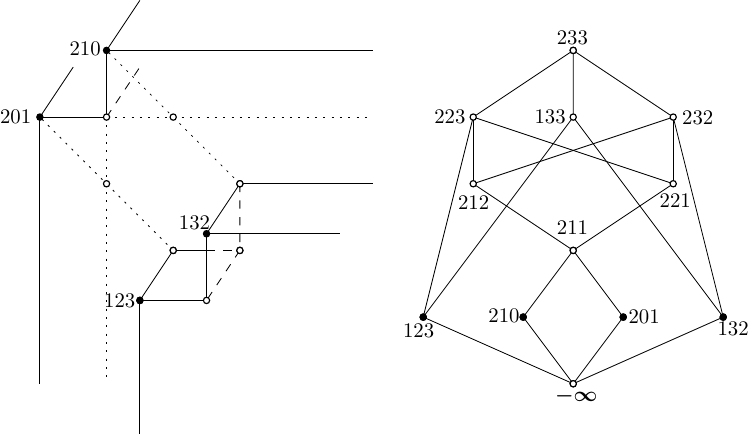}
\caption{The monomial tropical polyhedron described in Example \ref{ex:pseudovertex+non+lattice} and the interval $[\bm{-\infty},(2,3,3)]$ in its pseudovertex poset.}
\label{fig:pseudovertex+non+lattice}
\end{figure}
\begin{example} \label{ex:pseudovertex+non+lattice}
Let $\monomial{V}$ be the monomial tropical polyhedron generated by
\[
V = \{(2,1,0),(2,0,1),(1,2,3),(1,3,2)\} \enspace .
\]
Figure~\ref{fig:pseudovertex+non+lattice} shows $\monomial{V}$ and an interval in its pseudovertex poset.
In particular, elements $(2,2,1)$ and $(2,1,2)$ have no unique least upper bound.
\end{example}

\begin{remark} \label{rem:connection+products+simplices}
  The duality between tropical point configurations and subdivisions of products of simplices established in~\cite{DevelinSturmfels:2004} emphasises a different viewpoint on the partial orderings of covector graphs.
  The first `natural' choice of containment of covector graphs translates to a reversed order for the corresponding cells in the subdivision of a product of two simplices.
  However, the interpretation of the partial order on the pseudovertices does not have such a clear geometric analogue.
  Note that a combinatorial abstraction of the ordering derived from the graph structure is used to measure progress in abstract tropical linear programming~\cite{Loho:2020}. 
\end{remark}

While a natural first choice of poset to describe the `face structure' of $\monomial{V}$, Example \ref{ex:pseudovertex+non+lattice} highlights that the pseudovertex poset is not the correct choice.
Additionally, covector graphs encode a lot of information that is not relevant as the structure of $\monomial{V}$ is mostly encoded via the zeroth sectors and their interaction with other sectors.
We highlight this in Figure~\ref{fig:covectors+too+fine}. 
This motivates the study of a coarser poset that only encodes this information.

\begin{figure}
\includegraphics[width=0.34\textwidth]{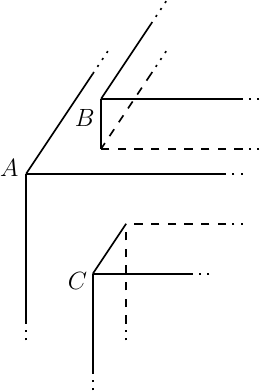}
\caption{A monomial tropical polyhedron with generators $A,B,C$. Vertices $B$ and $C$ can move freely without changing the vertex-facet lattice, but may drastically change the underlying covector graphs.}
  \label{fig:covectors+too+fine}
\end{figure}

\subsection{Max-lattice} \label{subsec:lcm+lattice}

We start with a lattice that is significantly simpler than the pseudovertex poset but still fine enough to capture all other posets in this section.

In the following we introduce the \emph{$\max$-lattice of the monomial tropical polyhedron $\monomial{V}$}. 
We define the \emph{modified $\max$-tropical unit vectors} $\modmaxunit = \{\bar{e}^{(1)}, \ldots, \bar{e}^{(d)}\} \subset \TTclosed^{d}$ as
\[
\bar{e}^{(i)}_{k} \ = \
\begin{cases}
  +\infty  & \mbox{ if }  i = k \\
  -\infty & \mbox{ if } i \in [d] \setminus k
\end{cases}
\enspace .
\]
For each $S \subseteq V \cup \modmaxunit$ such that $S \cap V \neq \emptyset$, we define the elements
\[
m_S = \max\SetOf{v}{v \in S}
\]
where $\max$ is taken componentwise.
The $\max$-lattice is the set of elements of this form, along with the unique minimal element $m_{\emptyset} = \bm{-\infty}$ ordered by the standard partial order on $\TTclosed^d$.
An element $m$ of the $\max$-lattice is naturally labelled by the maximal inclusionwise set $S$ such that $m = m_S$.
We note that the unique maximal element is $m_{V \cup \modmaxunit} = \bm{\infty}$.
\added{This forms a lattice as each pair of elements $m_S, m_T$ have a unique join and meet defined as
\begin{align*}
m_S \vee m_T &= \max\SetOf{v}{v \in S\cup T} = m_{S \cup T} \, , \\
m_S \wedge m_T &= \max\SetOf{v}{v \in S\cap T} = m_{S \cap T} \, .
\end{align*}}

There is geometric intuition behind the definition of the $\max$-lattice.
Let $X$ be a subset of $\TTclosed^d$, the \emph{$\max$-tropical barycenter} of $X$ is
\begin{equation} \label{eq:tropical+barycenter}
\tbary(X) = \sup\SetOf{x}{x \in X} \enspace ,
\end{equation}
where $\sup$ is taken componentwise.
This definition is obtained by directly tropicalising the classical notion of barycenter, and appears as an important tool in~\cite{ABGJ:2018}. 
Note that we must work in $\TTclosed^d$, else the tropical barycenter may not be well-defined.

Given $S \subseteq V \cup \modmaxunit$, we define the tropical polyhedron:
\[
P_S = \tconv{v\ |\ v \in S \cap V} \oplus \tcone{e^{(k)}\ |\ \bar{e}^{(k)} \in S \cap \modmaxunit} \enspace .
\]
If $S$ has the additional property that $S \cap V \neq \emptyset$, then $S$ is a tropical subpolyhedron of $\monomial{V}$.
The following lemma shows the max-lattice contains geometric information of $\monomial{V}$ in terms of tropical barycenters.
\begin{lemma} \label{lem:tropical+barycenter}
Let $S \subseteq V \cup \modmaxunit$ with $S \cap V \neq \emptyset$. Then
\[
m_S = \tbary(P_S) \enspace .
\]
\end{lemma}
\begin{proof}
The `largest' points of $P_S$ are obtained by the representation
\[
\left(\bigoplus 0 \odot v \right) \oplus \left(\bigoplus \lambda_k \odot e^{(k)} \right) \quad v \in S \cap V \ , \ \bar{e}^{(k)} \in S \cap \modmaxunit 
\]
where $\lambda_k$ are arbitrarily large.
Taking the supremum of these points gives the desired result.
\end{proof}

\begin{example} \label{ex:max+lattice}
Consider the monomial tropical polyhedron $\monomial{V}$ generated by $V = \{(1,2),(2,-\infty)\}$ and its $\max$-lattice shown in Figure \ref{fig:max+lattice}.
All subsets of $V \cup \maxunit$ whose intersection with $V$ is non-empty have a corresponding point in $\closedmonomial{V}$, in particular the vertices and facet-apices are contained.
\end{example}
\begin{figure}
\includegraphics[width=0.49\textwidth]{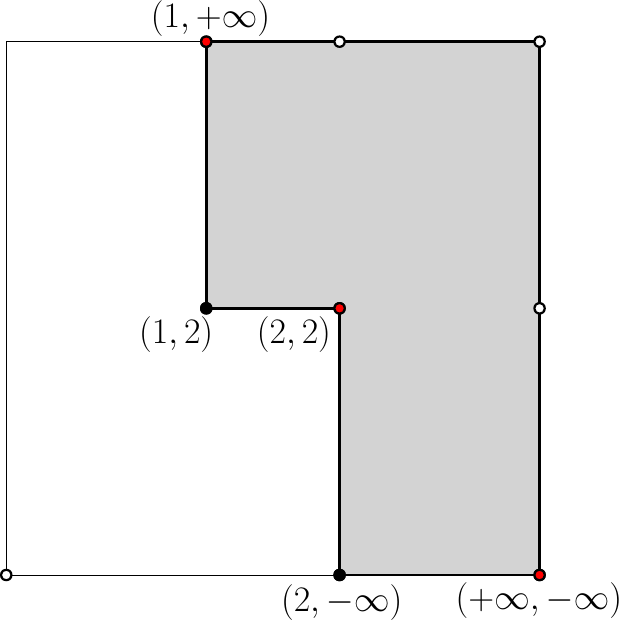}
\resizebox{0.49\textwidth}{!}{\begin{tikzpicture}
\node (0) at (0,-1.5){$(-\infty,-\infty)$};
\node (a) at (-1,0){$(1,2)$};
\node (b) at (1,0){$(2,-\infty)$};

\node (e2a) at (-2,1.5){$(1,+\infty)$};
\node (ab) at (0,1.5){$(2,2)$};
\node (e1b) at (2,1.5){$(+\infty,-\infty)$};

\node (e1ab) at (1,3){$(+\infty,2)$};
\node (e2ab) at (-1,3){$(2,+\infty)$};

\node (e1e2ab) at (0,4.5){$(+\infty,+\infty)$};

\draw (0) -- (a);
\draw (0) -- (b);
\draw (a) -- (ab);
\draw (a) -- (e2a);
\draw (b) -- (ab);
\draw (b) -- (e1b);
\draw (ab)--(e1ab);
\draw (ab)--(e2ab);
\draw (e1b)--(e1ab);
\draw (e2a)--(e2ab);
\draw (e1ab)--(e1e2ab);
\draw (e2ab)--(e1e2ab);

\end{tikzpicture}}
\caption{The monomial tropical polyhedron $\monomial{V}$ and its $\max$-lattice defined in Example \ref{ex:max+lattice}. Vertices of $\monomial{V}$ are marked \added{black}, facet-apices are marked \added{red} and other elements are marked white.}
\label{fig:max+lattice}
\end{figure}

The following two propositions describe the relationship between the $\max$-lattice, the vertex-facet lattice and the pseudovertex poset.
\begin{proposition}
The affine part of the vertex-facet lattice of $\monomial{V}$ is a sublattice of its $\max$-lattice.
\end{proposition}
\begin{proof}
Given a closed subset $S$ of $V \cup \maxunit$, the corresponding subset $\overline{S}$ of $V \cup \modmaxunit$ is defined by simply replacing $e^{(k)}$ with $\bar{e}^{(k)}$.
We claim that the map from the vertex-facet lattice to the $\max$-lattice defined by $S \mapsto m_{\overline{S}}$ is an order embedding.

Let $S_1,S_2$ be closed subsets of $V \cup \maxunit$ such that $S_i \cap V \neq \emptyset$ for $i = 1,2$.
It suffices to show $S_1 \subseteq S_2$ if and only if $m_{\overline{S}_1} \leq m_{\overline{S}_2}$.
If $S_1 \subseteq S_2$, it is immediate that $\overline{S}_1 \subseteq \overline{S}_2$ and therefore $m_{\overline{S}_1} \leq m_{\overline{S}_2}$.

Conversely, suppose $S_1 \nsubseteq S_2$.
By closedness, there exists an apex $a \in F$ such that all elements of $S_2$ are incident with $a$, but there exists some $w \in S_1 \setminus S_2$ that is not incident with $a$.
There exists some coordinate $i \in [d]$ such that
\begin{align*}
v_i \leq a_i < w_i  \ &\forall v \in S_2 \cap V \quad \text{ if } w \in V \\
v_i \leq a_i < \infty  \ &\forall v \in S_2 \cap V \quad \text{ if } w = e^{(i)}
\end{align*}
In either case, we have $m_{\{\overline{w}\}} \nleq m_{\overline{S}_2}$ therefore $m_{\overline{S}_1} \nleq m_{\overline{S}_2}$.
\end{proof}

\goodbreak

\begin{proposition}
The $max$-lattice is a subposet of the pseudovertex poset.
\end{proposition}
\begin{proof}
  Let $m = m_S$ be an element of the $\max$-lattice, we claim that $m$ is a pseudovertex as a point in $\TTclosed^d$.
Let $G_m$ be the covector graph associated to $m$, we show that $G_m$ is connected.

By Lemma \ref{lem:tropical+barycenter}, $m$ is the tropical barycenter of the tropical polyhedron $P_S$.
If $m_i$ is finite there exists some vertex $w$ such that $w_i = m_i$, therefore $0, i \in \neighbour_m(w)$.
If $m_i = \infty$ then $0, i \in \neighbour_m(e^{(i)})$, as every element is contained in $S_i(e^{(i)})$ and $m \in S_0(e^{(i)})$ if and only if $m_i = \infty$.
This implies there exists a connected component in $G_m$ containing every node of $[d]_0$.
Furthermore, no node of $V \cup \maxunit$ can be isolated in $G_m$, therefore $G_m$ is connected.
\end{proof}

\begin{corollary} \label{coro:characterization+covector+max+lattice}
A pseudovertex $p$ with covector graph $G_p$ is an element of the $\max$-lattice if and only if $\neighbour_p(\neighbour_p(0)) = [d]_0$. 
\end{corollary}
\begin{proof}
The previous proof shows elements of the $\max$-lattice satisfy this condition.
Conversely, let $S = \neighbour_p(0)$.
For all $i \in [d]$ there exists $w \in S$ such that $w_i = p_i$ or $p_i = \infty$ and $w = e^{(i)}$.
This implies $p = m_{\overline{S}}$.
\end{proof}

As the covector graph of an element of the $\max$-lattice only depends on the neighbourhood of $0$, we get the following.

\begin{corollary} \label{coro:max+lattice+dependence+order}
  The $\max$-lattice is purely determined by the componentwise order of the points and does not depend on the actual coordinates. 
\end{corollary}

\begin{remark} \label{rem:max+lattice+covector+comparison}
Example~\ref{ex:pseudovertex+non+lattice} and Corollary~\ref{coro:characterization+covector+max+lattice} demonstrate that the pseudovertex poset is in general not equal to the max-lattice.
 Furthermore, unlike the $\max$-lattice the pseudovertex poset does not only depend on the ordering of the coordinates.
This can be seen via the example in Figure \ref{fig:pseudovertex+non+lattice} by perturbing the vertex $132$ to $142$.
This does not change the ordering of the coordinates, however the vertex $210$ splits into $210$ and the new pseudovertex $220$.
\end{remark}

\subsection{Max-min poset} \label{subsec:max+min+poset}

We introduce a new poset which is motivated by the duality of monomial tropical polyhedra.
However, we show that its unexpected behaviour reflects the discrepancy of the tropical convex hull of the vertices and the intersection of the halfspaces corresponding to a face in the vertex-facet lattice.
It is this behaviour that causes a major problem for the face lattice defined by Joswig in~\cite{Joswig:2005}. 

Dual to the max-lattice, we introduce the min-lattice of a monomial tropical polyhedron. 
For this, we set $\modminunit = -\modmaxunit$ and let $A$ be the set of principal apices of $\monomial{V}$.
We mirror the construction of the $\max$-lattice from $\TTmax$ to $\TTmin$.
We replace $\max$ with $\min$, the ground set $V\cup\modmaxunit$ with $A\cup\modminunit$ and for each $T \subseteq A \cup \modminunit$ such that $T \cap A \neq \emptyset$, we define the elements
\[
n_T = \min\SetOf{a}{a \in T} \enspace .
\]
By reversing the partial order on $\TTclosed^d$ and letting $\bm{\infty}$ be the unique minimal element in this partial order, we obtain the \emph{$\min$-lattice} of $\monomial{V}$.
As a direct consequence of Theorem \ref{thm:complementary-cones}, we have the following relation between the $\min$ and the $\max$-lattice. 
\begin{lemma}
  \added{There is an order-reversing isomorphism between the $\min$-lattice of $\monomial{V}$ and the $\max$-lattice of $\monomial{-A} = -\complementarymonomial{V}$}. 
\end{lemma}
\begin{proof}
\added{There is a bijection between the two posets given by
\[
n_T = \min\SetOf{a}{a \in T} = -\max\SetOf{-a}{a \in T} = -m_T \, .
\]
Furthermore, this bijection is order reversing as $n_S \leq n_T \Leftrightarrow -m_S \geq -m_T$.}
\end{proof}
\added{The consequence of this lemma is that the $\min$-lattice is a reformulation of the $\max$-lattice of the complementary monomial cone, and vice versa.
As a result,} both the $\max$-lattice and the $\min$-lattice describe the face structure of the intersection $\monomial{V} \cap \complementarymonomial{V}$.
Furthermore, both have associated geometric data in the form of the max-barycenter $m_S$ and the min-barycenter $n_T$, respectively.
A set of facet-apices $T \subseteq F$ is also naturally a subset of $A \cup \modminunit$ by sending any boundary apices to the corresponding element in $\modminunit$.
This gives a natural embedding of the affine part of the vertex-facet lattice $\cV$ into the max-lattice $\cL_{\max}$ and the min-lattice $\cL_{\min}$ via the maps
\begin{align*}
\cV &\hookrightarrow \cL_{\max} &\quad \cV &\hookrightarrow \cL_{\min} \\
S &\mapsto m_S &\quad T &\mapsto n_T
\end{align*}
where $S$ is a closed set of vertices and $T$ is a closed set of facet-apices.

A natural question is whether the geometric data agrees on the affine part of the vertex-facet lattice.
Right from Definition~\ref{def:vertex+apex+incidence}, we obtain a weaker statement.

\begin{lemma}
Let $S \subseteq \overline{V}$ be a closed set of vertices with $S \cap V \neq \emptyset$ whose element in the $\max$-lattice is $m_S$, and let $T \subseteq \overline{F}$ its corresponding closed set of facets whose element in the $\min$-lattice is $n_T$.
Then $m_S \leq n_T$.
\end{lemma}
However, equality is not always attained as the following example demonstrates.
\begin{example} \label{ex:max+min+subposet+vf}
\added{Consider the monomial tropical polyhedron $\monomial{V}$ displayed in Figure~\ref{fig:min+max+counterexample} with generating set
\[
V \ = \
\begin{blockarray}{ccccc}
A & B & C & D & E \\
\begin{block}{(ccccc)}
1 & 1 & 1 & 0 & 0 \\
0 & 1 & 2 & 3 & 4 \\
2 & 1 & 0 & 4 & 3  \\
\end{block}
\end{blockarray}
\enspace .
\]
Its set of facet-apices is
\[
F \ = \
\begin{blockarray}{cccccccc}
s & t & u & v & w & x & y & z \\
\begin{block}{(cccccccc)}
1 & 1 & 1 & \infty & \infty & 0 & \infty & \infty \\
4 & 3 & \infty & 1 & 2 & \infty & 0 & \infty \\
4 & \infty & 3 & 2 & 1 & \infty & \infty & 0 \\
\end{block}
\end{blockarray}
\]
}
\begin{figure}
\includegraphics[width = 0.6\textwidth]{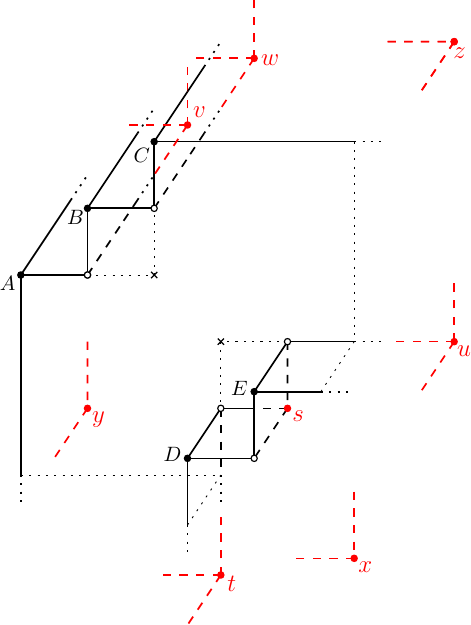}
\caption{The monomial tropical polyhedron $\monomial{V}$ given in Examples~\ref{ex:max+min+subposet+vf} and~\ref{ex:CP+nonlattice}.}
\label{fig:min+max+counterexample}
\end{figure}
\added{Consider the subset of vertices $\{A,B,C\}$; this is a closed set and its corresponding  closed set of apices is $\{s,t,u\}$.
The point in the $\max$-lattice corresponding to $\{A,B,C\}$ is $(1,2,2)$, while the point in the $\min$-lattice corresponding to $\{s,t,u\}$ is $(1,3,3)$.
These points are both marked on Figure~\ref{fig:min+max+counterexample} by crosses.}
\end{example}

It is desirable to have a face poset of a monomial tropical polyhedron that reflects the inherent duality of monomial tropical polyhedra which can be seen in Theorem~\ref{thm:complementary-cones} and in its translation to Alexander duality explained in Section~\ref{prop:alexander+duality}.
This motivates the following definition.
\begin{definition}
The \emph{max-min poset} $\cM$ is the induced subposet of $\cL_{\max}$ on the elements in the intersection $\cL_{\max} \cap \cL_{\min}$.
\end{definition}
We remark that it is equally reasonable to define $\cM$ as an induced subposet of $\cL_{\min}$, this will simply reverse the partial order.

\begin{proposition}
The max-min poset is a subposet of the affine part of the vertex-facet lattice.
\end{proposition}
\begin{proof}
It suffices to show if $m_S = n_T$ for some maximal subsets of vertices and facet-apices $S$ and $T$, then $S$ and $T$ are closed.
There must exist some maximal set, as if $m_S = m_{S'}$, then we can take their union $S \cup S'$ without increasing the max.
Consider $w \leq a$ for all $a \in T$, then $w \leq n_T = m_S$ and so $w \in S$ by maximality.
This implies $S$ is closed, $T$ is closed by a similar argument.
\end{proof}
However, note that Example~\ref{ex:max+min+subposet+vf} implies the max-min poset may be a strict subposet of the vertex-facet lattice.
We shall see in Section \ref{subsec:CP+order} that the max-min poset is not a lattice in these cases.
Furthermore, the following example highlights that the max-min poset is not recoverable from \added{certain combinatorial data of the monomial tropical polyhedron, such as its vertex-facet incidence graph and the covector graphs of points.}

\begin{example} \label{ex:max+min+noncovector}
\added{
We consider the following variant of the monomial tropical polyhedron from Example~\ref{ex:max+min+subposet+vf}.
Explicitly, we consider the generating set $\tilde{V}$ obtained from $V$ after replacing $D$ and $E$ with $\tilde{D} = (0,2,3)$ and $\tilde{E} = (0,3,2)$.
This affects the three facet-apices $s,t,u$: we replace them with 
\[
\tilde{s} = (1,3,3) \quad , \quad \tilde{t} = (1,2,\infty) \quad , \quad \tilde{u} = (1,\infty, 2) \, .
\]
We now have $\max(A,B,C) = \min(\tilde{s},\tilde{t},\tilde{u}) = (1,2,2)$, and so $\{A,B,C\}$ is an element of the max-min poset of $\monomial{\tilde{V}}$.

We note that $\monomial{V}$ and $\monomial{\tilde{V}}$ have isomorphic vertex-facet lattices as these replacements do not affect the vertex-facet incidences.
We also note that the covectors for $ABC$ in $\monomial{V}$ and $\monomial{\tilde{V}}$ are equal up to relabelling $D,E$ by $\tilde{D},\tilde{E}$.
This implies that one cannot determine if a point is contained in the max-min poset from its covector, or from the vertex-facet incidence graph.
}
\end{example}

While we cannot give a characterisation of the covector graph of a point in the max-min poset, we conclude with a characterisation of the covector graphs of principal apices, the coatoms of the max-min poset, which is essentially a reformulation of Proposition~\ref{prop:facetapex+characterisation}. 

\begin{lemma} \label{lem:characterization+covectors+facet+apices}
  A point $p \in \TTmin^d$ is an apex of a principal halfspace of $\monomial{V}$ if and only if
  \begin{enumerate}[label=\alph*)]
  \item $\neighbour_p(0)$ contains no node of degree $1$,
  \item for each $i \in [d]$, there is at least one element \added{$v \in \overline{V}$ such that $\neighbour_p(v) = \{0, i\}$.}
  \end{enumerate}
\end{lemma}

\subsection{CP-orders} \label{subsec:CP+order}

Felsner and Kappes \cite{FelsnerKappes:2008} introduced another poset associated to a monomial tropical polyhedron called the \emph{CP-order}.
Their work was motivated by the study of the order dimension of a poset, see in particular~\cite[Prop.~2]{FelsnerKappes:2008}.

We use a definition of CP-order tailored to our terminology by adapting combinatorial notions introduced in Kappes' thesis~\cite{Kappes:2006}.
We say $q$ is \emph{almost strictly greater than} $p$ if there exists some coordinate $i$ such that $p_i = q_i$ and $p_k < q_k$ for all $k \neq i$, denoted $p \triangleleft_i q$. 
An \emph{$i$-witness} for a point $p \in \partial\monomial{V}$ is a vertex $v \in V$ such that there exists $q \in \partial\monomial{V}$ satisfying $v \leq p \leq q$ and $v \triangleleft_i q$.
Note that this does not extend immediately to $\TTclosed^d$: any point $p$ with $p_i = \infty$ cannot have an $i$-witness as vertices are elements of $\TTmax^d$.
In this case, one can think of this as $e^{(i)}$ acting as an $i$-witness.
This motivates the following definition for characteristic points in $\TTclosed^d$.
\begin{definition}
A point $p$ is a \emph{characteristic point} if for each $i \in [d]$, either $p$ has an $i$-witness or $p_i = \infty$.
The \emph{CP-order} is the poset of characteristic points, along with the unique minimal element $\bm{-\infty}$, with the standard partial order.
\end{definition}
\added{As an example, we note that the vertices and facet-apices of $\monomial{V}$ are characteristic points.
A vertex $v \in V$ is its own $i$-witness with respect to $q$ where $q_i = v_i$ and $q_j = v_j + \varepsilon$ for all $j \neq i$.
By condition~\eqref{eq:apex+cond+2} of Proposition~\ref{prop:facetapex+characterisation}, a facet-apex $a$ must either have $a_i = \infty$ or some vertex $v$ be an $i$-witness.}

Characteristic points have a geometric interpretation that can informally be thought of as the ``corners" of $\monomial{V}$ where $d$ components of the boundary intersect, each parallel to a different coordinate hyperplane.
While intuitive, the geometric definition is subtle and much more cumbersome than the clean combinatorial formulation we use.

The CP-order is relatively hard to deal with. 
For example, it may not be a lattice as Example \ref{ex:CP+nonlattice} demonstrates. 
Furthermore, it appears to be computationally unwieldy: the best known approach is naively compute a poset containing it and manually check if each element satisfies the condition of being a characteristic point. 
However, the following results show the CP-order is intimately related to the computationally amenable face posets we have already discussed.
\begin{lemma} \label{lem:CP+covector}
A point $p \in \TTclosed^d$ with covector graph $G_p$ is a characteristic point if and only if for each $i \in [d]$ there exists $v \in \neighbour_p(0) \cap \neighbour_p(i)$ such that $\neighbour_p(u) \nsubseteq \neighbour_p(v)$ for all $u \in \neighbour_p(0) \setminus \neighbour_p(i)$.
\end{lemma}
\begin{proof}
It suffices to show this condition on $v$ is equivalent to being an $i$-witness to $p$.
The condition $v \in \neighbour_p(0) \cap \neighbour_p(i)$ ensures that $v \leq p$ and $v_i = p_i$.
The condition on $u$ is a translation of \cite[Proposition 4.15]{Kappes:2006}.
Paired with the first condition, this is equivalent to $v$ being an $i$-witness.
\end{proof}
\begin{proposition}
  The CP-order is a subposet of the max-min poset, and hence of the $\max$-lattice and the pseudovertex poset.
\end{proposition}
\begin{proof}
Given some characteristic point $p$, we construct subsets $S \subseteq \overline{V}$ and $T \subseteq \overline{F}$ such that their corresponding elements in the $\max$-lattice and $\min$-lattice respectively equal $p$.
As $p$ is a characteristic point, for each $i \in [d]$ either it has an $i$-witness $v^{(i)}$ or $p_i = \infty$.
In the former case, there exists some apex $a^{(i)}$ such that $v^{(i)} \leq p \leq a^{(i)}$ with equality in the $i$-th coordinate, therefore we add $v^{(i)}$ to $S$ and $a^{(i)}$ to $T$.
In the latter case, we add $e^{(i)}$ to $S$.
The elements corresponding to $S$ in the $\max$-lattice and $T$ in the $\min$-lattice gives the desired result.
\end{proof}

\begin{proposition} \label{prop:CP+subposet+VIF}
The CP-order is a subposet of the affine part of the vertex-facet lattice.
\end{proposition}
\begin{proof}
Let $p \in \TTclosed^d \setminus \bm{\infty}$ be a characteristic point of $\closedmonomial{V}$ and define
\begin{align*}
S &= \SetOf{v \in \overline{V}}{v \text{ is incident to } p} \subseteq \overline{V} \enspace , \\
T &= \SetOf{a \in F}{p \text{ is incident to } a} \subseteq \overline{F} \enspace .
\end{align*}
There must exist some $i$-witness for $p$, therefore we have $S \cap V \neq \emptyset$ immediately.
We claim that $S$ and $T$ are closed sets.

First we show that every element of $S$ is incident to every element of $T$.
Fix some vertex $v \in S$ and consider a facet-apex $a \in T$, we have $v \leq p \leq a$.
Furthermore, by the definition of facet-apex there exists some coordinate such that $v_i = a_i$, therefore $v$ in incident to every facet-apex of $T$.
Now fix some ray $e^{(i)} \in S$.
For any facet-apex $a \in T$ we have $a_i \geq p_i = \infty$ therefore $a$ is incident to $e^{(i)}$.

Suppose there exists $w \in \overline{V} \setminus S$ incident to all facet-apices $a \in T$.
Suppose $w = e^{(i)}$, then every facet-apex has $a_i = \infty$.
There does not exist an $i$-witness for $p$, as this would imply there exists a vertex $v \in \TTmax^d$ such that $v_i = a_i = \infty$.
By definition of characteristic points, this implies $p_i = \infty$, and therefore $e^{(i)} \in S$.

Instead suppose $w \in V$.
As $w \notin S$, there exists $i \in [d]$ such that $w_i > p_i$.
There exists an $i$-witness $v \in S$ for $p$ and some $a \in T$ such that $v_i = p_i = a_i$ and $v_k < a_k$ for all $k \neq i$.
This implies $w_i > a_i$, contradicting that $w$ is incident to $a$.

We now prove the equivalent statement for $T$.
Suppose there exists some facet-apex $a \in F \setminus T$ incident to all elements of $S$.
As $a \notin T$, there exists some $i \in [d]$ such that $p_i > a_i$.
By definition of characteristic point, $p$ has an $i$-witness $v \in S$, implying $v_i = p_i > a_i$, and so $a$ is not incident to all elements of $S$.
This completes the proof that $S$ and $T$ are corresponding closed sets.

We note as the CP-order is a subposet of the max-lattice, no two characteristic points can have the same $S$.
Therefore every characteristic point is an element of the vertex-facet lattice.
The final points to consider are the unique maximal point $\bm{\infty}$ and the unique minimal element $\bm{-\infty}$.
To each of these we have the natural embedding
\begin{align*}
\bm{\infty} \mapsto \overline{V} \quad \bm{-\infty} \mapsto \emptyset \enspace ,
\end{align*}
and there corresponding closed subsets of $\overline{F}$ of facets.
As the partial order on both is domination, the CP-order is a subposet of the vertex-facet \added{lattice}.
\end{proof}

\begin{example} \label{ex:CP+nonlattice}
We show the CP-order is distinct from the vertex-facet lattice and max-min poset by revisiting Examples~\ref{ex:max+min+subposet+vf} and~\ref{ex:max+min+noncovector}.

\added{Consider the monomial tropical polyhedron introduced in Example~\ref{ex:max+min+subposet+vf}, and the interval $[B,ABCDE]$ of its vertex-facet lattice displayed in Figure~\ref{fig:CP+nonlattice}.
For each element, the max of their vertices is a characteristic point other than $ABC$: the point $\max(A,B,C)$ is marked with a cross in Figure~\ref{fig:min+max+counterexample}, and is clearly not a characteristic point.
Therefore the corresponding interval in the CP-order is not a lattice as $AB$ and $BC$ have both $ABCD$ and $ABCE$ as minimal upper bounds, as shown in Figure~\ref{fig:CP+nonlattice}.
The corresponding interval in the max-min poset coincides with the CP-order, hence it differs from the vertex-facet lattice.}
\begin{figure}
\begin{tikzpicture}
\node (b) at (-3,0) {$B$};
\node (ab) at (-4,1) {$AB$};
\node (bc) at (-2,1) {$BC$};
\node (abc) at (-3,2) {$ABC$};
\node (abcd) at (-4,3) {$ABCD$};
\node (abce) at (-2,3) {$ABCE$};
\node (abcde) at (-3,4) {$ABCDE$};

\draw (b) -- (ab);
\draw (b) -- (bc);
\draw (ab) -- (abc);
\draw (bc) -- (abc);
\draw (abc) -- (abcd);
\draw (abc) -- (abce);
\draw (abce) -- (abcde);
\draw (abcd) -- (abcde);

\node (b2) at (3,0) {$B$};
\node (ab2) at (2,1) {$AB$};
\node (bc2) at (4,1) {$BC$};
\node (abcd2) at (2,3) {$ABCD$};
\node (abce2) at (4,3) {$ABCE$};
\node (abcde2) at (3,4) {$ABCDE$};

\draw (b2) -- (ab2);
\draw (b2) -- (bc2);
\draw (ab2) -- (abcd2);
\draw (bc2) -- (abce2);
\draw (bc2) -- (abcd2);
\draw (ab2) -- (abce2);
\draw (abce2) -- (abcde2);
\draw (abcd2) -- (abcde2);
\end{tikzpicture}
\caption{The interval \added{$[B,ABCDE]$} in the vertex-facet lattice (left) and the CP-order (right) \added{of the monomial tropical polyhedron studied in Examples~\ref{ex:max+min+subposet+vf} and~\ref{ex:CP+nonlattice}.}}
\label{fig:CP+nonlattice}
\end{figure}

\added{Now consider the monomial tropical polyhedron from Example~\ref{ex:max+min+noncovector}.
The point $\max(A,B,C) = (1,2,2)$ is now contained in the max-min poset.
However, it is still not a characteristic point as it has no 2-witness or 3-witness. 
This shows the CP-order may be a strict poset of the max-min poset.}
\end{example}

The previous results in this section demonstrate that there are many lattices that the CP-order may embed into, but we would like to find the smallest such lattice.
This motivates the following construction.
Let $P$ be a partially ordered set, for some subset of elements $A \subseteq P$, we denote
\begin{align*}
A^{\uparrow} &= \SetOf{p \in P}{p \geq q \ , \ \forall q \in A} \enspace , \\
A^{\downarrow} &= \SetOf{p \in P}{p \leq q \ , \ \forall q \in A} \enspace .
\end{align*}
\begin{definition}[\cite{DaveyPriestley:2002}]
Given a partially ordered set $P$, its \emph{Dedekind-MacNeille completion} is the complete lattice of \added{subsets}
\[
DM(P) = \SetOf{A \subseteq P}{(A^{\uparrow})^{\downarrow} = A}
\]
ordered by inclusion.
\end{definition}
$P$ embeds into $DM(P)$ via the map $p \mapsto p^{\downarrow}$. %
Moreover, $DM(P)$ is the smallest complete lattice that $P$ embeds into: any embedding $P \hookrightarrow L$ into a complete lattice $L$ induces an embedding $DM(P) \hookrightarrow L$ \cite[Theorem 8.27]{Schroder:2016}.
Note that all of our lattices are finite, therefore completeness comes for free.

\begin{theorem} \label{thm:CP+completion+VIF}
The affine part of the vertex-facet lattice is the Dedekind-MacNeille completion of the CP-order.
\end{theorem}
\begin{proof}
\added{Given a poset $P$, a subset $Q \subseteq P$ is \emph{join-dense} if for every $p \in P$, there exists a subset $A \subseteq Q$ such that $P$ is the join of the elements of $A$ in $P$, i.e. $p = \bigvee_{a \in A} a$; we define \emph{meet-dense} analogously with $\wedge$.
By~\cite[Theorem 7.41]{DaveyPriestley:2002}, if a poset $P$ is join and meet-dense in a complete lattice $L$, then $L \simeq DM(P)$.
Therefore it suffices to show that the CP-order is join and meet-dense in the affine part of the vertex-facet lattice.

Let $\cC$ be the CP-order of $\monomial{V}$ and $L$ the affine part of the vertex-facet lattice.
Proposition~\ref{prop:CP+subposet+VIF} shows that $\cC$ is a subposet of $L$, and so we can label the characteristic points by their corresponding closed subsets of either $\overline{V}$ or $F$.
To show that $\cC$ is meet-dense, recall that every facet-apex is a characteristic point, and therefore any closed set $T \subseteq F$ in $L$ can be realised as the meet of the characteristic points $\{a \mid a \in T\}$.

Showing that $\cC$ is join-dense is more involved as while the vertices $V$ are characteristic points, the rays $e^{(i)}$ are not.
Let $S \subseteq \overline{V}$ be a closed set in $L$ and $T \subseteq F$ its corresponding closed set of facet-apices.
For each $e^{(i)} \in S$, pick $v \in S\cap V$ such that $v_i \geq w_i$ for all $w \in S\cap V$.
Furthermore, define $v^{(i)}$ where $v_i^{(i)} = \infty$ and $v_k^{(i)} = v_k$ for all $k \neq i$.
We claim that $v^{(i)}$ is a characteristic point.

We show that for each $k \neq i$, the vertex $v$ is a $k$-witness for $v^{(i)}$, i.e. there exists some $q \in \partial\monomial{V}$ such that $v \triangleleft_k q$ where $v \leq v^{(i)} \leq q$.
We set $q$ to be
\[
q_j = \begin{cases} \infty & j = i \\ v_k & j = k \\ v_j + \varepsilon & j \neq i,k \\ \end{cases} \, ,
\]
for some small $\varepsilon > 0$.
This satisfies the necessary inequalities of a $k$-witness, it just remains to show that $q \in \partial\monomial{V}$ for sufficiently small $\varepsilon$.
Suppose $q$ is not in the boundary for any choice of $\varepsilon >0$, then there exists $w \in V$ such that $w < q$.
Letting $\varepsilon \rightarrow 0$, we get that $w_j \leq v_j$ for all $j \neq i$, implying that $w_i > v_i$; otherwise $v$ dominates $w$ and so cannot be a minimal generator.
For each facet-apex $a \in T$, we have $\infty = a_i > w_i$ and $a_j \geq v_j \geq w_j$.
We cannot have $w < a$ else this would contradict condition~\eqref{eq:apex+cond+1} of Proposition~\ref{prop:facetapex+characterisation}, therefore we must have equality in some coordinate.
This implies $w$ is incident to every facet-apex in $T$ and so is contained in $S$.
However, $v$ was picked such that it maximizes the $i$th coordinate over $S$, and so this is a contradiction.
Therefore $v^{(i)}$ has a $k$-witness for every $k \neq i$, as well as $v_i^{(i)} = \infty$ and so is a characteristic point.

We finally note that if we label $v_i^{(i)}$ by its closed set, this set contains $\{v, e^{(i)}\}$.
As a result, we can realise $S$ as the join of the characteristic points $\{v \mid v \in S \cap V\}$ and $\{v^{(i)} \mid e^{(i)} \in S\}$.
}

\end{proof}
As a corollary of this result, the CP-order is a lattice if and only if it is equal to the vertex-facet lattice.

\subsection{Scarf poset} \label{subsec:scarf+poset}
In his seminal work on computing fixed-points, see~\cite{Scarf:1967}, Scarf  essentially uses a tropicalised version of the famous algorithm by Lemke \& Howson for finding equilibria~\cite{LemkeHowson:1964}.
The algorithm pivots along the vertices of a generic monomial tropical polyhedron and is essentially iterated over the facets of the monomial tropical polyhedron.
This work, along with its connections to commutative algebra, lead to the introduction of the Scarf complex, see~\cite[\S 6.2]{MillerSturmfels:2005} for further details.

We define a poset for a monomial tropical polyhedra $\monomial{V}$ which generalises the face poset of the Scarf complex. 
A point $p \in \TTclosed^d$ is a \emph{Scarf point} if there is a unique subset $X \subseteq V \cup \modmaxunit$ with $\tbary(X) = p$ (see Equation~\ref{eq:tropical+barycenter}).
\added{The \emph{Scarf poset} is the set of all Scarf points with maximal element $\bm{\infty}$ and minimal element $\bm{-\infty}$ with the standard partial order on $\TTclosed^d$.} 

\begin{lemma} \label{lem:Scarf+covector}
  A point $p \in \TTclosed^d$ is a Scarf point if and only if
  \begin{enumerate}[label=\alph*)]
  \item $\neighbour_p(\neighbour_p(0)) = [d]$, \label{cond:existence}
  \item for each node $v$ in $\neighbour_p(0)$, there exists $i \in [d]$ such that $\neighbour_p(0) \cap \neighbour_p(i) = \{v\}$. \label{cond:unique}
  \end{enumerate}
\end{lemma}

\begin{proof}
Let $X \subseteq V \cup \modmaxunit$ be a set of points and $\tilde{X} \subseteq V \cup \maxunit$ its corresponding set obtained by swapping $\bar{e}^{(i)}$ for $e^{(i)}$.
  The set $X$ fulfils $\tbary(X) = p$ if and only if for each $i \in [d]$ there is an element $x \in \tilde{X} \cap \neighbour_p(0)$ such that $i$ is adjacent with $x$, which condition \ref{cond:existence} ensures the existence of. 
  The subset $X$ of $V$ is not the unique set defining $p$ as its barycenter if and only if there is a point in $X$ which does not uniquely define any coordinate of $p$.
  This translates exactly to condition \ref{cond:unique}.
\end{proof}

Recall that given two elements $x,y$ in a poset $P$, we say $y$ \emph{covers} $x$ if $x < y$ and there does not exists $z \in P$ such that $x< z < y$.

\begin{proposition} \label{prop:Scarf+cover+subposet}
  The Scarf poset is a cover-preserving subposet of the CP-order, the max-min poset, the vertex-facet lattice and the $\max$-lattice.
\end{proposition}
\begin{proof}
  Firstly, the Scarf poset is a subposet of all these posets as it is a subposet of the CP-order.
  This follows by comparing the condition in Lemma~\ref{lem:Scarf+covector} and Lemma~\ref{lem:CP+covector}.
  Moreover, it is enough to show that the Scarf poset is a cover-preserving subposet of the $\max$-lattice by the inclusions of the other posets.
  Now, for each element $p$ of the $\max$-lattice there is a unique inclusionwise maximal set $X_p$ of generators with $\tbary(X_p) = p$.
  The partial order of the elements in the $\max$-lattice is just the order of these sets by inclusion.
  If $p$ is a Scarf point covered by the Scarf point $q$, then $|X_q| = |X_p| + 1$.
  Hence, $p$ is also covered by $q$ in the $\max$-lattice.
\end{proof}

We say a monomial tropical polyhedron $\monomial{V}$ is \emph{generic} if $u, v \in V$ with $u_i = v_i$ implies there exists $w \in V$ such that $w_k < \max(u_k, v_k)$ for all $k \in [d]$ with $\max(u_k,v_k) \neq -\infty$.
The following theorem shows that under this genericity assumption, most of the face posets presented so far coincide.
This notion of genericity is motivated further in Section~\ref{subsec:genericity} by its connection to monomial ideals.

\begin{theorem} \label{thm:generic+points+all+posets+coincide}
  If a monomial tropical polyhedron is generic, then the affine part of the vertex-facet lattice, the max-min poset, the CP-order and the Scarf poset are isomorphic. 
\end{theorem}
\begin{proof}
By Proposition~\ref{prop:Scarf+cover+subposet}, it suffices to show a principal facet-apex $a$ is a Scarf point; equivalently that it satisfies the conditions of Lemma~\ref{lem:Scarf+covector}.
  Comparing the characterisation of facet-apices in Lemma~\ref{lem:characterization+covectors+facet+apices}, the first condition $\neighbour_p(\neighbour_p(0)) = [d]$ is satisfied regardless of genericity.
  Suppose the second condition does not hold, as $\neighbour_a(0)$ has no degree one nodes there must exist $u, v \in V$ and $i \in [d]$ such that $u, v \in \neighbour_a(0) \cap \neighbour_a(i)$, or equivalently that $u_i = v_i = a_i$.
  By genericity, there exists $w \in V$ such that $w_k < \max(u_k,v_k) \leq a_k$, or equivalently that $w$ is a node of degree one in $\neighbour_a(0)$, contradicting Lemma~\ref{lem:characterization+covectors+facet+apices}.
\end{proof}

\section{Monomial ideals} \label{sec:monomial+ideals}

\subsection{A dictionary}

Let $\KK$ be a field, we consider $S = \KK[x_1,\dots,x_d]$ the $d$-variate polynomial ring with coefficients in $\KK$.
An ideal $I$ of $S$ is \emph{monomial} if it is generated by monomials ${\bf x}^u = x_1^{u_1}\cdots x_d^{u_d}$.
Monomial ideals are completely determined by the monomials they contain.
Therefore we can encode monomial ideals as a subset of $\ZZ^d_{\geq 0}$ by considering the set of exponents of all monomials $\supp(I) = \SetOf{u \in \ZZ^d}{{\bf x}^u \in I}$ called the \emph{support} of the monomial ideal.

As $I$ is always finitely generated and is closed under multiplication by $S$, its support $\supp(I)$ is given by a finite set of lattice points with a copy of the positive orthant attached to each.
This structure is very similar to that of a monomial tropical polyhedron, and the following construction shows we can associate a monomial tropical polyhedron to $I$ that encodes much of its information.
Let $U \subset \supp(I)$ be the subset corresponding to the unique minimal generating set of $I$.
For each $u \in U \subset \ZZ^d_{\geq 0}$, we define $\tilde{u} \in \TTmax^d$ by
\begin{equation} \label{eq:cech+hull}
\tilde{u}_i =
\begin{cases}
-\infty & u_i = 0 \\
u_i & u_i \neq 0
\end{cases} \enspace .
\end{equation}
Note that as $\tilde{u}$ is always an element of $(\ZZ_{>0} \cup \{-\infty\})^d$, we can always recover $u$ from $\tilde{u}$.
Denote the set of all such $\tilde{u}$ by $V_I \subset \TTmax^d$, the monomial tropical polyhedron corresponding to $I$ is denoted $\monomial{V_I}$.

This construction appears in the monomial ideal literature as the \emph{\v{C}ech hull} of a monomial ideal.

\begin{definition}[\cite{Miller:1998}]
Given a monomial ideal $I \subseteq S$, the \emph{\v{C}ech hull} of $I$ is the $S$-module
\[
\tilde{I} = \IdealOf{{\bf x}^u}{u \in \ZZ^n \text{ and } {\bf x}^{u^+} \in I} \subseteq \KK[x_1^{\pm},\dots,x_d^{\pm}] \quad , \quad u^+ = \max(u,0) \enspace .
\]
\end{definition}
The first two images of Figure \ref{fig:alexander+duality} show a diagrammatic construction of the \v{C}ech hull.
From \cite[Proposition 2.6]{Miller:1998}, we immediately get the following lemma.

\begin{lemma} \label{lem:ideal+polyhedron+correspondence}
Let $I$ be a monomial ideal.
Its support is encoded by the monomial tropical polyhedron $\monomial{V_I}$, explicitly:
\[
\supp(I) = \monomial{V_I} \cap \ZZ_{\geq 0}^d \quad , \quad \supp(\tilde{I}) = \monomial{V_I} \cap \ZZ^d .
\]
\end{lemma}

\begin{remark}
The \v{C}ech hull was introduced to better exhibit Alexander duality of monomial ideals.
We shall see in Section \ref{sec:alexander+duality} that monomial tropical polyhedra carry this duality even more naturally than the \v{C}ech hull.
\end{remark}

A monomial ideal is irreducible if it is of the form $\mathfrak{m}^a = \IdealOf{x_i^{a_i}}{a_i \geq 1}$ where $a \in \ZZ_{\geq 0}^d$.
An \emph{irreducible decomposition} of a monomial ideal $I$ is a representation
\begin{equation} \label{eq:irreducible+decomposition}
I = \mathfrak{m}^{a^{(1)}} \cap \dots \cap \mathfrak{m}^{a^{(r)}} \enspace .
\end{equation}
Expression~\eqref{eq:irreducible+decomposition} is unique if we restrict to \emph{irredundant} decompositions i.e., no component can be omitted.
\added{If the decomposition is unique, the ideals listed in~\eqref{eq:irreducible+decomposition} are the \emph{irreducible components of $I$.}}

Let $A$ be the set of exponents $a$ defining the irredundant irreducible decomposition of $I$.
Suppose $a \in A$ has no coordinates equal to zero, the support of $\mathfrak{m}^a$ is the set of non-negative integer lattice points $p$ such that $p_i \geq a_i$ for some $i \in [d]$.
Equivalently, the support of $\mathfrak{m}^a$ is cut out by the tropical linear inequality $\bigoplus_{i \in [d]} -a_i \odot p_i \geq 0$.
As the support is precisely those points covered by the tropical halfspace with apex $a$, one may expect the irredundant irreducible decomposition can be recovered tropically from $\monomial{V_I}$.
However, if $a$ has zero coordinates then $\mathfrak{m}^a$ has different behaviour which our tropical inequality must capture.
Specifically, when $a_i = 0$ the corresponding variable $x_i$ is omitted from the generators of $\mathfrak{m}^a$ entirely.
For each $a \in A$, we define $\hat{a} \in \TTmin^d$ by
\begin{align*}
\hat{a}_i =
\begin{cases}
\infty & a_i = 0 \\
a_i & a_i \neq 0
\end{cases} \enspace .
\end{align*}
Denote the set of all such $\hat{a}$ by $F_I \subset \TTmin^d$.

\begin{proposition} \label{prop:irred+comp+facet+apex}
Let $I$ be a monomial ideal and $F_I$ the set of defining facet-apices of $\monomial{V_I}$.
The irredundant irreducible decomposition of $I$ is
\begin{align*}
I = \bigcap_{\hat{a} \in F_I} \mathfrak{m}^a \enspace .
\end{align*}
\end{proposition}
\begin{proof}
We prove the equivalent statement that $\supp(I) = \bigcap_{\hat{a} \in F_I} \supp(\mathfrak{m}^a)$.
By Lemma \ref{lem:ideal+polyhedron+correspondence}, $\supp(\mathfrak{m}^a)$ is equal to the monomial tropical polyhedron $\monomial{V_{\mathfrak{m}^a}}$ intersected with the \added{integer points of the} positive orthant.
This is the monomial tropical polyhedron generated by vertices $\{v^{(i)} \added{\mid a_i \neq 0\}}$ whose entries are $-\infty$ everywhere except $v_i^{(i)} = a_i$\added{; if $a_i = 0$ then $v_i^{(i)}$ is not included by the definition of $\mathfrak{m}^a$.}

\added{Let $b$ be the apex of a principal halfspace of $\monomial{V_{\mathfrak{m}^a}}$.
By Proposition \ref{prop:facetapex+characterisation}, whenever $a_i \neq 0$ we must have $b_i$ finite else $b > v^{(i)}$ by condition~\eqref{eq:apex+cond+1}.
Furthermore, as $v^{(i)}$ is the only generator with finite $i$th coordinate, we must have $b_i = a_i$ by condition~\eqref{eq:apex+cond+2}.
Whenever $a_i=0$, all generators of $\monomial{V_{\mathfrak{m}^a}}$ have $-\infty$ in their $i$-th coordinate and so $b_i = \infty$ to satisfy condition~\eqref{eq:apex+cond+2}.
This implies that $b = a$ and that} $\monomial{V_{\mathfrak{m}^a}}$ has a single principal halfspace $H(-\hat{a})$, and is therefore equal to it.

Let $F_I$ be the set of facet-apices of $\monomial{V_I}$, then its minimal exterior representation can be expressed
\begin{equation} \label{eq:ext+rep}
\monomial{V_I} = \bigcap_{\hat{a} \in F_I} H(-\hat{a}) = \bigcap_{\hat{a} \in F_I} \monomial{V_{\mathfrak{m}^a}} \enspace .
\end{equation}
Intersecting \eqref{eq:ext+rep} with $\ZZ_{\geq 0}^d$ and using Lemma \ref{lem:ideal+polyhedron+correspondence}, we get the irreducible decomposition
\begin{align*}
\supp(I) = \bigcap_{\hat{a} \in F_I} \supp(\mathfrak{m}^a) \enspace .
\end{align*}
The final claim that this is irredundant comes from \eqref{eq:ext+rep} being the minimal exterior description.
\end{proof}

\subsection{Alexander duality} \label{sec:alexander+duality}
{\added{We recall Alexander duality, a form of duality for monomial ideals which arises by reversing the roles of generators and irreducible components.}}

A monomial ${\bf x}^a$ \emph{strictly divides} another monomial ${\bf x}^c$ if $a_i < c_i$ or $a_i = c_i = 0$ for all $i \in [d]$.
Given two vectors $a, c \in \ZZ_{\geq 0}^d$ such that ${\bf x}^a$ strictly divides ${\bf x}^c$, we let $c \smallsetminus a$ denote the vector with $i$th coordinate
\[
c_i \smallsetminus a_i =
\begin{cases}
c_i - a_i & a_i \geq 1 \\
0 & a_i = 0
\end{cases} \enspace .
\]
\begin{definition}
Let $I$ be a monomial ideal whose minimal generators all strictly divide ${\bf x}^c$, the \emph{Alexander dual} of $I$ with respect to $c$ is
\begin{align*}
I^{[c]} = \IdealOf{{\bf x}^{c \smallsetminus a}}{\mathfrak{m}^a \text{ is an irreducible component of } I} \enspace .
\end{align*}
\end{definition}
A monomial tropical polyhedron $\monomial{V}$ also comes with a natural duality, namely the complementary monomial tropical polyhedron $\complementarymonomial{V}$.
The following proposition shows this duality is equivalent to Alexander duality for monomial ideals.

\begin{proposition} \label{prop:alexander+duality}
Let $I$ be a monomial ideal whose minimal generators all strictly divide ${\bf x}^c$.
The support of its Alexander dual $I^{[c]}$ is encoded by the complementary monomial tropical polyhedron $\complementarymonomial{V_I}$, explicitly:
\[
\supp(I^{[c]}) = \left(c - \complementarymonomial{V_I}\right) \cap \ZZ_{\geq 0}^d \enspace ,
\]
where $c - \complementarymonomial{V_I}$ is the Minkowski difference of the vector $c \in \ZZ_{\geq 0}^d$ and the complementary monomial polyhedron generated by $V_I$.
\end{proposition}
\begin{proof}
This is simply a restatement of \cite[Lemma 2.11]{Miller:1998} in the language of monomial tropical polyhedra.
\end{proof}

\begin{figure}
\includegraphics[width=\textwidth]{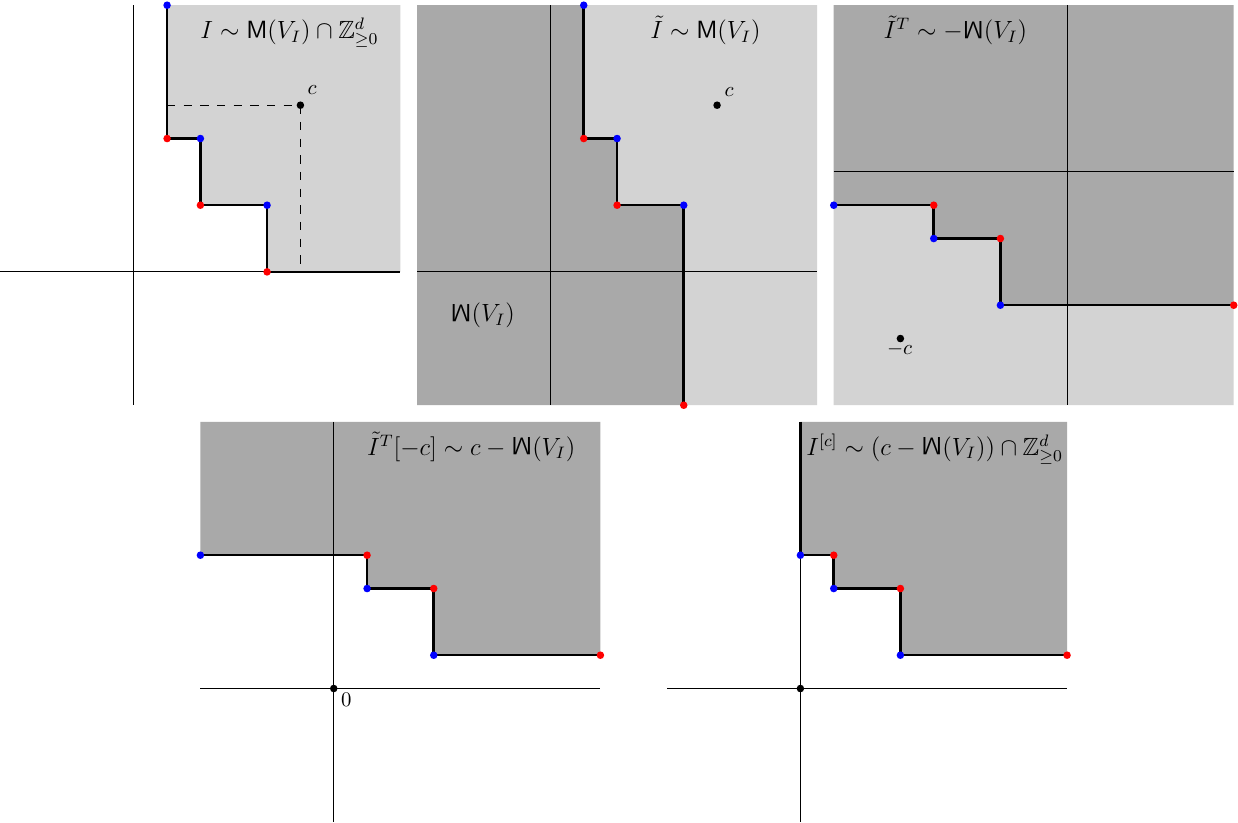}
\caption{The construction of $I^{[c]}$ from $I$ in the language of monomial tropical polyhedra and \cite{Miller:1998}.
Red dots are the vertices of $\monomial{V_I}$ and the blue dots are its facet-apices.}
\label{fig:alexander+duality}
\end{figure}

The \v{C}ech hull was introduced to exhibit Alexander duality more cleanly for monomial ideals.
Figure \ref{fig:alexander+duality} demonstrates the role the \v{C}ech hull plays in the construction of $I^{[c]}$ using the terminology of \cite{Miller:1998}.
Each step also displays the corresponding notions in the language of monomial tropical polyhedra.

We note that monomial tropical polyhedra carry this duality more naturally than \v{C}ech hulls.
The complementary cone $\complementarymonomial{V_I}$ carries all the information of $\monomial{V_I}$, and as it is a $\min$-tropical polyhedron, no translation or reflection is necessary.
Moreover, $\complementarymonomial{V_I}$ has the additional benefit of being a canonical choice and not requiring a choice of vector $c$.

\subsection{Resolutions} \label{sec:resolutions}
In the following, we study the link between resolutions of monomial ideals and \emph{lifts} of monomial tropical cones.
\added{We begin by briefly recalling the notion of cellular resolutions; for further details, see~\cite{MillerSturmfels:2005}.

A free $S$-module of rank $r$ is a module $F \simeq S(-{\bf{a}_1}) \oplus \dots \oplus S(-{\bf{a}_r})$ where $S(-{\bf a}_i)$ is the polynomial ring $S$ where the degree of each element is shifted by ${\bf a}_i$.
A \emph{chain complex} $\cF$ is a chain of maps between free $S$-modules
\[
\cF : 0 \leftarrow F_0 \leftarrow_{\delta_1} F_1 \leftarrow_{\delta_2} \cdots \leftarrow_{\delta_k} F_k \leftarrow 0
\]
such that $\delta_{i+1} \circ \delta_i = 0$ for all $i$.
A chain complex is \emph{exact} if $\ker(\delta_i) = \im(\delta_{i+1})$.
A \emph{free resolution} $\cF$ of a module $M$ is an exact chain complex such that $M \simeq F_0/\im(\delta_1)$.

We are particularly interested in free resolutions arising from polyhedral complexes.
Let $X$ be a polyhedral complex, we say that $X$ is \emph{labelled} if each vertex $v$ is labelled by some vector ${\bf a}_v \in  \ZZ_{\geq 0}^d$.
These induce labels on all faces of $X$, where the label of a face $P$ is equal to the componentwise maximum ${\bf a}_P = \max\left({\bf a}_v \mid v \subset P\right)$.
We will always consider these labels coming from the exponents of the generating set of some ideal $I_X = \langle {\bf x^{a_v}} \mid v \text{ vertex of } X \rangle$.
We will also need to consider the subcomplex $X_{\leq {\bf b}}$ of $X$ of faces $P$ with labels ${\bf a}_P \leq {\bf b}$ under the standard partial order on $\ZZ^d$.

The associated chain complex $\cF_X$ is the complex of free modules
\[
F_i = \bigoplus_{\substack{P \in X \\ \dim(P) = i-1}} S(-{\bf a}_P) \quad , \quad \delta(P) = \sum_{Q \text{ facet of } P} \sign(Q,P) {\bf x}^{{\bf a}_P - {\bf a}_Q}Q \, ,
\]
where $\sign(Q,P)$ is $\pm 1$ chosen in some consistent way such that $\delta^2 = 0$.
Note that we consider the symbols $P$ and $Q$ both as faces of $X$ and as basis vectors in degrees ${\bf a}_P$ and ${\bf a}_Q$.
We say $\cF_X$ is a \emph{cellular resolution} for $S/I_X$ if it is also a free resolution.

Cellular resolutions allow us to get a handle on purely algebraic invariants via methods from polyhedral geometry and topology.
In particular, the following proposition shows that whether $\cF_X$ is a resolution or not depends entirely on the topology of $X$.

\begin{proposition}\cite[Proposition 4.5]{MillerSturmfels:2005}\label{prop:acyclic}
$\cF_X$ is a cellular resolution for $S/I_X$ if the complex $X_{\leq {\bf b}}$ is contractible for all ${\bf b} \in \ZZ_{\geq 0}^d$.
\end{proposition}} 
\added{Note that Proposition~\ref{prop:acyclic} was originally stated as an `if and only if' statement, where `contractible' is replaced with `acyclic over $\KK$', i.e. has trivial homology over $\KK$.
However, this formulation suffices for our purposes as all of our approaches will be topological.}

We recall a class of cellular resolutions that is already well-established in the literature.
Let ${\bf t}^u = (t^{u_1},\dots,t^{u_d})$ for some real $t \added{ > 1}$.
\added{Given a monomial ideal $I$}, define the polyhedron
\begin{align*}
\added{\cP_{I,t}} &= \conv\SetOf{{\bf t}^u}{{\bf x}^u \in I} \subset \RR^d \\
&= \conv\SetOf{{\bf t}^v}{{\bf x}^v \text{ a minimal generator of } I} + \RR_{\geq 0}^d \, .
\end{align*}

\begin{definition}
The \emph{hull complex} $\hull(I)$ of a monomial ideal $I$ is the polyhedral cell complex of all bounded faces of \added{$\cP_{I,t}$} for $t\gg 0$.
This complex is naturally labelled, with each vertex corresponding to a minimal generator of $I$.
The associated cellular resolution $\cF_{\hull(I)}$ is called the \emph{hull resolution} of $I$.
\end{definition}

The hull complex is a special case of a more general class of polyhedra~\cite{DevelinYu:2007} that we detail now.
In the following, we consider the field of \added{\emph{(generalized)} real Puiseux series} $\puiseux{\RR}{t}$, whose elements 
\[
\gamma = \sum c_i t^{a_i} \ , \ a_i,c_i \in \RR \ , \ a_0 > a_1 > a_2 > \dots
\]
are (locally finite) formal power series with real exponents\added{~\cite{Markwig:2010}}.
We say $\gamma$ is positive if its leading term has positive coefficient.
This makes $\puiseux{\RR}{t}$ an ordered field via $\gamma > \delta$ if and only if $\gamma - \delta$ is positive.
As a result, one can form polyhedra over $\puiseux{\RR}{t}$ as solution sets to linear inequalities, as one would over $\RR$.
Furthermore, given a polyhedron over $\puiseux{\RR}{t}$ one can substitute $t$ for some $\tau \in \RR$ yielding an ordinary polyhedron over $\RR$.
For large $\tau$, these polyhedra are combinatorially isomorphic \cite{JoswigLohoLorenzSchroeter:2016}.

Tropical polyhedra arise as the image of polyhedra over certain valued fields. 
In particular, $\puiseux{\RR}{t}$ carries the valuation map
\begin{align*}
\val : \puiseux{\RR}{t} &\rightarrow \added{\TTmax} \\
\sum c_i t^{a_i} &\mapsto a_0
\end{align*}
where $\val(0) = -\infty$.
Restricting to $\puiseux{\RR}{t}_{\geq 0}$, the semiring of nonnegative Puiseux series, turns $\val$ into an order-preserving homomorphism of semirings.

\added{\begin{definition}
Let $Q \subseteq \TTmax^d$ be a tropical polyhedron with minimal representation
\[
Q = \tconv{V}\oplus \tcone{W} \, , \, V, W \subseteq \TTmax^d \, .
\]
A polyhedron $P \subseteq \puiseux{\RR}{t}_{\geq 0}^d$ is called a \emph{lift} of $Q$ if there exists a representation
\[
P = \conv(X) + \cone(Y) \, , \, X, Y \subseteq \puiseux{\RR}{t}_{\geq 0}^d \, ,
\]
such that $\val(X) = V$ and $\val(Y) = W$.
\end{definition}
As $\val$ is a surjective order-preserving homomorphism, it follows that $P$ being a lift of $Q$ implies $\val(P) = Q$.
Often in the literature, this condition is taken to be the definition of a lift.
However, our definition is a strengthening of this, as it ensures vertices are mapped to vertices and rays are mapped to rays.}

We can consider $\cP_{I,t}$ as a polyhedron over $\puiseux{\RR}{t}$, namely
\[
\cP_I = \conv\SetOf{{\bf t}^v}{{\bf x}^v \text{ a minimal generator of } I} + \puiseux{\RR}{t}_{\geq 0}^d \, .
\]
Furthermore 
its bounded \added{faces} give a cellular resolution of $I$: the hull complex.
The following proposition shows we can \added{perform} this procedure for any lift of $\monomial{V_I}$.
\added{Specifically, we get a cellular resolution of $I$ from the complex of bounded faces of an associated lifted polyhedron $M$ whose vertices are labelled by their image under the valuation map, i.e. $x$ is labelled by $\val(x) \in V_I$.}

\begin{proposition} \label{prop:lift+resolution}
Let $M \subset \puiseux{\RR}{t}_{>0}^d$ be a lift of $\monomial{V_I}$.
The complex of bounded faces of $M$ yields a cellular resolution of $I$.
\end{proposition}
\begin{proof}
\added{The following is an adaptation of the proof of~\cite[Theorem 3.2]{DevelinYu:2007}.
By Proposition~\ref{prop:acyclic}, it suffices to check that for each ${\bf b} = (b_1, \dots, b_d) \in \ZZ^d_{\geq 0}$, the bounded faces of $M$ with labels dividing ${\bf b}$ form an acyclic complex.
Let the coordinates on $\puiseux{\RR}{t}^d$ be given by $z_1, \dots, z_d$ and consider the linear function on $\puiseux{\RR}{t}^d$ given by
\[
f(x) = t^{-b_1}z_1 + \dots + t^{-b_d}z_d \, .
\]
Then for a vertex of $M$, $f(z)$ has positive leading exponent if and only if the corresponding generator does not divide $x^{\bf b}$.
Thus the bounded faces of $M$ in the halfspace given by $f(z) \leq t^{\frac{1}{2}}$ are precisely those in the subcomplex $X_{\leq {\bf b}}$.
Consider the polytope obtained as the intersection $M \cap \SetOf{p}{f(p) \leq t^{\frac{1}{2}}}$.
The complex $X_{\leq {\bf b}}$ is precisely the subcomplex of faces disjoint from the face $M \cap \SetOf{p}{f(p) = t^{\frac{1}{2}}}$, and is therefore acyclic by~\cite[Lemma 4.18]{MillerSturmfels:2005}.}
\end{proof}

\begin{remark}
\added{While Proposition~\ref{prop:lift+resolution} is based on a similar statement for the hull complex, $\cP_I$ may not be a lift of $\monomial{V_I}$ as it does not modify the generators by replacing $0$ by $-\infty$.}
The replacement of $0$ by $-\infty$ does not play a big role from the point of view of the resolution.
The monomial ideals $I$ and $x_1\cdots x_d\cdot I$ have isomorphic resolutions, which equates to shifting the monomial tropical polyhedron in the direction $(1,\dots,1)$ to replace $0$ by a positive number.
\end{remark}

\subsection{LCM-lattice}

The \emph{LCM-lattice} is a natural sublattice of the $\max$-lattice which was introduced in~\cite{GasharovPeevaWelker:1999} for monomial ideals, defined as follows.
Let $I$ be a monomial ideal minimally generated by monomials ${\bf x}^{v^{(1)}},\dots,{\bf x}^{v^{(n)}}$.
Its LCM-lattice $L_I$ is the lattice of elements
\[
\added{{\bf x}^{J}} = \lcm\SetOf{{\bf x}^{v^{(i)}}}{i \in J} \ , \ J \subseteq [n]
\]
ordered by divisibility.
The unique maximal element is $\lcm({\bf x}^{v^{(1)}},\dots,{\bf x}^{v^{(n)}})$ and set $ \hat{0} = \lcm(\emptyset) = 1$ to be the unique minimal element.
\added{As with the $\max$-lattice, the LCM-lattice is a lattice as each pair of elements ${\bf x}^{J}, {\bf x}^{K}$ have a unique meet and join defined by
\begin{align*}
{\bf x}^{J} \vee {\bf x}^{K} &= \lcm\SetOf{{\bf x}^{v^{(i)}}}{i \in J\cup K} = {\bf x}^{J \cup K} \, , \\
{\bf x}^{J} \wedge {\bf x}^{K} &= \lcm\SetOf{{\bf x}^{v^{(i)}}}{i \in J\cap K} = {\bf x}^{J \cap K} \, .
\end{align*}}

The LCM-lattice is a powerful invariant: it determines a minimal free resolution \cite{GasharovPeevaWelker:1999} and the Stanley depth of $I$ \cite{IchimKatthanMoyanoFernandez:2017}.
The result of particular note to us is that it encodes the Betti numbers of $I$ via its homology.
We briefly recall some basics of lattice homology at the end of this subsection.
\begin{theorem}\cite[Theorem 2.1]{GasharovPeevaWelker:1999} \label{thm:lcm+betti}
For $i \geq 1$, we have
\[
\beta_{i,u}(S/I) =
\begin{cases}
\dim\tilde{H}_{i-2}((\hat{0},{\bf x}^u)_{L_I};\KK) \ & \ {\bf x}^u \in L_I \\
0 \ & \ {\bf x}^u \notin L_I
\end{cases}
\enspace .
\]
\end{theorem}

The LCM-lattice of $I$ is a sublattice of the $\max$-lattice of $\monomial{V_I}$ via the embedding
\[
\lcm\SetOf{{\bf x}^{v^{(i)}}}{i \in J} \mapsto \max\SetOf{\tilde{v}^{(i)}}{i \in J} \enspace .
\]
Furthermore, the only elements of the $\max$-lattice that are \added{not} elements of the LCM-lattice are those with $\infty$ in some coordinate.
As a corollary to Theorem \ref{thm:lcm+betti}, we get that the $\max$-lattice of $\monomial{V_I}$ determines the Betti numbers of $I$ by restricting to elements without $\infty$ in any coordinate.
As a result, we define the \emph{LCM-lattice of a monomial tropical polyhedron} to be the induced sublattice of the $\max$-lattice consisting of elements without $\infty$ in some coordinate.

We end this subsection recalling some concepts from poset homology necessary for Section~\ref{subsec:properties+facet+complex}. 
Let $P$ be a poset with minimal and maximal elements $\hat{0}, \hat{1}$.
The \emph{order complex} $\Delta(P)$ is the abstract simplicial complex whose faces are chains in $P \setminus \{\hat{0},\hat{1}\}$.
The homology groups of $P$ are given by the simplicial homology groups of $\Delta(P)$, i.e. $\tilde{H}_i(P) = \tilde{H}_i(\Delta(P))$.
This notion extends for open intervals $(p,q)_P = \SetOf{r \in P}{p<r<q}$ of $P$.

Let $L$ be a lattice, we have more homological tools available than for arbitrary posets.
A crosscut of $L$ is a subset of elements $C \subset L$ such that
\begin{enumerate}
	\item $\hat{0}, \hat{1} \notin C$,
	\item If $p,q \in C$ then $p \nless q$ and $q \nless p$,
	\item any finite chain in $L$ can be extended to a chain which contains an element of $C$.
\end{enumerate}
A notable example is that the set of atoms of $L$ form a crosscut.
We say a finite subset $\{p_1, \dots, p_n\} \subset C$ \emph{spans} if $p_1 \wedge \dots \wedge p_n = \hat{0}$ and $p_1 \vee \dots \vee p_n = \hat{1}$.
Let $\Delta(C)$ be the abstract simplicial complex whose faces are the non-spanning subsets of $C$.
We again define the homology of a crosscut by $\tilde{H}_i(C) = \tilde{H}_i(\Delta(C))$.

A key observation is that the homology of a crosscut is isomorphic to the homology of the lattice, and that homology is invariant under the choice of crosscut.
\begin{theorem} [Theorem 3.1 \cite{Folkman:1966}] \label{thm:crosscut+homology}
Let $L$ be a lattice and $C$ a crosscut of $L$.
Then $\tilde{H}_i(L) \cong \tilde{H}_i(C)$ for all $i \in \ZZ$.
\end{theorem}

\subsection{Syzygy points and the Betti poset} \label{subsec:syzygy+betti+poset}
The Betti numbers of an ideal are encoded by the Koszul complexes of $I$ in various degrees.
These complexes have geometric formulation that can be naturally recovered from monomial tropical polyhedra.

\begin{definition}[\cite{MillerSturmfels:2005, Kappes:2006}]\label{def:syzygy}
The \emph{Koszul complex} of a point $p \in \added{\monomial{V}}$ is the simplicial complex
\[
\Delta_{p} = \SetOf{J \subseteq [d]}{\added{\exists q \in \monomial{V} \text{ such that } q \leq p \, , \, q_j < p_j \, \forall j \in J }} \enspace .
\]
Then $p$ is a \emph{syzygy point} if $\Delta_p$ has non-trivial homology.
The syzygy points are equipped with the standard partial order on $\TTclosed^d$. 
\end{definition}
\added{Note that one could extend Definition~\ref{def:syzygy} to $\closedmonomial{V}$.
However, for any point $p$ with $p_i = +\infty$, the Koszul complex $\Delta_p$ is a cone, i.e. $F \in \Delta_p$ implies that $i \in F$.
As cones are contractible, these points are never syzygy points.
As a result, we only consider points of $\monomial{V}$.}

\added{We first note that syzygy points fit neatly into our hierarchy of posets on $\closedmonomial{V}$.
\begin{lemma}[{\cite{Kappes:2006}}] \label{lem:syzygy+points+CP+points}
  Syzygy points are a strict subset of characteristic points.
\end{lemma}
\begin{proof}
  By~\cite[Lemma 5.12]{Kappes:2006}, a point \added{$p \in \monomial{V}$ (with $p_i \neq +\infty$)} is a characteristic point if and only if $\Delta_p$ is not a cone.
  As a cone is contractible, the claim follows.
  \added{Note that this is strict as $\closedmonomial{V} \setminus \monomial{V}$ always contains a facet-apex, and facet-apices are characteristic points.}
\end{proof}

\added{Note that even if we restrict the CP-order to $\monomial{V}$, the syzygy points may still be a strict subset.}
In dimension three, a case check demonstrates that all characteristic points are also syzygy points, but for dimension four and higher there are examples of characteristic points $p$ whose complex $\Delta_p$ has trivial homology\added{; see~\cite{Kappes:2006} for an explicit example}.}

Given a monomial tropical polyhedron $\monomial{V_I}$ and $p \in \ZZ^d$, the complex $\Delta_p$ is the Koszul complex of $I$ in degree $p$ as defined in \cite[Definition 1.33]{MillerSturmfels:2005}.
Moreover, \cite[Theorem 1.34]{MillerSturmfels:2005} states that these points encode the Betti numbers of $I$ via the equation
\begin{equation} \label{eq:syzygy+point+homology}
\beta_{i,p}(I) = \beta_{i+1,p}(S/I) = \dim\tilde{H}_{i-1}(\Delta_p;\KK) \enspace .
\end{equation}
Recently, Koszul complexes have played a key role in the construction of a canonical minimal free resolution for arbitrary monomial ideals \cite{EagonMillerOrdog:2019}.
These results show that they play a fundamental role in the homology of monomial ideals.
The following proposition shows they are encoded by the combinatorial structure of monomial tropical polyhedra, in particular via covector graphs.

\begin{proposition}
The Koszul complex of a point $p$ can be determined from its covector graph $G_p$.
Explicitly, it is the simplicial complex
\[
\Delta_{p} = \SetOf{J \subseteq [d]}{\neighbour_p(0) \nsubseteq \neighbour_p(J)} \enspace .
\]
\end{proposition}
\begin{proof}
\added{The condition that $J \subseteq [d]$ is a face of the Koszul complex $\Delta_p$ is equivalent to there existing} some $v \in V$ such that $p \geq v$ and $p_j > v_j$ for all $j \in J$.
These two conditions are equivalent to the covector conditions $v \in \neighbour_p(0)$ and $v \notin \neighbour_p(j)$ for all $j \in J$.
\end{proof}

A related homological construction is the \emph{Betti poset}, introduced by Clark and Mapes \cite{ClarkMapes:2014a,ClarkMapes:2014b} as a distillation of a poset to its homologically nontrivial part.
\begin{definition}
The \emph{Betti poset} of $P$ is the induced subposet of homologically contributing elements
\[
B(P) = \SetOf{p \in P}{\tilde{H}_i((\hat{0},p)_P; \KK) \neq 0 \text{ for some } i} \enspace .
\]
\end{definition}
Let $I$ be a monomial ideal and $L_I$ its LCM-lattice.
Theorem \ref{thm:lcm+betti} shows $L_I$ totally determines the Betti numbers of $I$, combining with \cite[Theorem 1.4]{ClarkMapes:2014b} implies we can restrict to the Betti poset of $L_I$
\begin{equation} \label{eq:betti+poset+homology}
\beta_{i,p}(S/I) = \dim_{\KK}\tilde{H}_{i-2}((\hat{0},p)_{B(L_I)};\KK) \enspace ,
\end{equation}
and that $B(L_I)$ is the minimal poset with this property.
As a result, when we refer to the Betti poset, we consider only the Betti poset of the LCM-lattice unless explicitly stated.

Equations \eqref{eq:syzygy+point+homology} and \eqref{eq:betti+poset+homology} show two distinct methods to compute the Betti numbers of a monomial ideal.
The following proposition shows this equivalence holds for an arbitrary monomial tropical polyhedron.

\begin{proposition} \label{prop:syzygy+Betti+poset}
Fix a monomial tropical polyhedron. The poset of its syzygy points with the standard partial order is equal to its Betti poset.
\end{proposition}
\begin{proof}
Let $L$ be the LCM-lattice of $\monomial{V}$.
\added{We first note that every syzygy point is contained in $L$, as the syzygy points are a subset of the CP-order, and therefore also the $\max$-lattice of $\monomial{V}$.
As syzygy points have no coordinates equal to $+\infty$, the claim follows.}

Fix some $p \in L$; we show that $\Delta_p$ and $(\hat{0},p)_L$ have the same homology.
As both posets have the same partial order, the claim of the proposition follows from this.

Let $v^{(1)},\dots,v^{(s)} \subseteq V$ be the set of vertices such that $p \geq v^{(i)}$.
These are the set of atoms of the interval $(\hat{0},p)_L$, and therefore form a crosscut.
By Theorem \ref{thm:crosscut+homology}, the homology of $(\hat{0},p)_L$ is equivalent to the homology of the simplicial complex
\begin{align*}
\Sigma = \SetOf{I \subseteq [s]}{\max_{i \in I}\left(v^{(i)}\right) < p} \enspace .
\end{align*}

For each $j \in [d]$, define the simplicial complex
\[
\sigma_j = \SetOf{I \subseteq [s]}{\max_{i \in I}\left(v_j^{(i)}\right) < p_j} \enspace .
\]
Note that each $\sigma_j$ is a simplex and is a face of $\Sigma$.
Furthermore, for any face $I \subseteq \Sigma$, there exists some $j \in [d]$ such that $\max_{i \in I}\left(v_j^{(i)}\right) < p_j$.
Therefore $\Sigma = \bigcup_{j \in [d]} \sigma_j$ and any intersection $\bigcap_{j \in J} \sigma_j$ is contractible or empty.
Let
\[
N = \SetOf{J \subseteq [d]}{\bigcap_{j \in J} \sigma_j \neq \emptyset}
\]
be the nerve of $\Sigma$.
By the Nerve theorem~\cite[Theorem 10.7]{Bjoerner:1995}, we obtain that $\Sigma$ and $N$ are homotopy equivalent.
However, the intersection $\bigcap_{j \in J} \sigma_j$ is non-empty if and only if there exists some $v \in V$ such that $p \geq v$ and $p_j > v_j$ for all $j \in J$.
This is equivalent to $p - \sum_{j \in J} \varepsilon e_j \in \monomial{V}$, and therefore $\Delta_p$ is the nerve of $\Sigma$.
\end{proof}

\added{Combining this result with Lemma~\ref{lem:syzygy+points+CP+points}, we deduce that the Betti poset of $\monomial{V}$ is a (strict) subposet of the CP-order.
This allows it to fit neatly into our family of posets from Section~\ref{sec:face+posets}, as other than the Scarf poset, it is a subposet of them all.}

\subsection{Genericity} \label{subsec:genericity}

For monomial ideals and monomial tropical polyhedra there are several notions of genericity, as discussed in~\cite{MillerSturmfels:2005} and~\cite{AllamigeonBenchimolGaubertJoswig:2015}, respectively. 

\begin{definition}
A monomial ideal $\langle m_1,\dots,m_n\rangle$ is \emph{generic} if whenever two distinct minimal generators $m_i$ and $m_j$ have the same nonzero degree in some variable, a third generator $m_k$ divides $\lcm(m_i,m_j)x_i^{-1}$ for all variables $x_i$ dividing $\lcm(m_i,m_j)$.
\end{definition}

This is equivalent to the notion of genericity for monomial tropical polyhedra given at the end of Section~\ref{subsec:scarf+poset}.
There are two stronger notions of genericity arising from commutative algebra and tropical convexity.
A monomial tropical polyhedron is \emph{strongly generic} if $u_i \neq v_i$ for all $u, v \in V, i \in [d]$.
There is an even stronger notion of \emph{tropically generic} where the generators $v_1,\dots,v_n \in \TTmax^d$ are in tropically generic position i.e. the matrix of (homogenised) generators has no tropically singular minors; see~\cite[\S 2.1.2]{AllamigeonBenchimolGaubertJoswig:2015} for details. 
This notion of genericity is too refined for monomial tropical polyhedra, as discussed in Remark \ref{rem:max+lattice+covector+comparison}: the combinatorics of ideals and the $\max$-lattice are purely determined by the orderings of coordinates, while the combinatorics of covectors are more refined.

One can study properties of arbitrary monomial ideals via the generic monomial ideals arising via \emph{genericity by deformation} \cite[Section 6.3]{MillerSturmfels:2005}.
The same procedure extends very naturally to monomial tropical polyhedra.
Let $\monomial{V}$ be a monomial tropical polyhedron where $V = \{v^{(1)},\dots, v^{(n)}\}$.
\begin{definition} \label{def:deformation+monomial+polyhedron}
A \emph{deformation} of $\monomial{V}$ is the monomial tropical polyhedron arising from a choice of vectors $\varepsilon^{(j)} \in \RR^d$ for $j \in [n]$ such that
\begin{equation} \label{eq:deformation}
v^{(j)}_i < v^{(k)}_i \Rightarrow v^{(j)}_i + \varepsilon^{(j)}_i < v^{(k)}_i + \varepsilon^{(k)}_i \enspace .
\end{equation}
Define $\monomial{V_{\varepsilon}}$ to be the monomial tropical polyhedron generated by 
\[
V_{\varepsilon} = \SetOf{v^{(j)} + \varepsilon^{(j)}}{i \in [n]} \enspace .
\]
\end{definition}
We say a face poset $P$ of a monomial tropical polyhedron $\monomial{V}$ is \emph{invariant under deformation} if the respective face poset of any deformation $\monomial{V_{\varepsilon}}$ is isomorphic to $P$.
Using this notion, we get a nice connection between notions of genericity and deformations of face posets.
\begin{proposition}
A monomial tropical polyhedron is
\begin{itemize}
\item generic \added{if and only if} its Scarf poset is invariant under deformation,
\item strongly generic \added{if and only if} its max-lattice is invariant under deformation,
\item tropically generic \added{if and only if} its pseudovertex poset is invariant under sufficiently small deformations.
\end{itemize}
\end{proposition}
\begin{proof}
The first and second equivalence are given by \cite[Theorem 6.26]{MillerSturmfels:2005} and Corollary~\ref{coro:max+lattice+dependence+order} respectively.
\added{For the third, \cite[Proposition 24]{DevelinSturmfels:2004} shows that a tropically generic monomial tropical polyhedron has its pseudovertices in bijection with maximal cells of a regular subdivision of the product of simplices $\Delta_{n-1} \times \Delta_{d-1}$.
Furthermore, this subdivision is given by weighting the vertices of $\Delta_{n-1} \times \Delta_{d-1}$ by the matrix $V$.
By \cite[Lemma 2.3.16]{triangulations}, this regular subdivision is invariant under sufficiently small perturbations of $V$; in particular, for any sufficiently small deformation $V_\varepsilon$.}
\end{proof}
Note that by Theorem~\ref{thm:generic+points+all+posets+coincide} we could replace Scarf poset with CP-order, max-min poset or vertex-facet lattice.
\begin{remark}
  While genericity by deformation was originally defined for monomial ideals, it is a more natural construction for monomial tropical polyhedra than monomial ideals for two reasons.
Firstly, as monomial ideals naturally have integer exponents, non-integer deformations must be defined formally as an ideal of a larger polynomial ring with real exponents.
Monomial tropical polyhedra have no such constraints as objects naturally living in $\TTmax^d$.
Secondly, condition \eqref{eq:deformation} is a simplified version of the equivalent condition for monomial ideals, as one does not need to differentiate between the cases where $u_i$ is zero and non-zero for monomial tropical polyhedra; $u_i = -\infty$ is automatically accounted for.
\end{remark}

\section{Facet complex} \label{sec:facet+complex}

Recall the definition of the vertex-incidence graph from Definition~\ref{def:vertex-facet-lattice}. We define the \emph{facet complex} as the abstract simplicial complex whose maximal simplices are the sets of vertices incident with a facet. 
Letting $\overline{V}$ be the set of vertices and rays of a monomial tropical polyhedron, we denote the corresponding facet complex by $\cF(V)$.

\subsection{Embedding in the facet complex} \label{subsec:embedding+facet+complex}
In the following we use the natural correspondence between the tropical inequality $\max(x_i - a_i) \geq 0$ and the apex $a$ of the corresponding halfspace.
Recall that an inequality is \emph{valid} if it is satisfied by all points of $\monomial{V}$.
\begin{lemma} \label{lem:incident+vertices+valid+inequalities}
  Let $S$ be a subset of $\overline{V}$ such that there is an apex $a \in \TTclosed^d$ which corresponds to a valid inequality for $\monomial{V}$ and which is incident with $S$ but not with $\overline{V} \setminus S$.
  Then $S$ is a simplex in the facet complex. 
\end{lemma}
\begin{proof}
  An apex corresponds to a valid inequality exactly if there is no element $v$ in $V$ with $v < a$.
By the extremality of facet-apices and the duality in Theorem~\ref{thm:complementary-cones}, there is a facet-apex $b \in \TTclosed^{d}$ with $a \leq b$.
  By the definition of incidence and as $b$ is also a valid inequality, this implies that $b$ is incident with the elements in $S$.
\end{proof}
A similar statement holds for the set $\maxunit$ of rays, as the far-apex is incident with precisely them, and corresponds to the valid inequality \eqref{eq:far+face}.

For an arbitrary polytope \added{$P$} over $\RR$ or $\puiseux{\RR}{t}$, we can define its \emph{facet complex} \added{$\cF(P)$} to be the complex whose maximal simplicies are those vertex sets that form a facet.
This notion can be extended to a polyhedron by considering a projectively equivalent polytope. 

\added{Let $M \subset \puiseux{\RR}{t}_{\geq 0}^{d}$ be a lift of $\monomial{V}$, i.e.,
\[
M = \conv(X) + \cone(Y) \quad , \quad \val(X) = V \, , \, \val(Y) = \maxunit \, .
\]
All vectors in the preimage of $e^{(i)}$ are equivalent to the $i$-th unit vector of $\puiseux{\RR}{t}_{\geq 0}^d$
\[
y^{(i)}_k =  \begin{cases} 1 & i = k \\ 0 & i \neq k \end{cases} \, ,
\]
upto scaling by an element of $\puiseux{\RR}{t}_{\geq 0}$, and so $\cone(Y) = \puiseux{\RR}{t}_{\geq 0}^d$.
We make the additional restriction that $\val$ induces a bijection between the vertices of $M$ and $\monomial{V}$, i.e., $X = \{x^{(1)},\dots, x^{(n)}\}$ such that $\val(x^{(i)}) = v^{(i)}$.
By identifying elements of $X\cup Y$ with their images in $V \cup \maxunit$ via the valuation map, the facet complex $\cF(M)$ of $M$ is naturally on the same vertex set as $\cF(V)$.}
The following theorem shows the facet complex \added{$\cF(M)$} of $M$ is \added{naturally} a subcomplex of $\cF(V)$.

\begin{theorem} \label{thm:lift+subcomplex}
Let $M$ be a lift of $\monomial{V}$ \added{such that $\val$ induces a bijection between the vertices of $M$ and $\monomial{V}$.
Then $\cF(M)$ is a subcomplex of $\cF(V)$.}
\end{theorem}
\begin{proof}
  \added{Without loss of generality}, let $\{x^{(1)}, \dots, x^{(k)}\}$ be the vertices of a facet of $M$ \added{with rays $\{y^{(1)}, \dots, y^{(\ell)}\}$}. 
  Then there is a non-negative vector $c \in \puiseux{\RR}{t}_{\geq 0}^d$ such that $\trans{c}\cdot x^{(i)} = 1$ and $\trans{c} \cdot u > 1$ for all other vertices $u$ of $M$. 
  \added{Furthermore, $\trans{c}\cdot (x^{(i)} + \lambda y^{(j)}) = 1$ for all $\lambda \in \puiseux{\RR}{t}_{\geq 0}$ and $1 \leq j \leq \ell$: this implies that $c_j = 0$ in the first $\ell$ entries.}

  As the valuation map is an order preserving semiring homomorphism, the inequality $\val(c) \odot x \geq 0$ is valid for the monomial tropical polyhedron.
  Furthermore, it is tight at the vertices $\val(x^{(i)})$.
  \added{In addition, the rays $e^{(1)}, \dots, e^{(\ell)}$ are incident to the apex $\val(c)$ as $\val(c_j) = \infty$ for $1 \leq j \leq \ell$.}
  Now the claim follows from Lemma~\ref{lem:incident+vertices+valid+inequalities}.

\end{proof}

\added{
\begin{remark}
For an arbitrary lift $M$ of $\monomial{V}$, we may have multiple vertices of $M$ that map to the same vertex of $\monomial{V}$ under the valuation map.
As a result, $\cF(M)$ and $\cF(V)$ may not be on the same vertex set.
We can state a weaker version of Theorem~\ref{thm:lift+subcomplex} for arbitrary lifts by defining a ``degenerate'' facet complex that identifies vertices of $M$ if they have the same image in the valuation map.
This degenerate facet complex will be subcomplex of $\cF(V)$.
\end{remark}
}

\added{The facet complex also captures the crucial information of the deformations of a monomial ideal.
We make this precise in the next theorem. }

Let $\varepsilon^{(j)} \in \RR^d$ for $j \in [n]$, such that $\monomial{V_{\varepsilon}}$ is a deformation of $\monomial{V}$\added{; recall that this is equivalent to $\varepsilon^{(j)}$ satisfying
\[
v^{(j)}_i < v^{(k)}_i \Rightarrow v^{(j)}_i + \varepsilon^{(j)}_i < v^{(k)}_i + \varepsilon^{(k)}_i \enspace .
\]
}
\added{By identifying $v^{(j)}$ with the deformed vertex $v^{(j)} + \varepsilon^{(j)}$, the facet complex $\cF(V_{\varepsilon})$ can be considered on the same vertex set as $\cF(V)$.}
The following theorem shows $\cF(V_{\varepsilon})$ is a subcomplex of $\cF(V)$.

\begin{theorem}
Let $\monomial{V_{\varepsilon}}$ be a deformation of $\monomial{V}$.
\added{Then $\cF(V_{\varepsilon})$ is a subcomplex of $\cF(V)$.}
\end{theorem}
\begin{proof}
\added{
Denote the vertices of $\monomial{V_\varepsilon}$ by $w^{(j)} := v^{(j)} + \varepsilon^{(j)}$.
Recall that the definition of deformation gives us the following implications that will be of use throughout:
\begin{align*}
v_i^{(j)} < v_i^{(k)} &\Rightarrow w_i^{(j)} < w_i^{(k)} \, , \\
w_i^{(j)} = w_i^{(k)} &\Rightarrow v_i^{(j)} = v_i^{(k)} \, . \\
\end{align*}
Let $b^\varepsilon$ be a facet-apex of $\monomial{V_\varepsilon}$ and define the point $b \in \monomial{V}$ by
\[
b_i = \begin{cases} v_i^{(j)} & b^\varepsilon_i = w_i^{(j)} \\ \infty & b^\varepsilon_i = \infty  \end{cases} \, .
\]
We note that $b$ is defined such that $b^\varepsilon \rightarrow b$ as we let each $\varepsilon^{(j)} \rightarrow 0$.

We first claim that $b$ is incident to some principal facet-apex $a$ of $\monomial{V}$.
Recall that the interior of $\monomial{V}$ is given by $\bigcup_{v \in V} (v + \RR_{>0})$, therefore $b$ is incident to some principal facet-apex if there exists no $v \in V$ such that $v < b$.
Suppose there exists some $v^{(k)} < b$, then for each $i \in [d]$ with $b_i \neq \infty$, there exists some $j$ such that $v_i^{(k)} < v_i^{(j)} = b_i$.
This implies $w_i^{(k)} < w_i^{(j)} = b_i^\varepsilon$, and so $w^{(k)} < b^\varepsilon$.
However, this contradicts condition~\eqref{eq:apex+cond+1} of Proposition~\ref{prop:facetapex+characterisation} that $b^\varepsilon$ is a facet-apex of $\monomial{V_\varepsilon}$.
Therefore $b$ must be incident to some facet-apex $a$.

Let $\Delta_{b^{\varepsilon}}$ and $\Delta_{a}$ be the maximal simplices of $\cF(V_\varepsilon)$ and $\cF(V)$ corresponding to $b^{\varepsilon}$ and $a$ respectively.
If the vertex $w^{(j)}$ is incident to the facet-apex $b^\varepsilon$ of $\monomial{V_\varepsilon}$, deformation implies $v^{(j)}$ must also be incident to the point $b$ of $\monomial{V}$, and therefore also $a$.
Similarly, if the ray $e^{(i)}$ is incident to $b^\varepsilon$ then it is also incident to $a$.
As a result, $\Delta_{b^{\varepsilon}}$ is a subsimplex of $\Delta_a$.
}
\end{proof}

\begin{remark}
  Recall that in the context of monomial ideals and their resolutions, one usually uses the Scarf complex instead of the Scarf poset.
  Using a slight modification of \cite[Theorem 6.13]{MillerSturmfels:2005} and Proposition~\ref{prop:Scarf+cover+subposet}, one can deduce that the Scarf complex embeds into the facet complex.
\end{remark}

\subsection{Properties of the facet complex} \label{subsec:properties+facet+complex}
\cite[Conjecture 4.7]{DevelinYu:2007} is a list of desirable properties for tropical face posets.
In particular, the third item states the homology of the face poset should be that of a sphere.
The following theorem shows that the facet complex fulfils this property.

\added{
For its proof, we use a version of the famous Nerve theorem.
We use a version close to the original one by Borsuk~\cite{Borsuk:1948} and refer to a more general overview in recent work~\cite{BauerKerberRollRolle:2022}. 

Let $T$ be a finite-dimensional metric space and let $U \subset 2^T$ be a finite set of compact subsets of $T$.
The \emph{nerve complex} of $U$ is the simplicial complex with vertex set $U$ such that a subset $S \subseteq U$ forms a simplex exactly if $\bigcap_{s \in S} s \neq \emptyset$. 

\begin{lemma}[Nerve Theorem] \label{lem:nerve-theorem}
 If all intersections $\bigcap_{s \in S} s$ for $S \subseteq U$ are contractible, then the nerve complex is homotopy equivalent with the whole union~$\bigcup_{s \in U} s$. 
\end{lemma}

We use this to prove the following.
}
\begin{theorem} \label{thm:facet+complex+sphere}
  The facet complex of a monomial tropical polyhedron $\monomial{V}$ in $\TTmax^d$ is homotopy equivalent to a $(d-1)$-sphere.
\end{theorem}
\begin{proof}

  \added{
  For each generator $v \in V$, we define 
\[
C_v = \SetOf{p \in \closedmonomial{V}}{\exists a \in F \ , \ i \in [d] \text{ such that } v \leq p \leq a \ , \ v_i = a_i = p_i} \enspace .
\]
This is equivalent to $v,p$ being incident with $a$ and each other.
Furthermore, for $i \in [n]$, we define
\[
D_i = \SetOf{p \in \closedmonomial{V}}{p_i = +\infty} \enspace .
\]

\smallskip

Firstly, we prove that $\SetOf{C_v}{v \in V} \cup \SetOf{D_i}{i \in [d]}$ forms a set cover of the boundary of $\closedmonomial{V}$.
The points with an $+\infty$ entry are covered by the sets $D_i$. 
For the other points, recall that the finite boundary points are also the finite boundary points of the dual monomial tropical polyhedron. 
The characterization of vertices from Corollary~\ref{cor:vertex+containment} shows that a set $C_v$ is the intersection of the finite part of the boundary of $\complementarymonomial{V}$ with the $\min$-tropical halfspace with apex $v$. 
Now, the representation of the dual monomial tropical polyhedron in Theorem~\ref{thm:complementary-cones} implies that each finite point in the boundary of $\monomial{V}$ is covered by some $C_v$. 

\smallskip

To see the compactness of the sets, we apply an order preserving homeomorphism. 
We can think of the closed hypercube $[-1,1]^d$ instead of $\TTclosed^d$ via the order preserving homeomorphism
\begin{equation} \label{eq:hypercube+homeo}
  \begin{aligned}
    \left[-1,1\right]\ &\to\ \TTclosed \\
    x\ &\mapsto\ \tan\left(\frac{\pi x}{2}\right)
    \enspace ,
  \end{aligned}
\end{equation}
which extends componentwise.
Under this map, the sets $C_v$ and $D_i$ are mapped to compact sets.

\smallskip

  }
  For each set $S \subseteq \overline{V}$, we consider 
\[
A_S = \left(\bigcap_{v \in S \cap V} C_v \right) \cap \left(\bigcap_{e^{(i)} \in S} D_i \right) \subseteq \TTclosed^d \enspace .
\]

The final observation for applying the Nerve Theorem is that $A_S$ is contractible.
Each $C_v$ is $\min$-tropically convex: let $r = \min(p, \lambda+q)$ where $p,q \in C_v$ and $\lambda \geq 0$.
Then there exists $a \in F$ such that $v \leq r \leq p \leq a$ with equality in some coordinate, and so $r$ is also in $C_v$.
This implies any intersection $\bigcap_{v \in S} C_v$ is $\min$-tropically convex and therefore contractible.
Each $\bigcap_{i \in I} D_i$ is homeomorphic to a closed ball and therefore contractible.
Finally, we note that the intersection of any $\min$-tropically convex space $C$ with $D_i$ is also $\min$-tropically convex.
For any $p,q \in C$ with $p_i = q_i = \infty$ and $\lambda \geq 0$, the element $r = \min(p, \lambda + q)$ is in $C \cap D_i$ and therefore is $\min$-tropically convex.
Iterating this gives the claim that each $A_S$ is contractible or empty.

\added{Hence, we can apply Lemma~\ref{lem:nerve-theorem} to deduce that the boundary of $\monomial{V}$ is homotopy equivalent to the nerve complex of its cover $\SetOf{C_v}{v \in V} \cup \SetOf{D_i}{i \in [d]}$. }

\bigskip

Next, we show that the nerve complex equals the facet complex. 
We claim $S$ is a face of the facet complex if and only if $A_S$ is non-empty.

Pick some facet-apex $a \in F$ and consider the corresponding closed set $S \subseteq \overline{V}$ of vertices and rays incident with $a$.
By the second condition of Proposition \ref{prop:facetapex+characterisation}, we have $A_S = \{a\}$.
\added{To see this, let $I \subseteq [d]$ such that $a_i \neq \infty$.
For each $i \in I$, choose an element $w^{(i)} \in V$ with $w^{(i)}_i = a_i$ as in Proposition \ref{prop:facetapex+characterisation}.
  Then $w^{(i)} \in S$ for all $i \in I$ and the intersection $\left(\bigcap_{i \in I} C_{w_i}\right) \cap \left(\bigcap_{j \notin I} D_j\right)$ is just $\{a\}$. }
For the far apex $b^{\infty}$, the corresponding closed set is $S = \{e^{(1)},\dots,e^{(d)}\}$ with $A_S = \{\bm{\infty}\}$.

Conversely, consider some $S \subseteq \overline{V}$ such that $A_S \neq \emptyset$.
For each $p \in A_S$, we have $p \geq v$ for all $v \in S \cap V$ and $p_i = \infty$ for each $e^{(i)} \in S$.
If $p \in C_v$ for some $v\in V$, there exists some facet-apex $a \in F$ incident with $p$, and therefore incident with all elements of $S$.
If $p \notin \closedcomplementarymonomial{V}$ then $p$ is in the boundary of $\TTclosed^d$.
Otherwise $S \subseteq \maxunit$, the closed set of rays corresponding to the far-apex $b^{\infty}$.
This shows the maximal faces of the facet complex are those closed sets incident to a single apex.

\bigskip

\added{Finally, we show that the boundary of $\closedmonomial{V}$ is homotopy equivalent with a sphere.
Again, consider the image of $\TTclosed$ under the homeomorphism described in \eqref{eq:hypercube+homeo} (to avoid an extra treatment of $+\infty$). 
Let $\Omega$ be a point which is coordinate-wise bigger than all points in $V$, and $\varepsilon > 0$ sufficiently small. 
We get a homotopy equivalence from the boundary of $\closedmonomial{V}$ to a Euclidean ball around $\omega = \Omega + \varepsilon \1$ with radius $\varepsilon$ by the retraction along the lines emerging from $\omega$.
  This concludes the proof. }
\end{proof}

We claim the facet complex of $\monomial{V}$ encodes homological data of $I_V$.
To do so, we label each vertex of the facet complex by the corresponding generator or tropical unit vector, and label faces by the maximum of its contained vertices.
To get data of $I_V$, we restrict to the \emph{bounded complex} $\cF_{fin}(V)$.
This subcomplex of the facet complex consists of those faces which do not contain a tropical unit vector.
It is the analogue of the bounded complex of bounded faces of an unbounded polyhedron, see e.g.~\cite{JoswigKaibelPfetschZiegler:2001}.
Note that the resulting complex is labelled by finite vectors.

The finite generators $V$ of $\monomial{V}$ are the atoms of the affine part of the vertex-facet lattice, hence they form a crosscut.

\begin{lemma} \label{lem:crosscut+facet+complex}
  The crosscut complex $\Delta(V)$ of the affine part of the vertex-facet lattice is the bounded complex $\cF_{fin}(V)$. 
\end{lemma}
\begin{proof}
  A subset of $V$ is spanning if and only if its componentwise maximum lies in the interior of $\monomial{V}$.
  By definition of the facet-apices, a subset is not spanning exactly if all its points are incident with a facet apex.
  This implies the claim.
\end{proof}

\begin{theorem} \label{thm:facet+complex+Betti}
The finite facet complex $\cF_{fin}(V)$ of $\monomial{V}$ encodes the Betti numbers of $I_V$.
\end{theorem}
\begin{proof}
  By Proposition~\ref{prop:syzygy+Betti+poset}, the syzygy poset equals the Betti poset of the LCM-lattice.
  Combining Lemma~\ref{lem:syzygy+points+CP+points} and Proposition~\ref{prop:CP+subposet+VIF}, we see that all syzygy points are in the image of the vertex-facet lattice in the LCM-lattice.
  Hence, the LCM-lattice and the bounded part of the vertex-facet lattice have the same lattice homology. 
  As a lattice has the same homology as its crosscut complex by Theorem~\ref{thm:crosscut+homology}, we can deduce from Lemma~\ref{lem:crosscut+facet+complex} that the finite facet complex $\cF_{fin}(V)$ has the same homology as the ideal $I_V$.
\end{proof}

\section{Representation of tropical polyhedra via monomial tropical polyhedra} \label{sec:representation+by+mono+trop+poly}

The following construction demonstrates that the study of monomial tropical polyhedra lays the foundations for face structures of more general tropical polyhedra.

\subsection{$i$th monomial tropical cones} 

In~\cite{JoswigLoho:2017}, Joswig and the first author introduced \emph{monomial tropical cones}.
Let $\homogmaxunit = \{\widehat{e}^{(0)},\widehat{e}^{(1)},\dots,\widehat{e}^{(d)}\}$ where
\[
\hat{e}^{(i)}_{k} \ = \
\begin{cases}
  0  & \mbox{ if }  i = k \\
  -\infty & \mbox{ otherwise }
\end{cases}
\qquad \text{for } 0 \leq i,k \leq d \enspace .
\]
be the set of tropical unit vectors in $\TTmax^{d+1}$.
\begin{definition} \label{def:monomial+tropical+cone}
The \emph{$i$th monomial tropical cone} of a finite set $U \subset \TTmax^{d+1}$ is the tropical cone
\begin{align*}
\monocone{i}{U} = \tcone{U \cup \homogmaxunit \setminus \widehat{e}^{(i)}} \enspace .
\end{align*}
\end{definition}
We remark that~\cite{JoswigLoho:2017} defined monomial tropical cones for generating sets $U$ where $u_i \neq -\infty$ for all $u \in U$, however Definition \ref{def:monomial+tropical+cone} does not require this assumption.
If $U$ does satisfy this assumption, we get their original characterisation of monomial tropical cones:
\begin{align*}
\monocone{i}{U} &= \bigcup_{u \in U} \SetOf{x \in \TTmax^{d+1}}{x_i - u_i \leq x_k - u_k} \mbox{ for all } k \in [d]_0 \mbox{ with } u_k \neq -\infty \enspace .
\end{align*}
We also remark that one can relax the finiteness condition on $U$ for the results in this section, but $\monocone{i}{U}$ may not be a tropical cone, rather a tropical conic set.
Monomial tropical cones already appear in~\cite{AllamigeonGaubertKatz:2011} under the name \emph{$i$th polar cones} as building blocks for a canonical exterior description of tropical cones.

Given $V \in \TTmax^d$, observe that $\monocone{0}{\widehat{V}}$, where $\widehat{V} = \SetOf{(0,v)}{v \in V}$, is the homogenisation of the monomial tropical polyhedron $\monomial{V}$.
This observation offers two directions for generalisation: we can consider the dehomogenisation of $i$th monomial cones for $i \neq 0$, and we can consider those with generators whose first coordinate is $-\infty$.
This leads to a more general definition of monomial tropical polyhedron.
\begin{definition} \label{def:i+monomial+polyhedron}
Fix $i \in [d]_0$ and let $P = \tconv{V} \oplus \tcone{W}$ for $V,W \subset \TTmax^d$.
The \emph{$i$th monomial tropical polyhedron} induced by $P$ is the tropical polyhedron
\begin{equation}
\begin{split}
\imonomial{i}{P} &= \tconv{V \cup e^{(0)}} \oplus \tcone{W \cup (\maxunit \setminus e^{(i)})} \quad \text{ for } i \neq 0 \enspace , \\
\imonomial{0}{P} &= \tconv{V} \oplus \tcone{W \cup \maxunit} \enspace ,
\end{split}
\end{equation}
where $e^{(0)} = (-\infty,\dots,-\infty)$.
\end{definition}

The $i$th monomial tropical polyhedron $\imonomial{i}{P}$ is precisely the dehomogenisation of the $i$th monomial cone $\monocone{i}{\widehat{V}\cup\widehat{W}}$ as defined in Section \ref{sec:tropical+hypercube}.
Note that Definition \ref{def:i+monomial+polyhedron} is not symmetric due to the fact that we have dehomogenised the monomial tropical cone with respect to $x_0$.

One can see that Definition~\ref{def:monomial+tropical+polyhedron} is a special case of Definition~\ref{def:i+monomial+polyhedron} by setting $i = 0$ and $W = \emptyset$. 
Note that outside of this section, the simplified definition suffices for our purposes.

\subsection{Intersection of monomial tropical cones}

The following Proposition implies that tropical convexity is encoded in the interplay of these $i$th monomial tropical polyhedra.

\begin{proposition} \label{prop:monomial+cone+decomp}
Let $P = \tconv{V} \oplus \tcone{W}$ be a tropical polyhedron in $\TTmax^d$.
$P$ is equal to the intersection
  \begin{equation} \label{eq:representation+intersection+mono+poly}
  \bigcap_{i \in [d]_0} \imonomial{i}{P} \enspace .
  \end{equation}
\end{proposition}

\begin{remark}
  By taking the unique minimal exterior description of each $i$th monomial tropical polyhedron in~\eqref{eq:representation+intersection+mono+poly}, we obtain a canonical exterior description of an arbitary tropical polyhedron.
  This representation already occurs in the proof of~\cite[Prop.~2]{AllamigeonGaubertKatz:2011}.

    Addressing again one of our main motivations for this work, this representation leads to an extension of our results to arbitrary tropical polyhedra.
    To capture all the combinatorial data, one can define a \emph{face stack}, which contains the information of the $i$th monomial tropical polyhedra for all $i \in [d]_0$.
    The properties of such a face stack are subject to further work. 
\end{remark}

We shall show this by proving its analogous homogeneous statement, for which we require some additional machinery.
Given a point $u \in \TTmax^{d+1} \setminus \{(-\infty,\dots,-\infty)\}$, we define its \emph{$i$th sector} as the set of points
\begin{align*}
\widehat{S}_i(u) = \bigcap_{k \in [d]_0} \SetOf{z \in \TTmax^{d+1}}{z_i + u_k \leq z_k + u_i} \enspace .
\end{align*}
Note that when restricted to $\RR^{d+1}$, this definition aligns with the usual definition \cite{JoswigLoho:2016}, but this change in formulation allows us to account for points with infinite coordinates.
In particular, if $u_i = -\infty$ then we have $\widehat{S}_i(u) = \SetOf{z \in \TTmax^{d+1}}{z_i = -\infty}$.
We remark that the usual definition of a sector breaks down for $u = (-\infty,\dots,-\infty)$, however this will not be an issue as $(-\infty,\dots,-\infty)$ is always contained in a tropical cone and never a minimal generator.

With this definition, the Tropical Farkas Lemma with infinity \cite[Lemma 28]{JoswigLoho:2016} extends to $\TTmax^{d+1}$.

\begin{lemma} \label{lem:cone+sector+containment}
Let $U \subset \TTmax^{d+1}$.
A point $z \in \TTmax^{d+1}$ is contained in $\tcone{U}$ if and only if for every $i \in [d]_0$, there exists a minimal generator $u$ such that $z \in \widehat{S}_i(u)$.
\end{lemma}
\begin{proof}
Fix $I \subseteq [d]_0$ and define
\begin{align*}
\RR^{|I|} &= \SetOf{z \in \TTmax^{d+1}}{z_i \in \RR \ \forall i \in I \ , \ z_k = -\infty \ \forall k \notin I} \\
\TTmax^{|I|} &= \SetOf{z \in \TTmax^{d+1}}{z_k = -\infty \ \forall k \notin I} \enspace .
\end{align*}
Furthermore, we define $U^I = \SetOf{u \in U}{u_k = -\infty \ \forall k \notin I} \subset \TTmax^{|I|}$.
By \cite[Lemma 28]{JoswigLoho:2016}, a point $z \in \RR^{|I|}$ is contained in $\tcone{U^I}$ if and only if for each $i \in I$ there exists $u \in U^I$ such that $\widehat{S}_i(u)$.
Furthermore $z \in \widehat{S}_k(u)$ for all $k \notin I$, extending the result to each $i \in [d]_0$.
It remains to show $\tcone{U^I} = \tcone{U} \cap \RR^{|I|}$.

One containment is straightforward; for the other, consider $z \in \tcone{U} \cap \RR^{|I|}$.
There exists a representation $z = \bigoplus_{u \in U} \lambda_u \odot u$, in particular $z_k = \max\{\lambda_u + u_k\} = -\infty$ for all $k \notin I$.
If $u_k \neq -\infty$, we must have $\lambda_u = -\infty$ and so we can equivalently simply remove $v$ from the representation.
Repeating this, we get a representation of $z$ using just elements from $U_I$.
\end{proof}

\begin{corollary} \label{cor:monomial+cone+sector}
Let $U \subset \TTmax^{d+1}$.
The $i$th monomial tropical cone generated by $U$ is the union of the $i$th sectors of its generators, i.e.,
\[
\monocone{i}{U} = \bigcup_{u \in U} \widehat{S}_i(u)
\]
\end{corollary}
\begin{proof}
By definition, the $k$th sector of the $k$th tropical unit vector $\widehat{S}_k(\widehat{e}^{(k)})$ is the whole of $\TTmax^{d+1}$.
Therefore Lemma \ref{lem:cone+sector+containment} reduces to $z$ is contained in $\monocone{i}{U}$ if and only if there exists $u \in U$ such that $z \in \widehat{S}_i(u)$, giving the required equality.
\end{proof}

\begin{proof} (Proof of Proposition \ref{prop:monomial+cone+decomp})
Combining Lemma \ref{lem:cone+sector+containment} and Corollary \ref{cor:monomial+cone+sector} gives the projective version of this statement: for a finite generating set $U \subset \TTmax^{d+1}$ we have
\begin{equation} \label{eq:monomial+cone+intersection}
\tcone{U} = \bigcap_{i \in [d]_0} \monocone{i}{U} \enspace .
\end{equation}
By considering the tropical cone generated by $U = \widehat{V}\cup\widehat{W}$ and dehomogenising, the affine version follows immediately.
\end{proof}

\begin{remark}
Equations \eqref{eq:aff+point+sector} and \eqref{eq:aff+ray+sector} are the dehomogenisation of the projective sectors $\widehat{S}_i((0,v))$ and $\widehat{S}_i((-\infty,w))$ respectively.
This immediately gives affine versions of the statements Lemma \ref{lem:cone+sector+containment} and Corollary \ref{cor:monomial+cone+sector}.
\end{remark}

\section{Conclusion}

\added{We finish with some thoughts on the classification of the combinatorial types of monomial tropical polyhedra and potential applications.}

\subsection{Combinatorial types of monomial tropical polyhedra} \label{sec:combinatorial+types}

\added{
  As discussed in Section~\ref{sec:vertex+facet+lattice}, the vertex-facet incidence graph (Def.~\ref{def:vertex-facet-incidence}) captures the combinatorial type of a monomial tropical polyhedron, in an analogous way to how one captures the combinatorial type of a classical polyhedron.
  However, while the latter depends subtly on the values of the subdeterminants of the generator matrix, the situation is significantly simpler for monomial tropical polyhedra.  
This allows us to define `abstract' monomial tropical polyhedra that are encoded purely combinatorially in the following way.
}
  
The \emph{order pattern} of a matrix $V \in \TTmax^{d \times n}$ is the $d$-tuple of total preorders on $n$ elements represented as a matrix of indeterminates $X = (x_{ij})_{(i,j) \in [d] \times [n]}$ with
\[
x_{ij_1} \leq x_{ij_2} \qquad \Leftrightarrow \qquad v_{ij_1} \leq v_{ij_2} \enspace .
\]
\added{We are mainly interested in those order patterns where the columns actually represent non-redundant generators of a monomial tropical polyhedron. } 
The columns of an order pattern form a \emph{valid generator pattern} if the columns of $X$ form an antichain in the weak partial order defined as the Cartesian product of the preorders in the rows.

\added{
\begin{example}
  The matrix $V =
  \begin{pmatrix}
    2 & 4 & 8 & 4 \\
    6 & -\infty & 2 & 4 \\
    5 & 5 & -1 & 5 \\
  \end{pmatrix}
  $
  gives rise to the three total preorders
  \begin{align*}
    x_{11} < x_{12} = x_{14} < x_{13} \\
    x_{22} < x_{23} < x_{24} < x_{21} \\
    x_{33} < x_{31} = x_{32} = x_{34} \\
  \end{align*}
  This gives rise to a weak partial order on $\SetOf{(x_{1i},x_{2j},x_{3k})}{i,j,k \in [3]}$.
  Here, the columns are pairwise incomparable except for 
  \begin{align*}
    \begin{pmatrix} x_{12} \\ x_{22} \\ x_{32} \end{pmatrix} \leq \begin{pmatrix} x_{14} \\ x_{24} \\ x_{34} \end{pmatrix} \enspace .
  \end{align*}
  Hence, columns $1,3,4$ or $1,2,3$ would form an antichain but all three together, they do not form a valid generator pattern. 
  This is reflected when considering the monomial tropical polyhedron $\monomial{V}$: the point $(4,4,5)$ is not a minimal generator as it dominates $(4, -\infty, 5)$.
\end{example}
}

\added{
Proposition~\ref{prop:facetapex+characterisation} gives criteria for a point to be a facet-apex that depends purely on the order pattern of the generators.
This motivates the following combinatorial abstraction of a facet-apex.
The \emph{pattern type} of a facet-apex $a \in \TTclosed^d$ is the bipartite graph on $[d] \sqcup [n]$ with an edge $(i,j) \in [d] \times [n]$ if and only if $a_i = v_{ij}$ and $v_{kj} < a_k$ for all $k \neq i$. 
Isolated nodes on the side $[d]$ in $P$ correspond to components of $a$ which are $+\infty$. 

Hence, recalling again the definition of vertex-facet incidence graph (Def.~\ref{def:vertex-facet-incidence}), we obtain the following consequence of Proposition~\ref{prop:facetapex+characterisation}.

\begin{corollary}
 The vertex-facet incidence graph of a monomial tropical polyhedron only depends on the order pattern of the generator matrix. 
\end{corollary}

We also note that deformations become very natural when working purely with order patterns.
Explicitly, a deformation in the sense of Definition~\ref{def:deformation+monomial+polyhedron} becomes a \emph{refinement} of the order pattern in that more elements per row are \emph{strictly ordered}.

Recall that the \emph{braid fan} $\cB_d$ in $\RR^d$ is the complete polyhedral fan which is cut out by the hyperplanes $x_p = x_q$ for $p \neq q \in [d]$.
For an introduction on the combinatorics of the braid fan, see~\cite{PostnikovReinerWilliams:2008}.
The stratification of the space of real $(d \times n)$-matrices by their order pattern is the product of the braid fans $\cB_d \times \dots \times \cB_d = \cB_d^n$.
Note that $\TTmax^{d \times n}$ can be stratified by the same set of hyperplanes, resulting in the product of braid fans plus some extra stratification at the boundary.
This means one can consider the space of monomial tropical polyhedra as a subfan of $\cB_d^n$ whose associated order patterns form a valid generator pattern.
Furthermore, the face structure of this fan indicates how one can deform from one monomial tropical polyhedron to another.

}

\subsection{Further Questions}

We conclude with several problems which are motivated from tropical convexity and commutative algebra. 

\begin{question}
Can one characterise the face posets of ordinary polyhedra arising as some face poset of a (generic) monomial tropical polyhedron?
\end{question}
This question goes back to~\cite{BayerPeevaSturmfels:1998} for generic monomial ideals.
In the terminology of orthogonal surfaces, some necessary conditions have been established in~\cite{Kappes:2006,FelsnerKappes:2008} and also the condition on the facet-ridge graph in~\cite{DaechertKlamrothLacourVanderpooten:2017} could be applied. 
Our notion of \emph{pattern type} from Section~\ref{sec:combinatorial+types} could be used to enumerate the finite number of occurring vertex-facet lattices in fixed dimension.
While previously the combinatorial types of tropical polytopes were considered via the secondary fan of products of two simplices through the connection discussed in Remark~\ref{rem:connection+products+simplices}, we propose to restrict to the more tractable fan $\cB_n^d$, the $n$-fold product of the $d$-dimensional braid fan.

Focusing directly on monomial tropical polyhedra without their connection to classical polyhedra leads to the following. 
\begin{question}
  Which atomic and coatomic lattices arise as vertex-facet lattices of monomial tropical polyhedra?
\end{question}
A similar motivation lies at the heart of which atomic lattices arise as the LCM-lattice of a monomial ideal studied in~\cite{Mapes:2013,HeWang:2018,IchimKatthanMoyanoFernandez:2017}. 

\smallskip

Using the combinatorial framework of covector graphs, it is tempting to generalise the face poset constructions to tropical oriented matroids, see~\cite{ArdilaDevelin:2009} for an introduction (where `type' is used instead of covector graph) and~\cite{Loho:2020} for an application to tropical linear programming. 
The pseudovertex poset is defined in terms of the partial ordering on $\TTclosed^d$.
\begin{question}
  Can one derive this partial ordering directly from the graph structure of the covector graphs?
\end{question}
  This would allow one to extend the study of the face posets discussed in Section~\ref{sec:face+posets} and~\ref{sec:monomial+ideals} to general tropical oriented matroids that may not be realisable.

  \smallskip

We finish by returning to monomial ideals and their resolutions.
We saw the Betti poset contains all the essential homological data, but is hard to compute.
However, it is contained in far more computationally amenable posets, in particular the max-min poset.

\begin{question}
Can one derive a resolution of a monomial ideal $I$ from the max-min poset of the monomial tropical polyhedron $\monomial{V_I}$. 
\end{question}

\section{Acknowledgements}
We are grateful to Stefan Felsner and Ezra Miller for enlightening discussions and bringing relevant previous work to our attention.
We are also thankful to Michael Joswig, Kathrin Klamroth, Sara Lamboglia and Lewis Mead for insightful conversations and comments throughout the project.
Finally, we thank the reviewers for helpful suggestions. 
\bibliographystyle{amsplain}
\bibliography{main}

\providecommand{\bysame}{\leavevmode\hbox to3em{\hrulefill}\thinspace}
\providecommand{\MR}{\relax\ifhmode\unskip\space\fi MR }
\providecommand{\MRhref}[2]{%
  \href{http://www.ams.org/mathscinet-getitem?mr=#1}{#2}
}
\providecommand{\href}[2]{#2}
\begin{thebibliography}{10}

\bibitem{AllamigeonBenchimolGaubertJoswig:2015}
Xavier Allamigeon, Pascal Benchimol, St{\'e}phane Gaubert, and Michael Joswig,
  \emph{Tropicalizing the simplex algorithm}, SIAM J. Discrete Math.
  \textbf{29} (2015), no.~2, 751--795. \MR{3336300}

\bibitem{ABGJ:2018}
Xavier {Allamigeon}, Pascal {Benchimol}, St\'ephane {Gaubert}, and Michael
  {Joswig}, \emph{{Log-barrier interior point methods are not strongly
  polynomial.}}, {SIAM J. Appl. Algebra Geom.} \textbf{2} (2018), no.~1,
  140--178 (English).

\bibitem{AllamigeonGaubertGoubault:2013}
Xavier Allamigeon, St\'{e}phane Gaubert, and \'{E}ric Goubault, \emph{Computing
  the vertices of tropical polyhedra using directed hypergraphs}, Discrete
  Comput. Geom. \textbf{49} (2013), no.~2, 247--279. \MR{3017909}

\bibitem{AllamigeonGaubertKatz:2011}
Xavier Allamigeon, St\'ephane Gaubert, and Ricardo~D. Katz, \emph{Tropical
  polar cones, hypergraph transversals, and mean payoff games}, Linear Algebra
  Appl. \textbf{435} (2011), no.~7, 1549--1574. \MR{2810655}

\bibitem{AllamigeonKatz:2013}
Xavier Allamigeon and Ricardo~D. Katz, \emph{Minimal external representations
  of tropical polyhedra}, J. Combin. Theory Ser. A \textbf{120} (2013), no.~4,
  907--940. \MR{3022620}

\bibitem{AllamigeonKatz:2017}
\bysame, \emph{Tropicalization of facets of polytopes}, Linear Algebra Appl.
  \textbf{523} (2017), 79--101. \MR{3624667}

\bibitem{ArdilaDevelin:2009}
Federico Ardila and Mike Develin, \emph{Tropical hyperplane arrangements and
  oriented matroids}, Math. Z. \textbf{262} (2009), no.~4, 795--816.
  \MR{2511751 (2010i:52031)}

\bibitem{BauerKerberRollRolle:2022}
Ulrich Bauer, Michael Kerber, Fabian Roll, and Alexander Rolle, \emph{A unified
  view on the functorial nerve theorem and its variations}, 2022.

\bibitem{BayerPeevaSturmfels:1998}
Dave Bayer, Irena Peeva, and Bernd Sturmfels, \emph{Monomial resolutions},
  Math. Res. Lett. \textbf{5} (1998), no.~1-2, 31--46. \MR{1618363}

\bibitem{Birkhoff:1967}
Garrett Birkhoff, \emph{Lattice theory}, Third edition. American Mathematical
  Society Colloquium Publications, Vol. XXV, American Mathematical Society,
  Providence, R.I., 1967. \MR{0227053}

\bibitem{Bjoerner:1995}
A.~Bj\"{o}rner, \emph{Topological methods}, Handbook of combinatorics, {V}ol.
  1, 2, Elsevier Sci. B. V., Amsterdam, 1995, pp.~1819--1872. \MR{1373690}

\bibitem{BlockYu:2006}
Florian Block and Josephine Yu, \emph{Tropical convexity via cellular
  resolutions}, J. Algebraic Combin. \textbf{24} (2006), no.~1, 103--114.
  \MR{2245783}

\bibitem{Borsuk:1948}
Karol Borsuk, \emph{On the imbedding of systems of compacta in simplicial
  complexes}, Fundamenta Mathematicae \textbf{35} (1948), no.~1, 217--234
  (eng).

\bibitem{Chen:2019}
Ri-Xiang Chen, \emph{Lcm-lattice, taylor bases and minimal free resolutions of
  a monomial ideal}, 2019,
  \href{https://arxiv.org/abs/1901.05865}{arXiv:1901.05865}.

\bibitem{ClarkMapes:2014a}
Timothy Clark and Sonja Mapes, \emph{Rigid monomial ideals}, J. Commut. Algebra
  \textbf{6} (2014), no.~1, 33--51.

\bibitem{ClarkMapes:2014b}
\bysame, \emph{{The Betti poset in monomial resolutions}}, 2014,
  \href{https://arxiv.org/abs/1407.5702}{arXiv:1407.5702}.

\bibitem{DaechertKlamrothLacourVanderpooten:2017}
Kerstin D\"achert, Kathrin Klamroth, Renaud Lacour, and Daniel Vanderpooten,
  \emph{Efficient computation of the search region in multi-objective
  optimization}, European J. Oper. Res. \textbf{260} (2017), no.~3, 841--855.
  \MR{3626174}

\bibitem{DaveyPriestley:2002}
B.~A. Davey and H.~A. Priestley, \emph{Introduction to lattices and order},
  second ed., Cambridge University Press, New York, 2002. \MR{1902334}

\bibitem{triangulations}
Jes\'{u}s~A. De~Loera, J\"{o}rg Rambau, and Francisco Santos,
  \emph{Triangulations}, Algorithms and Computation in Mathematics, vol.~25,
  Springer-Verlag, Berlin, 2010, Structures for algorithms and applications.
  \MR{2743368}

\bibitem{DevelinSturmfels:2004}
Mike Develin and Bernd Sturmfels, \emph{Tropical convexity}, Doc. Math.
  \textbf{9} (2004), 1--27. \MR{2054977}

\bibitem{DevelinYu:2007}
Mike Develin and Josephine Yu, \emph{Tropical polytopes and cellular
  resolutions}, Experiment. Math. \textbf{16} (2007), no.~3, 277--291.
  \MR{2367318}

\bibitem{DochtermannJoswigSanyal:2012}
Anton Dochtermann, Michael Joswig, and Raman Sanyal, \emph{Tropical types and
  associated cellular resolutions}, J. Algebra \textbf{356} (2012), 304--324.
  \MR{2891135}

\bibitem{EagonMillerOrdog:2019}
John Eagon, Ezra Miller, and Erika Ordog, \emph{Minimal resolutions of monomial
  ideals}, 2019, \href{https://arxiv.org/abs/1906.08837}{arXiv:1906.08837}.

\bibitem{Ehrgott:2005}
Matthias Ehrgott, \emph{Multicriteria optimization}, second ed.,
  Springer-Verlag, Berlin, 2005. \MR{2143243}

\bibitem{FelsnerKappes:2008}
Stefan Felsner and Sarah Kappes, \emph{Orthogonal surfaces and their
  cp-orders}, Order \textbf{25} (2008), no.~1, 19--47.

\bibitem{FinkRincon:2015}
Alex Fink and Felipe Rinc\'{o}n, \emph{Stiefel tropical linear spaces}, J.
  Combin. Theory Ser. A \textbf{135} (2015), 291--331. \MR{3366480}

\bibitem{Folkman:1966}
Jon Folkman, \emph{The homology groups of a lattice}, J. Math. Mech.
  \textbf{15} (1966), 631--636. \MR{0188116}

\bibitem{GasharovPeevaWelker:1999}
Vesselin Gasharov, Irena Peeva, and Volkmar Welker, \emph{The lcm-lattice in
  monomial resolutions}, Math. Res. Lett. \textbf{6} (1999), no.~5-6, 521--532.
  \MR{1739211}

\bibitem{GaubertKatz:2007}
St\'{e}phane Gaubert and Ricardo~D. Katz, \emph{The {M}inkowski theorem for
  max-plus convex sets}, Linear Algebra Appl. \textbf{421} (2007), no.~2-3,
  356--369. \MR{2294348}

\bibitem{GaubertKatz:2011}
St\'ephane Gaubert and Ricardo~D. Katz, \emph{Minimal half-spaces and external
  representation of tropical polyhedra}, J. Algebraic Combin. \textbf{33}
  (2011), no.~3, 325--348. \MR{2772536}

\bibitem{HeWang:2018}
Peng He and Xue-ping Wang, \emph{Finite atomic lattices and their monomial
  ideals}, Rocky Mountain J. Math. \textbf{48} (2018), no.~8, 2503--2542.
  \MR{3894991}

\bibitem{IchimKatthanMoyanoFernandez:2017}
Bogdan Ichim, Lukas Katth\"{a}n, and Julio~Jos\'{e} Moyano-Fern\'{a}ndez,
  \emph{Stanley depth and the lcm-lattice}, J. Combin. Theory Ser. A
  \textbf{150} (2017), 295--322. \MR{3645578}

\bibitem{Jahn:2011}
Johannes {Jahn}, \emph{{Vector optimization. Theory, applications, and
  extensions. 2nd ed.}}, 2nd ed. ed., Berlin: Springer, 2011 (English).

\bibitem{Joswig:2005}
Michael Joswig, \emph{Tropical halfspaces}, Combinatorial and computational
  geometry, Math. Sci. Res. Inst. Publ., vol.~52, Cambridge Univ. Press,
  Cambridge, 2005, pp.~409--431. \MR{2178330}

\bibitem{JoswigKaibelPfetschZiegler:2001}
Michael Joswig, Volker Kaibel, Marc~E. Pfetsch, and G\"{u}nter~M. Ziegler,
  \emph{Vertex-facet incidences of unbounded polyhedra}, Adv. Geom. \textbf{1}
  (2001), no.~1, 23--36. \MR{1823950}

\bibitem{JoswigLoho:2016}
Michael Joswig and Georg Loho, \emph{Weighted digraphs and tropical cones},
  Linear Algebra Appl. \textbf{501} (2016), 304--343. \MR{3485070}

\bibitem{JoswigLoho:2017}
\bysame, \emph{Monomial tropical cones for multicriteria optimization}, SIAM
  Journal on Discrete Mathematics \textbf{34} (2020), no.~2, 1172--1191.

\bibitem{JoswigLohoLorenzSchroeter:2016}
Michael Joswig, Georg Loho, Benjamin Lorenz, and Benjamin Schr{\"{o}}ter,
  \emph{{Linear programs and convex hulls over fields of Puiseux fractions}},
  Lecture Notes in Computer Science (including subseries Lecture Notes in
  Artificial Intelligence and Lecture Notes in Bioinformatics), vol. 9582,
  Springer, 2016, pp.~429--445.

\bibitem{Kappes:2006}
Sarah~Johanna Kappes, \emph{{Orthogonal Surfaces: A Combinatorial Approach}},
  Ph.D. thesis, Technischen Universitat Berlin, 2006.

\bibitem{LemkeHowson:1964}
C.~E. Lemke and J.~T. Howson, Jr., \emph{Equilibrium points of bimatrix games},
  J. Soc. Indust. Appl. Math. \textbf{12} (1964), 413--423. \MR{0173556}

\bibitem{Loho:2020}
Georg Loho, \emph{Abstract tropical linear programming}, {Electron. J. Comb.}
  \textbf{27} (2020), no.~2, research paper p2.51, 33 (English).

\bibitem{LohoVegh:2020}
Georg Loho and L{\'a}szl{\'o}~A. V{\'e}gh, \emph{{Signed Tropical Convexity}},
  11th Innovations in Theoretical Computer Science Conference (ITCS 2020)
  (Thomas Vidick, ed.), Leibniz International Proceedings in Informatics
  (LIPIcs), vol. 151, 2020, pp.~24:1--24:35.

\bibitem{Mapes:2013}
Sonja Mapes, \emph{Finite atomic lattices and resolutions of monomial ideals},
  J. Algebra \textbf{379} (2013), 259--276. \MR{3019256}

\bibitem{Markwig:2010}
T.~Markwig, \emph{A field of generalised {P}uiseux series for tropical
  geometry}, Rend. Semin. Mat. Univ. Politec. Torino \textbf{68} (2010), no.~1,
  79--92. \MR{2759691 (2012e:14126)}

\bibitem{Miller:1998}
Ezra Miller, \emph{Alexander duality for monomial ideals and their
  resolutions}, 1998,
  \href{https://arxiv.org/abs/math/9812095}{arXiv:math/9812095}.

\bibitem{Miller:2002}
\bysame, \emph{Planar graphs as minimal resolutions of trivariate monomial
  ideals}, Doc. Math. \textbf{7} (2002), 43--90. \MR{1911210}

\bibitem{MillerSturmfels:2005}
Ezra Miller and Bernd Sturmfels, \emph{Combinatorial commutative algebra},
  Graduate Texts in Mathematics, vol. 227, Springer-Verlag, New York, 2005.
  \MR{2110098}

\bibitem{Miller:2000}
Ezra {Miller}, Bernd {Sturmfels}, and Kohji {Yanagawa}, \emph{{Generic and
  cogeneric monomial ideals}}, {J. Symb. Comput.} \textbf{29} (2000), no.~4-5,
  691--708 (English).

\bibitem{Noren:2016}
Patrik Nor\'{e}n, \emph{Slicing and dicing polytopes}, 2016,
  \href{https://arxiv.org/abs/1608.05372}{arXiv:1608.05372}.

\bibitem{PostnikovReinerWilliams:2008}
Alex Postnikov, Victor Reiner, and Lauren Williams, \emph{Faces of generalized
  permutohedra}, Doc. Math. \textbf{13} (2008), 207--273. \MR{2520477}

\bibitem{Scarf:1967}
Herbert Scarf, \emph{The approximation of fixed points of a continuous
  mapping}, SIAM J. Appl. Math. \textbf{15} (1967), 1328--1343. \MR{0242483}

\bibitem{Scarf:1973}
\bysame, \emph{The computation of economic equilibria}, Yale University Press,
  New Haven, Conn.-London, 1973, With the collaboration of Terje Hansen, Cowles
  Foundation Monograph, No. 24. \MR{0391909}

\bibitem{Schroder:2016}
Bernd Schr\"{o}der, \emph{Ordered sets}, second ed., Birkh\"{a}user/Springer,
  [place of publication not identified], 2016, An introduction with connections
  from combinatorics to topology. \MR{3469976}

\bibitem{Trotter:1992}
William~T. Trotter, \emph{Combinatorics and partially ordered sets}, Johns
  Hopkins Series in the Mathematical Sciences, Johns Hopkins University Press,
  Baltimore, MD, 1992, Dimension theory. \MR{1169299}

\end{thebibliography}

\end{document}